%% file: arxiv.tex
\documentclass[paper,hidelinks,onefignum,onetabnum]{siamart250211}

\input{shared}

\ifpdf
\hypersetup{
  pdftitle={Learning Memory \& Material Dependent Constitutive Laws},
  pdfauthor={K. Bhattacharya, L. Cao,  G. Stepaniants, A. Stuart, and M. Trautner}
}
\fi

\begin{document}

\maketitle

\begin{abstract}
We propose and study a neural operator framework for learning memory- and material microstructure-dependent constitutive laws for heterogeneous materials. We work in the two-scale setting where homogenization theory provides a systematic approach to deriving macroscale constitutive laws, obviating the need to resolve complex microstructure repeatedly. However, the unit cell problems defining these constitutive models are typically not amenable to explicit evaluation. It is therefore of interest to learn constitutive models from data generated by the unit cell problem. Our proposed framework models homogenized constitutive laws with both memory- and microstructure-dependence through the use of Markovian recurrent and Fourier neural operators. The homogenization problem for Kelvin–Voigt viscoelastic materials is studied to provide firm theoretical foundations for our model. The theoretical properties of the cell problem in this Kelvin–Voigt setting motivate the proposed learning framework; and are also used to prove a universal approximation theorem for the learned macroscale constitutive model. Numerical experiments show that the proposed learning framework accurately learns memory- and microstructure-dependent viscoelastic and elasto-viscoplastic constitutive models, beyond the setting of the theory. Furthermore, we show that the learned constitutive models can be successfully deployed in macroscale simulation of material deformation for different microstructures without retraining.

\end{abstract}



\input{article_text.tex}

\clearpage

\appendix

\input{supplement_text.tex}

\clearpage

\bibliographystyle{siamplain}
\bibliography{references}

\end{document}

%% file: shared.tex
\input{macros}

\newsiamremark{remark}{Remark}
\crefname{remark}{Remark}{Remarks}
\newtheorem{assumptions}{Assumptions}
\crefname{assumptions}{Assumptions}{Assumptions}
\crefname{section}{Section}{Sections}
\crefname{subsection}{Subsection}{Subsections}

\headers{Learning Memory \& Material Dependent Constitutive Laws}{K. Bhattacharya, L. Cao,  G. Stepaniants, A. Stuart \& M. Trautner}

\title{Learning Memory and Material Dependent Constitutive Laws\thanks{Submitted to the editors DATE.
\funding{This work is supported by the ONR MURI on Data-Driven Closure Relations N00014-23-1-2654. Additionally, GS is supported by an NSF Mathematical Sciences Postdoctoral Research Fellowship (MSPRF) under award number 2402074 and KB and AMS by the ONR SciAI Center under grant N00014-23-1-2729. AMS is also supported by a Department of Defense Vannevar Bush Faculty Fellowship.}}}

\author{Kaushik Bhattacharya\thanks{Mechanical and Civil Engineering, California Institute of Technology, Pasadena, CA, USA (\email{bhatta@caltech.edu})}
\and Lianghao Cao\thanks{Computing and Mathematical Sciences, California Institute of Technology, Pasadena, CA, USA (\email{lianghao@caltech.edu}, \email{gstepan@caltech.edu}, \email{astuart@caltech.edu}, \email{trautner@caltech.edu})}
\and George Stepaniants\footnotemark[3]\and Andrew Stuart\footnotemark[3] \and Margaret Trautner\footnotemark[3]}


%% file: macros.tex
\usepackage{amsmath,amsfonts,amssymb,latexsym}
\usepackage{mathtools}
\usepackage{graphicx}
\usepackage{algpseudocode}
\usepackage{placeins}
\usepackage{hyperref}
\usepackage{subcaption}
\usepackage{thmtools}
\usepackage{thm-restate}
\usepackage{dsfont}
\usepackage{placeins}
\usepackage{wrapfig}
\usepackage{enumerate}
\usepackage{float}
\usepackage{algorithm}
\usepackage{wasysym}
\usepackage{hyperref}
\usepackage{cleveref}
\usepackage{braket}
\usepackage{makecell}
\usepackage{mathrsfs}
\usepackage{multirow}
\usepackage[normalem]{ulem}
\usepackage{caption}
\definecolor{celadon}{rgb}{0.17, 0.8, 0.69}
\definecolor{darkred}{rgb}{.7,0,0}
\definecolor{darkpink}{rgb}{0.95, 0.1, 0.8}

\DeclareMathOperator{\inn}{in}
\DeclareMathOperator{\out}{out}
\DeclareMathOperator{\lift}{lift}
\DeclareMathOperator{\proj}{proj}

\DeclareMathOperator{\pc}{pc}

\DeclareMathOperator{\RNO}{\sc{RNO}}
\DeclareMathOperator{\FNM}{\sc{FNM}}
\DeclareMathOperator{\FNMRNO}{\sc{FNM--RNO}}

\DeclareMathOperator*{\esssup}{ess\,sup}
\DeclareMathOperator*{\essinf}{ess\,inf}

\def\p{\partial}

\def\eps{\epsilon}

\def\R{\mathbb{R}}
\def\D{\mathbb{D}}

\def\eps{\epsilon}
\def\e{\mathfrak{e}}

\def\Td{\mathbb{T}^d}

\def\epsbar{\overline{\epsilon}}
\def\sigmabar{\overline{\sigma}}

\def\dx{\;\mathrm{d}x}
\def\dy{\;\mathrm{d}y}
\def\ds{\;\mathrm{d}s}
\def\dt{\;\mathrm{d}t}
\def\dz{\;\mathrm{d}z}
\def\d{\;\mathrm{d}}
\def\cT{\mathcal{T}}

\def\cZ{\mathcal{Z}}
\def\cX{\mathcal{X}}
\def\cY{\mathcal{Y}}
\def\cD{\mathcal{D}}
\newcommand{\sD}{\mathscr{D}}

\def\Zd{\mathbb{Z}^d}
\def\N{\mathbb{N}}
\def\sL{\mathscr{L}}
\def\cK{\mathcal{K}}
\def\C{\mathbb{C}}
\def\sG{\mathscr{G}}
\newcommand{\iunit}{\mathrm{i}}

\def\cC{\mathcal{C}}
\def\cM{\mathcal{M}}

\def\one{\mathds{1}}
\def\L{\mathsf{L}}
\def\T{\mathbb{T}}
\def\Z{\mathbb{Z}}
\def\cX{\mathcal{X}}

\newcommand{\BV}{\mathsf{BV}}

\newcolumntype{C}{ >{\centering\arraybackslash} m{0.1\linewidth} }
\newcolumntype{D}{ >{\centering\arraybackslash} m{0.25\linewidth} }
\newcolumntype{E}{ >{\centering\arraybackslash} m{0.05\linewidth} }
\newcolumntype{F}{ >{\centering\arraybackslash} m{0.27\linewidth} }
\newcolumntype{G}{ >{\centering\arraybackslash} m{0.55\linewidth} }
\newcolumntype{H}{ >{\centering\arraybackslash} m{0.35\linewidth} }
\newcolumntype{I}{ >{\centering\arraybackslash} m{0.4\linewidth} }

\usepackage{tikz}
\usetikzlibrary{positioning}

\DeclarePairedDelimiterX{\iptemp}[2]{\langle}{\rangle}{#1, #2}
\newcommand{\ip}{\iptemp}
\newcommand{\slot}{{\,\cdot\,}}

\usepackage{enumitem}
\setlist[enumerate]{leftmargin=.5in}
\setlist[itemize]{leftmargin=.5in}
\makeatletter
\newcommand{\myitem}[1]{%
\item[#1]\protected@edef\@currentlabel{#1}%
}
\makeatother

%% file: article_text.tex
\section{Introduction}
\label{sec:I}
Physical materials have multiple scales \cite{weinan2011principles}, including the atomic scale governed by quantum interactions between atoms, the microscale encompassing fine microstructures such as grains or defects, and the macroscale which describes a material in bulk and analyzes its observable physical properties. Modeling this cascade of information across scales is a problem of immense interest in the materials science community as it holds the potential to connect structure at the atomic scale and microscale to observable material properties~\cite{phillips2001crystals}. This, in turn, can allow the prediction of macroscale phenomena by appropriately summarizing the smaller scales~\cite{zohdi2008introduction}. Multiscale materials can be simulated by constructing a hierarchy of physical models that capture the relevant phenomena at each scale and allowing these scales to interact pairwise. This simulation method is very accurate when there is clear scale separation but prohibitively expensive as it requires simulating the dynamics of an atomic scale or microscale problem within every representative volume element of the coarser scale~\cite{zohdi2008introduction}.

The theory of homogenization~\cite{abdulle2012heterogeneous,bensoussan2011asymptotic,pavliotis2008multiscale} defines macroscale constitutive laws that are found by averaging over smaller scales and, thus, provides an alternative, cheaper route for simulating the behavior of multiscale materials. In this work, we consider heterogeneous materials in the two-scale settings with a microscale $\varepsilon \ll 1$, which is often determined by the typical grain or defect size. The relationship between average stress and average strain, over a unit cell of scale $\mathcal{O}(\varepsilon)$, defines a constitutive law that governs the material at the macroscale of $\mathcal{O}(1)$. Once this macroscale constitutive law is found, we can avoid solving for the microscale dynamics entirely. 

In many settings, homogenization does not provide usable closed formulae for the relationship between cell-averaged strain and stress. In this context, it is of interest to \emph{learn} the relationship from data. For viscoelastic and inelastic materials, the stress--strain relationship often involves memory\footnote{Often referred to as history or path-dependence in the mechanics literature.} by taking strain history, a temporally varying function, as input to determine stress at a given time. 
Furthermore, it is of interest to learn how the constitutive relationship depends on the properties of the microstructure represented by a spatially varying function. Learning this dependence obviates the need to solve a different homogenization problem for each different material and has the potential to further accelerate computations. This paper aims to develop, analyze, and numerically study neural operator architectures, which take as input temporally and spatially varying functions, suitable for learning memory- and microstructure-dependent macroscale constitutive laws.

In \cref{subsec:CPO} we summarize our contributions and overview the remainder of the paper.
\cref{subsec:LR} contains a literature review, detailing the context for our contributions. In \cref{subsec:N} we define notation used throughout the paper.

\subsection{Contributions and Paper Overview}
\label{subsec:CPO}
This work makes the following novel contributions to data-driven constitutive modeling in the homogenization setting:
\begin{enumerate}[label=(C\arabic*)]
    \item  We propose a recurrent neural operator constitutive model that predicts the stress-strain relationship for a wide array of material microstructures without retraining.
    \item In the context of one-dimensional Kelvin--Voigt linear viscoelasticity, we prove Lipschitz continuity of the map from properties of the microstructure to the homogenized stress; using this we prove a universal approximation property for the proposed recurrent neural operator (\cref{thm:main_ua_theorem}).
    \item We provide insight into advantageous choices of measures from which to draw training data, namely the strain trajectories and material microstructures, when learning a model.
    \item We showcase the empirical success of our proposed constitutive model in
    both linear viscoelastic (where the theory applies directly) and nonlinear elasto-viscoplastic materials (where it does not). We (i) accurately predict stress responses when incorporating both the strain trajectory and properties of the material microstructure as inputs; and (ii) deploy the
    learned model in macroscale simulations for different material microstructures.
\end{enumerate}

In \cref{sec:PACL}, we describe our hypotheses about the exact homogenized constitutive law that
we wish to approximate, and we describe the new class of neural operators, FNM--RNO, that we introduce, study,
analyze and test in the remainder of the paper; this is contribution (C1). In \cref{sec:PB} we introduce the multiscale Kelvin--Voigt model
of one-dimensional viscoelastic materials, concentrating on describing the cell problem.
This is a constitutive model for which we are able to prove a universal
approximation theorem for FNM--RNO, motivating its form. \cref{sec:A} starts by studying the Lipschitz properties of the cell problem with respect to the material microstructure, the first part of contribution (C2). We then use this continuity to prove our universal approximation result (\cref{thm:main_ua_theorem}), showing that the cell problem partial differential equation (PDE) solution operator can be efficiently approximated by an FNM--RNO,
the second part of contribution (C2). \cref{sec:N} addresses contributions
(C3) and (C4) by describing numerical results implementing this FNM--RNO model to determine the model for
the dynamics of viscoelastic and elasto-viscoplastic homogenized materials and then simulate from this model. The supplementary material contains detailed technical discussions on the equivalence of different cell problems, Lipschitz properties of the cell problem PDE, universal approximation result, and training formulations.

\subsection{Literature Review}
\label{subsec:LR}
Our work touches upon several classical and modern topics in the constitutive modeling of materials, including homogenization methods, memory-dependent modeling, machine learning of constitutive laws, and model dependence on material microstructure. We discuss prior research in each of these areas below.

\paragraph{\bf \bf Theoretical and Numerical Homogenization} In its simplest formulation, homogenization theory~\cite{pavliotis2008multiscale} studies elliptic or parabolic PDEs whose coefficients vary periodically on a small lengthscale $\varepsilon \ll 1$. Such PDEs are \textit{multiscale} since their solutions have coarse-grained features as well as fine-grained features of scale $\varepsilon$. Homogenization uses a power series expansion to determine the limit of the PDE solution as the lengthscale $\varepsilon$ is taken to zero. This results in a new averaged or \textit{homogenized} PDE of the same form with an \textit{effective} coefficient function that is now independent of the microscale $\varepsilon$. This effective coefficient is determined by a boundary value problem (BVP) called the \textit{cell problem} which is solved at the microscale level. Homogenization theory can be extended to PDEs with random or even nonperiodic coefficients~\cite{cioranescu1999introduction, kozlov1980averaging}, and care must be taken to establish convergence of the true solution to the homogenized limit~\cite{allaire1992homogenization, cioranescu1999introduction, pavliotis2008multiscale}. Viscoelastic materials are governed by elliptic PDEs, where the coefficient function known as the \textit{elastic modulus} encodes the material microstructure. Multiscale materials can be modeled by elliptic PDEs whose elastic modulus similarly varies on a small length scale $\varepsilon$. In this setting, homogenization derives the effective elastic modulus of the material that is again independent of this microscale. The texts of Milton~\cite{milton_book} and Zohdi \& Wriggers~\cite{zohdi2008introduction} give a comprehensive review of effective material properties that result from homogenization. 

As mentioned above, the effective modulus of a homogenized material is determined by solving a cell boundary value problem at the level of the microscale $\varepsilon$. Predicting the macroscale behavior of a material by numerically solving this cell problem BVP is called \textit{numerical homogenization} and is a core focus of \textit{computational micromechanics}~\cite{zohdi2008introduction}. Cell problems are typically solved with periodic, Dirichlet, or Neumann boundary conditions using spectral~\cite{mishra2016comparative, moulinec1998numerical} or finite element methods~\cite{guedes1990preprocessing, suquet1987elements}. The cell problem resulting from homogenization allows us to compute the effective properties of a multiscale material, but this computation must be repeated for every new microstructure, making it an expensive procedure. We discuss below how machine learning methods allow us to perform homogenization over a range of microstructures simultaneously.

\paragraph{\bf Memory and Internal Variables}
Viscoelastic materials model viscous as well as elastic behavior; hence, their strain and stress dynamics explicitly depend on time. In particular, the application of a sudden strain deformation or stress load at one location is remembered throughout the material at all future times, and this memory is quantified by exponentially decaying \textit{memory kernels} called the creep compliance and relaxation modulus functions~\cite{ferry1980viscoelastic, tschoegl2012phenomenological}. This observation that viscoelastic materials have fading memory was formally shown to hold under very general mathematical assumptions in a series of papers by Coleman \& Noll~\cite{coleman1961foundations, coleman1961recent}.

Fading memory also arises in multiscale viscoelastic materials whose microstructure varies periodically on a cell of size $\varepsilon$. Even though the original multiscale material locally exhibits no memory effects in strain or stress (e.g., Markovian behavior), homogenizing by averaging at the $\varepsilon$ scale and taking $\varepsilon \to 0$ introduces local creep compliance and relaxation modulus functions that dictate long term memory in the strain and stress dynamics at every point of the homogenized material. This remarkable result was first proven by Sanchez--Palencia on the Kelvin--Voigt model using semigroup theory~\cite[Chapter 6]{sanchez1980non}. Further extensions to thermo-viscoelasticity were proven in the seminal paper of Francfort and Suquet~\cite{francfort1986homogenization}. Tartar~\cite{tartar1991memory} showed that the memory kernel (relaxation modulus) relating strain to stress after homogenization is given by a possibly infinite sum of exponentials. Suquet and coauthors~\cite{brenner2013overall, lahellec2024effective, lahellec2024effective2} discuss more about the structure of these memory kernels and their approximation by finite sums of exponentials. In one-dimensional piecewise constant materials, the creep compliance and relaxation modulus memory kernels are exactly given by a finite sum of exponentials, and this has been rediscovered in various classical and modern texts~\cite{gross1968mathematical, bhattacharya2023learning}. Approximating these kernels by a finite sum of exponentials is known as a \textit{Prony series}, and this technique has been very well explored both in theory~\cite{lahellec2024effective, serra2019viscoelastic, tschoegl2012phenomenological} and experiments~\cite{kim2024experimental, kraus2017parameter, nikonov2005determination, shanbhag2023computer}.

Viscoelastic materials whose memory kernels are given by finite sums of exponentials can be transformed into differential equations with an internal state vector whose dimension is equal to the number of exponential terms. These internal state variables integrate in their dynamics all the history of the material, but they do so in a Markovian way, leading to more efficient simulations of material stress-strain dynamics~\cite{bhattacharya2023learning, liu2023learning}. Internal variables also arise in models of plastic~\cite{rice1971inelastic} and viscoplastic~\cite{liu2023learning} materials, and reviews of this subject can be found in~\cite{billington1982physics, horstemeyer2010historical}. Hence, memory in materials is fundamentally linked to internal variable and differential equation representations. A review of the equivalence between such model representations can be found in ~\cite{eggersmann2019model}.

\paragraph{\bf Machine Learning of Constitutive Models}
Two central applications of data-driven methods in materials science are the discovery of unknown constitutive laws and, related to this, the acceleration of composite multiscale material simulations~\cite{liu2021review}. Data-driven learning of constitutive laws is an actively developing field that has incorporated a variety of tools including model-constrained parameter inference~\cite{akerson2024learning}, physics-informed machine learning~\cite{haghighat2023constitutive}, probabilistic machine learning~\cite{fuhg2022physics}, deep learning~\cite{liu2019deep}, and operator learning~\cite{bhattacharya2024learning}. We refer readers to a recent comprehensive review paper~\cite{fuhg2024review} on this topic.

For inelastic materials, a constitutive model must use the strain history to predict the evolution of stress; the relationship between strain and stress is no longer instantaneous. Learning such a mapping between strain and stress time series has been approached with several data-driven architectures. Liu et al.~\cite{liu2022learning} featurize strain and stress time series through principal component analysis and learn a mapping between these feature spaces. This approach has the benefit of being invariant to the level of time discretization of the data, but suffers from a lack of causality in its learned strain-to-stress map. Causality can be enforced through the use of recurrent neural networks (RNNs), and the LSTM~\cite{ghavamian2019accelerating} and  GRU~\cite{mozaffar2019deep} recurrent architectures have been very effective at learning strain-to-stress maps with history dependence.

A natural approach to enforce both causality and independence to time discretization is to model the constitutive law as a differential equation which is forced by the strain trajectory and whose output is the stress. Compared to the non-physical architecture of RNN models, this approach is guided by the internal variable theories of memory-dependent materials discussed above and has motivated a large application of neural ODEs~\cite{chen2018neural, jones2022neural}, also referred to as recurrent neural operators~\cite{liu2023learning}, to constitutive modeling of materials~\cite{bhattacharya2023learning, jones2022neural, karimi2024learning, liu2023learning, zhang2024iterated}.

\paragraph{\bf Microstructure-Dependent Architectures}
Since data-driven constitutive models must be retrained for each material microstructure, there is a need to develop {\it microstructure-dependent} architectures that can predict constitutive laws of new materials without retraining. One approach is to allow data-driven models to depend on summary statistics of a material, such as the volume fraction, elastic modulus of different phases, or mean sizes and distances between fibers and grains embedded in a material. This idea has been used in several important architectures such as the Deep Material Network~\cite{liu2019deep} and material-dependent recurrent networks~\cite{mozaffar2019deep}. Bishra et al.~\cite{bishara2023state} provide a good review of such methods. These approaches assume that a material microstructure can be sufficiently described by predetermined statistics, mostly applicable to $n$-phase media, and hence do not generalize to more complicated spatially varying microstructures.

The fact that a material microstructure must generally be interpreted as a full \textit{functional} input into a data-driven constitutive model has been considered in a few recent papers. In~\cite{bhattacharya2024learning}, the Fourier neural operator (FNO) architecture was trained on elastic multiscale materials to learn a map from their microstructure as a spatially-varying function on the cell domain to the solution of the cell problem. Crucially, the regularity or Lipschitz continuity of the cell problem resulting from homogenization was used to prove that this FNO architecture is a universal approximator: it can predict the homogenized elastic modulus across a range of microstructures with uniformly bounded error. In this paper, we show how cell problems of \textit{viscoelastic} materials also satisfy Lipschitz regularity conditions, allowing us to build differential equation FNO architectures with universal approximation guarantees over a range of microstructures.

Jones et al.~\cite{jones2022neural} was the first work to propose a neural ODE architecture that was microstructure dependent and hence could simulate stress-strain dynamics with internal state variables across a wide array of microstructures. Their approach was to featurize the material microstructure function using a graph convolutional neural network and to augment the initial conditions of the internal state variables with this feature vector. This way of encoding the material microstructure in the initial conditions of the internal state variables is motivated by the improved numerical performance of augmented neural ODEs~\cite{dupont2019augmented}. In contrast, the theory of viscoelastic~\cite{bhattacharya2023learning} and viscoplastic~\cite{liu2023learning} materials instead shows that material dependence must be encoded in the \textit{functional form} of the differential equation driving the internal state variables rather than in their initial conditions. This is the approach we take here, which allows us to accurately predict strain stress dynamics for a variety of microstructures and, crucially, obtain theoretical guarantees for our method.

\subsection{Notation}
\label{subsec:N}
\paragraph{\bf Sets}
 We denote the set of all integers by $\Z$. The set of $d$-dimensional integer vectors is given by $\Zd$. We denote by $\N$  the natural numbers including zero, and by $\N_{>0}$ the natural numbers excluding zero. For $M \in \N_{>0}$ we denote $[M]$ as the set of indices $\{1, \dots, M\}$.

\paragraph{\bf Euclidean Spaces}
Let $m$ be an arbitrary positive integer. Define $\R$ as the real line and $\R_+$ as the positive real line, including zero. The set of complex numbers is denoted by $\C$. Let $\R^m$ denote the $m$-dimensional Euclidean space and $\R^{m \times m}$ be the space of $m \times m$ real matrices. We also use $\R_+$ and $\R_+^m$ for the space of nonnegative real numbers and $m$ dimensional vectors with nonnegative entries respectively. We define $\D^m$ to be the space of $m \times m$ real diagonal matrices and $\D_+^m$ to be the space of real diagonal matrices with nonnegative diagonal entries. 
We write the Euclidean inner product and norm on $\R^d$ as $\langle \cdot, \cdot\rangle_2$ and $\|\cdot\|_2$ respectively. For matrices in $\R^{d \times d}$ we write the Frobenius norm as $\|\cdot\|_F$ and the infinity norm as $\|\cdot\|_\infty$.

\paragraph{\bf Function Spaces}
We define the domain $\Omega \subset \R^d$ to be a bounded open set and $\cT \subset \R_+$ to be a time interval which can be finite or infinite and open or closed.  Define the Hilbert space $L^2(\Omega; \R^m)$ whose inner product and norm we denote by $\langle \cdot, \cdot\rangle_{L^2}$ and $\|\cdot\|_{L^2}$ respectively. We also define the space $L^1(\Omega; \R^m)$ equipped with the norm $\|\cdot\|_{L^1}$. Define the space $L^\infty(\cT;\R^m)$ and denote its norm by $\|\cdot\|_{L^\infty}$. The space 
$W^{k, p}(\cT)=W^{k, p}(\cT;\R)$, or multivariate generalizations taking values in $\R^d$ or $\R^{d \times d}$, denotes the Sobolev space of functions defined on the time interval $\cT$ with weak derivatives up to order $k$ which are all in $L^p$, $1\leq p \leq \infty$. We also introduce the function space $\cZ = L^\infty(\cT; L^2(\Omega; \R^m))$ with the norm $\|u\|_\cZ = \esssup_{t \in \cT}(\|u(\cdot, t)\|_{L^2})$. We denote by $H_0^1(\Omega; \R^m)$ the closure of the space of infinitely differential functions compactly supported in $\Omega$ in the Sobolev space $W^{1,2}(\Omega)$.

When working with function spaces such as $H_0^1(\Omega; \R^m)$ or $L^1(\Omega; \R^m)$, we often write $H_0^1(\Omega)$ or $L^1(\Omega)$ when it is clear from the context that the functions take values in $\R^m$.
We denote by $\mathds{1}$ the constant unit function on $\Omega$. In much of our theory we work with the specific choice $\Omega = [0, 1]$. Additionally, we extend all the preceding definitions of functions defined
on $\Omega$ to functions defined on the $d-$dimensional torus, denoted $\Td$.

We denote the \textit{total variation} of a function $u \in L^1(\T)$ by 
\begin{equation*}\label{eq:tv_norm}
    |u|_{\BV} = \sup\Big\{\sum_{i=0}^{N-1}|u(x_{i+1}) - u(x_i)| \ \Big| \ 0 = x_0 < x_1 < \hdots < x_{N} = 1, \ N \geq 1\Big\}
\end{equation*}
and the set of functions of bounded variation on $\T$ as 
\begin{equation}
    \BV = \{u \in L^1(\T): \; |u|_{\BV} < \infty\}.
\end{equation}

\paragraph{\bf Maps}
Let $A$ (resp. $B$) be a map from input domain $\cX_A$ (resp. $\cX_B$)
into a spaces of functions defined over $\Td$ and taking values in $\R^{d_A}$ (resp. $\R^{d_B}$). The notation $(A, B)$ denotes the map from $\cX_A \times \cX_B$ into 
the function space on domain $\Td$ taking values in $\R^{d_A + d_B}$ such that, for $(x_a, x_b) \in \cX_A \times \cX_B$, $((A,B)(x_a,x_b))_j$ equals $A(x_a)_j$ for $j \in [d_A]$ and equals $B(x_b)_{j+d_A}$ for $j \in [d_B]$.

\paragraph{\bf Trajectories}
For any time-dependent function $g$ we denote by $\{g(t)\}_{t \in \cT}$ the set that includes pointwise
evaluation of $g(t)$ and its time-derivative for all $t \in \cT$. When it is clear in the appropriate context, we write $g$  as shorthand for $\{g(t)\}_{t\in\cT}$. We use $\dot{g}$ to indicate a time derivative of the trajectory $g$. In particular $\dot{\epsbar}$ is a time derivative of $\epsbar$. Note however that, in the context of elasto-viscoplasticity, we use the commonly adopted convention that $\dot{\epsilon}_{p0}$ denotes the rate constant; in particular it is not the derivative of a time-dependent function.

\paragraph{\bf Subsets of Banach Spaces}

Denote by $\cM_{f_{\min}, f_{\max}}^B$  the set of functions $f \in \BV(\Omega)$ satisfying
\begin{equation}
    |f|_{\BV} \leq B, \quad \esssup_{y \in \Omega}f(y) \leq f_{\max}, \quad \essinf_{y \in \Omega}f(y) \geq f_{\min}
\end{equation}
for some $0 < f_{\min} \leq f_{\max} < \infty$ and $B>0.$ And we denote by $\cC_{g_{\max}, \dot{g}_{\max}}$ the set of functions $g\in W^{1, \infty}(\cT)$ satisfying
\begin{equation}
    \esssup_{t \in \cT}|g(t)| \leq g_{\max}, \quad \esssup_{t \in \cT}|\dot{g}(t)| \leq \dot{g}_{\max}
\end{equation}
for some constants $0 < g_{\max}, \dot{g}_{\max} < \infty$.

\section{Proposed Approximate Constitutive Law}\label{sec:PACL}

In \cref{subsec:HCL}, we introduce the concept behind the derivation of homogenized constitutive laws, in a general setting. In \cref{subsec:neuralop0}, we propose a form of neural operator architecture to learn the homogenized constitutive law. The general setting encompasses the particular case of Kelvin--Voigt linear viscoelasticity which we use to motivate and to analyze the proposed architecture for the constitutive model, but it is not restricted to this case.

\subsection{Homogenized Constitutive Law}\label{subsec:HCL}

Let $\cD \subset \mathbb{R}^d$ denote a bounded open set and consider the following multiscale material model:
\begin{subequations}\label{PDE-ge}
\begin{align}
    \rho \p_t^2 u_{\varepsilon}(x,t) & = \nabla_x \cdot \sigma_{\varepsilon}(x,t) + f(x,t) && x \in \cD, t \in \cT \\
    \sigma_{\varepsilon}(x,t)&= \Psi(\{\nabla_x u_{\varepsilon}(x,s)\}_{s \in \cT}; M, x, \varepsilon)(t) && x \in \cD, t\in \cT\\
    u_{\varepsilon}(x,0) &= \p_tu_{\varepsilon}(x,0)=0 && x \in \cD\\
    u_{\varepsilon}(x,t) &= 0 && x \in \partial \cD, t \in \cT .
\end{align}
\end{subequations}
Here $u_{\varepsilon}$ denotes displacement and $\sigma_{\varepsilon}$ stress, and $\varepsilon$ is a small parameter defining the spatial microscale; the material properties encapsulated in the spatial fields $M$ vary on this scale. Function $\Psi$ is a multiscale constitutive model taking as input the history of the strain $\nabla_x u_{\varepsilon}$\footnote{Because we primarily work in one spatial dimension in this paper we have, for simplicity of exposition, not expressed the constitutive law in terms of the symmetrized gradient of the displacement.}. Recall, that for any time-dependent function $g$ we use the notation $\{g(s)\}_{0 \le s \le t}$ to include pointwise
evaluation of $g(s)$ and its time-derivative for all $0 \leq s \leq t$.

The objective of homogenization \cite{bensoussan2011asymptotic,blanc2023homogenization,pavliotis2008multiscale} is to remove the small parameter $\varepsilon$ and obtain homogenized constitutive law $\Psi_0$ and homogenized displacement $u$ and stress $\sigma$, related by an equation of the form
\begin{subequations}\label{PDE-e}
\begin{align}
    \rho \p_t^2 u(x,t) & = \nabla_x \cdot \sigma(x,t) + f(x,t) && x \in \cD, t \in \cT \\
    \sigma(x,t)&= \Psi_0(\{\nabla_x u(x,s)\}_{s \in \cT}; M)(t) && x \in \cD, t\in \cT\label{eq:Psi_0}\\
    u(x,0) &= \p_tu(x,0)=0 && x \in \cD\\
    u(x,t) &= 0 && x \in \partial \cD, t \in \cT.
\end{align}
\end{subequations}
When designed properly, this homogenized model delivers $(u,\sigma) \approx (u_{\varepsilon},\sigma_{\varepsilon}).$ However, as it does not involve small parameter $\varepsilon$, it offers considerable computational cost savings over \cref{PDE-ge}.

\begin{remark} \label{rem:causal}
    We note that $\Psi, \Psi_0$ are \textit{causal} functions of the strain trajectory meaning that $\Psi(\cdot)(t)$ (respectively) $\Psi_0(\cdot)(t)$) depends only on the strain\\ history$~\{\nabla_x u_{\varepsilon}(x,s)\}_{0 \leq s \leq t}$,
    (respectively$~\{\nabla_x u(x,s)\}_{0 \leq s \leq t}$) up to time $t$.
    \hspace*{\fill}$\blacklozenge$
\end{remark}

In many situations, an exact expression for $\Psi_0$ is not available. We consider a setting in which the microscale is periodic. Hence, the microstructure takes $M:\T^d \to \R^{d_M}$ where $d_M$ is the dimension of the output material properties, viewed as living in a space isomorphic to a Euclidean space, and where the explicit $x$-dependence in $\Psi$ is through $M(x/\varepsilon)$ only. This results in $\Psi_0$ not depending explicitly on $x$. We let $\sigmabar(t)$ (resp.\ $\epsbar(t)$) denote time-dependent stress (resp.\ strain) found by averaging $\sigma_{\varepsilon}$ and $\nabla_x u_{\varepsilon}$ over the
unit cell with side of length $\varepsilon$. Then we can write $\Psi_0: (\epsbar, M) \mapsto \sigmabar$; in particular, $\Psi_0$ takes as input both a time-dependent function (strain) and spatially varying functions $M$ capturing the microstructure. Our goal in this paper is to define, analyze and numerically study a methodology to determine $\Psi_0$ from numerically generated data. This data will be found by studying PDE \cref{PDE-ge} on a unit cell (one period)
and creating pairs of stress-strain histories, both averaged over the unit cell, for a variety of material properties $M$. From this, we wish to learn an approximation of $\Psi_0$.

\subsection{Neural Operator Constitutive Law}\label{subsec:neuralop0}
Here we define the neural network architecture $\Psi^{\FNMRNO}$, designed to approximate the homogenized constitutive law $\Psi_0$ in \cref{eq:Psi_0}. 
Since $\Psi_0$ does not depend explicitly on $x$, the same is true of our model
$\Psi^{\FNMRNO}$. There are natural generalizations to allow for $x$ dependence
in $\Psi^{\FNMRNO}$, to account for situations where $\Psi_0$ depends explicitly on $x$, but we do not consider these here.

\begin{definition}[FNM--RNO Architecture]\label{def:FNMRNO}
Define the mapping
\begin{equation}\label{eqn:PsiRNO}
\begin{gathered}
        \Psi^{\FNMRNO}: C^1(\cT; \R^{d \times d}) \times L^2(\Td; \R^{d_M}) \to C(\cT; \R^{d \times d}),\\
    (\{\epsbar(t)\}_{t\in \cT}, M) \mapsto \{\sigmabar(t)\}_{t\in \cT},
\end{gathered}
\end{equation}
through the equations
\begin{subequations}\label{eqn:FNMRNO}
\begin{align}
    \sigmabar(t) &= F_{\FNM}(\epsbar(t), \dot{\epsbar}(t), \xi(t); M), \quad t \in \cT,\\
    \dot{\xi}(t) &= G_{\FNM}(\epsbar(t), \xi(t); M), \quad t \in \cT,\\
    \xi(0) &= 0
\end{align}
\end{subequations}
where
\begin{equation}\label{eq:FNM_dims}
\begin{aligned}
    F_{\FNM}&: \R^{d \times d} \times \R^{d \times d} \times \R^L \times L^2(\Td; \R^{d_M}) \to \R^{d \times d}\\
    G_{\FNM}&: \R^{d \times d} \times \R^L \times L^2(\Td; \R^{d_M}) \to \R^{L}.
\end{aligned}
\end{equation}
Here $\xi \in \R^L$ denotes the internal state variable with dimension  $L \in \N_{>0}$ and $M \in L^2(\Td; \R^{d_M})$ is a vector-valued function that specifies the material microstructure on the unit cell.
\end{definition}

As defined, $\Psi^{\FNMRNO}$ is a causal function of the strain trajectory (see \cref{rem:causal}) since it is given by the simulation of a differential equation forced by the strain trajectory. The functions $F_{\FNM}, G_{\FNM}$ are chosen to be Fourier Neural Mappings (FNMs), introduced in~\cite{huang2024operator}; these are neural networks that act on function as well as vector inputs. They are generalizations of Fourier Neural Operators (FNOs)~\cite{li2020fourier}, neural networks that map functions to functions by composing pointwise linear and nonlinear operations on functions defined in the original space, and linear operations in Fourier space. We note that the FNM architecture in \cref{def:FNM} is a slight generalization of the definition given in~\cite{huang2024operator} to allow our architecture to accommodate both finite and infinite-dimensional inputs simultaneously. The architecture has the potential to learn both history dependence, through the recurrent structure in time, and material dependence, through the FNMs. We prove in \cref{lem:FNO_lip} that FNMs have bounded Lipschitz constant with respect to their vector inputs, and therefore implying by regularity of ODEs that $\sigmabar \in C(\cT; \R^{d \times d})$.

\begin{remark}\label{rem:dims}
    The following definition of an FNM incorporates a function input, a vector input, and a vector output. When dealing with multiple function or vector inputs, we assume they are concatenated to form a single input to the FNM. When dealing with matrix inputs or outputs, we assume they are flattened to vector inputs or outputs. Hence, $\epsbar, \dot{\epsbar}, \sigmabar \in \R^{d \times d}$ are flattened to become vectors in $\R^{d^2}$.

    For multiple function inputs, their evaluations at spatial points are concatenated (pointwise flattening). Hence, the input material microstructure $M$ is defined generally as a vector-valued function with a $d_M$-dimensional output. This allows us to use the FNM--RNO architecture to model a wide class of elastic and inelastic materials. For example, microstrucures of viscoelastic materials are defined by the spatial elasticity and viscosity tensor fields $E, \nu: \Td \to \R^{d \times d \times d \times d}$ which can be flattened and concatenated to define $M = (E, \nu): \Td \to \R^{d_M}$ where in this case $d_M = 2d^4$. Of course, under further symmetries such as isotropy conditions on the elasticity and viscosity tensors, they can be summarized into a material microstructure function $M$ with a much smaller dimensionality $d_M$. In the setting of elasto-viscoplastic materials (see \cref{subsec:plastic}), their microstructure is defined by four functions: the elasticity tensor (Young's modulus) $E: \Td \to \R^{d \times d \times d \times d}$, strain rate constant $\dot{\eps}_{p0}: \Td \to \R^{d \times d}$, yield stress $\sigma_Y: \Td \to \R$, and rate exponent $n: \Td \to \R$. Hence, the material microstructure can concatenate these four functions as $M = (E, \dot{\eps}_{p0}, \sigma_Y, n): \Td \to \R^{d_M}$ where in this case $d_M = d^4 + d^2 + 2$.
    \hspace*{\fill}$\blacklozenge$
\end{remark}

With these remarks in mind, we are ready to define the FNM architecture.
In the following definition, $\psi_k = e^{2\pi \iunit \ip{k}{\slot}_{\R^d}}$ are the complex Fourier basis elements of $L^2(\Td;\C).$

\scalebox{0.8}{
\begin{tikzpicture}[->, thick, every node/.style={draw, minimum size=1cm}]
    
    \node (FuncIn) at (-1.5, 0) {$M$};
    \node (VecIn) at (-1.7, -3) {$v_{\text{in}}$};
    \node (Sf) at (3.6, -1.2) {$S_f$};
    \node (Sv) at (0.8, -3) {$S_v$};
    \node (D) at (2.6, -3) {$\sD$};
    \node (L1) at (5.5, -1.2) {$\sL_1$};
    \node (LT) at (7, -1.2) {$\sL_T$};
    \node (G) at (8, -3) {$\sG$};
    \node (Qv) at (10, -3) {$Q_v$};
    \node (VecOut) at (12, -3) {$v_{\text{out}}$};

    \node[above=of FuncIn, yshift=-1cm, draw=none, align=center, font=\bfseries] (InputLabel) {Function\\Input};

    \node[above=of VecIn, yshift=-1cm, draw=none, align=center, font=\bfseries] (InputLabel) {Vector\\Input};
    
    \node[above=of Sf, yshift=-1cm, draw=none, align=center, font=\bfseries] (LiftingLabel) {Function\\Lifting};

    \node[above=of Sv, yshift=-1cm, draw=none, align=center, font=\bfseries] (LiftingLabel) {Vector\\Lifting};

    \node[below=of D, yshift=1cm, draw=none, align=center, font=\bfseries] (DecoderLabel) {Vector to\\Function};

    \node[above=of L1, yshift=-1cm, xshift=0.7cm, draw=none,
    font=\bfseries] (FourierLabel) {Fourier Layers};

    \node[below=of G, yshift=1cm, draw=none, align=center, font=\bfseries] (FunctionalLabel) {Function to\\Vector};

    \node[above=of Qv, yshift=-1cm, draw=none, align=center, font=\bfseries] (ProjectionLabel) {Vector\\Projection};

    \node[above=of VecOut, yshift=-1cm, draw=none, align=center, font=\bfseries] (OutputLabel) {Vector\\Output};

    \draw (FuncIn) -- (Sf);
    \draw (VecIn) -- (Sv);
    \draw (Sv) -- (D);
    \draw (D) -- (Sf);
    \draw (Sf) -- (L1);
    \draw[dashed] (L1) -- (LT);
    \draw (LT) -- (G);
    \draw (G) -- (Qv);
    \draw (Qv) -- (VecOut);
\end{tikzpicture}
}

\begin{definition}[Fourier Neural Mapping (FNM)] \label{def:FNM}
    Let the function input $M \in L^2(\Td; \R^{d_M})$. Define the vector input $v_{\inn} \in \R^{d_{\inn}^v}$ and vector output $v_{\out} \in \R^{d_{\out}^v}$. Let $x\in \Td$. Now we define the following layers: 
    \begin{itemize}[topsep=1.67ex,itemsep=0.5ex,partopsep=1ex,parsep=1ex,leftmargin=25ex]
        \myitem{(Vector Lifting)} $S_v: \R^{d_{\inn}^v} \to \R^{d_{\lift}^v}$
        
        \myitem{(Vector to Function)} $\sD: \R^{d_{\lift}^v} \to L^2(\Td; \R^{d_{\lift}^{vf}})$ \newline 
        $z \mapsto \sD z = \kappa_v(\cdot)z$ \newline
        $z \mapsto \sD z = \left\{\sum_{k \in \Zd} \left(P_v^{(k)}z\right)_j \psi_k\right\}_{j \in [d_{\lift}^{vf}]}$

        \myitem{(Function Lifting)} $S_f: L^2(\Td; \R^{d_M + d_{\lift}^{vf}}) \to L^2(\Td; \R^{d_0})$

        \myitem{(Fourier)} $\sL_t: L^2(\Td; \R^{d_{t-1}}) \to L^2(\Td;\R^{d_t})$, $t \in [T]$, \newline $\bigl(\sL_t(u)\bigr)(x) = \sigma\bigl(W_t u(x) + (\cK_t u)(x) + b_t\bigr)$, 

        \myitem{(Function to Vector)} $\sG: L^2(\Td;\R^{d_T}) \to \R^{d_{\proj}^{fv}}$ \newline
        $h \mapsto \sG h = \int_{\Td} \kappa_f(x) h(x) \dx$ \newline
        $h \mapsto \sG h = \left\{\sum_{k \in \Zd}\left(\sum_{j=1}^{d_T} (P^{(k)}_f)_{\ell j} \ip{\psi_k}{h_j}_{L^2(\Td;\C)}\right)\right\}_{\ell \in [d_{\proj}^{fv}]}$

        \myitem{(Vector Projection)} $Q_v: \R^{d_{\proj}^{fv}} \to \R^{d_{\out}^v}.$
    \end{itemize}
The convolution operator is given, for $u: \Td \to\R^{d_{t-1}}$ and $x\in\Td$, by
\begin{equation}\label{eqn:fno_Kt}
    (\cK_tu)(x) = \left\{\sum_{k \in \Zd}\left(\sum_{j=1}^{d_{t-1}} (P^{(k)}_t)_{\ell j}\ip{\psi_k}{u_j}_{L^2(\Td;\C)} \right)\psi_k(x)\right\}_{\ell\in[d_{{t}}]} \in\R^{d_{t}}\,.
\end{equation}
For given layer index $t$
and wave vector $k \in \Zd$, the matrix $P^{(k)}_t \in \mathbb{C}^{d_{{t}} \times d_{t-1}}$ comprises learnable parameters of the integral operator $\cK_t$; furthermore, $W_t \in \R^{d_{{t}}\times d_{t-1}}$ is a weights matrix, $b_t\in \R^{d_{t}}$ is a bias vector, both learnable. And, for given wave vector $k \in \Zd$, $P_v^{(k)} \in \C^{d_{\lift}^{vf} \times d_{\lift}^v}$ are the learnable parameters of the vector to function map $\sD$, and $P_f^{(k)} \in \C^{d_{\proj}^{fv}\times d_T}$ are the learnable parameters of the function to vector map $\sG$. The vector lifting and
projection layers, $S_v$ and $Q_v$, are either shallow neural networks or linear maps, and hence also contain learnable parameters. Finally the function lifting layer $S_f$ is
applied pointwise in $\Td-$a.e. and is also defined by either a shallow neural network or
a linear map containing learnable parameters.
\end{definition}

\begin{remark}
Note that the function $\kappa_f: \Td \to \R^{d_{\proj}^{fv} \times d_T}$ is parametrized in the Fourier domain, where the coefficients $P_f^{(k)}$ correspond to the Fourier coefficients of $\kappa_f$. Similarly, the function $\kappa_v: \Td \to \R^{d_{\lift}^{vf} \times d_{\lift}^v}$ for the vector to function layer $\sD$ is parameterized in the Fourier domain such that $P^{(k)}_v$ correspond to the Fourier coefficients of $\kappa_v$.
\hspace*{\fill}$\blacklozenge$
\end{remark}

Using the definition of the Fourier Neural Mapping above and comparing to~\cref{eq:FNM_dims}, we know the input and output dimensionalities for $F_{\FNM}$ are $d_{\inn}^v = 2d^2 + L, d_{\out}^v = d^2$ and for $G_{\FNM}$ are $d_{\inn}^v = d^2 + L, d_{\out}^v = d^2$. The input dimensionality $d_M$ of the material microstructure $M$ depends on the material model as discussed in \cref{rem:dims}.

\section{Kelvin--Voigt Viscoelasticity}\label{sec:PB}
We now introduce the classical Kelvin--Voigt (KV) model for a multiscale visocelastic material. The structure of the FNM--RNO architecture $\Psi^{\FNMRNO}$ introduced in the previous section is motivated by the homogenized form $\Psi_0$ of the Kelvin--Voigt model described below. In fact, we will later prove that this neural architecture approximates the homogenized constitutive law of Kelvin--Voigt viscoelasticity to arbitrary accuracy.

We begin in \cref{subsec:KV_homog} by introducing the multiscale Kelvin--Voigt model in one dimension and describe the structure of its average strain-to-stress map $\Psi_0$ resulting from homogenization. We then show in \cref{subsec:mat_dep}, that for piecewise-constant microstructures, the map $\Psi_0$ has an explicit analytical form, with memory captured through a differential equation forced by the strain, whose parameters depend continuously on the material microstructure pieces.

\subsection{Homogenization and Cell Problem}\label{subsec:KV_homog}
Let $E, \nu: \T \to \R$ and $E_{\varepsilon}(x) = E(\frac{x}{\varepsilon})$ and $\nu_{\varepsilon}(x) = \nu(\frac{x}{\varepsilon})$ where $\varepsilon \ll 1$ denotes a small spatial lengthscale. One-dimensional, multiscale Kelvin--Voigt viscoelasticity is governed by the following partial differential equation on a spatial domain $\cD = [0,D]$ and time interval $\cT = [0,T]:$
\begin{subequations}\label{PDE-eps}
\begin{align}
    \rho \p_t^2 u_{\varepsilon}(x,t) & = \p_x \sigma_{\varepsilon}(x,t) + f(x,t) && x \in \cD, t \in \cT \\
    \sigma_{\varepsilon}(x,t)&= E_{\varepsilon}(x)\p_xu_{\varepsilon}(x,t) + \nu_{\varepsilon}(x)\p^2_{xt}u_{\varepsilon}(x,t) && x \in \cD, t\in \cT\\
    u_{\varepsilon}(x,0) &= \p_tu_{\varepsilon}(x,0)=0 && x \in \cD\\
    u_{\varepsilon}(0,t) &= u_{\varepsilon}(D, t) = 0 && t \in \cT.
\end{align}
\end{subequations}
Thus the material properties $E_{\varepsilon}, \nu_{\varepsilon}$ depend only on the microscale variable $y=\frac{x}{\varepsilon}$ and have no dependence on the macroscale variable $x$ independent of $y$. We note that the external forcing $f$ is assumed independent of $\varepsilon$. These assumptions can be relaxed but doing so leads to greater computational complexity when learning homogenized models.

Equations \cref{PDE-eps} are a specific instance of the general setting of \cref{PDE-ge}. In this specific setting the homogenization procedure is outlined in \cite[section 2.2]{bhattacharya2023learning}, a one-dimensionalization of the general case of homogenization for Kelvin--Voigt viscoelasticity developed in \cite{francfort1986homogenization}. The homogenized operator
\begin{equation}\label{eqn:Psi0}
\begin{gathered}
        \Psi_0: C^1(\cT; \R) \times L^2(\T; \R^2) \to C(\cT; \R),\\
    (\{\epsbar(t)\}_{t\in \cT}, E, \nu) \mapsto \{\sigmabar(t)\}_{t\in \cT},
\end{gathered}
\end{equation}
mapping strain to stress is given by the solution of the \textit{cell problem}
\begin{subequations}\label{eq:cell_problem}
    \begin{align}
        &\sigmabar(t) = \smallint_\Omega \sigma(y, t) \dy, &\quad t&\in \cT,\\
        &-\p_y\sigma(y, t) = 0, &\quad (y,t) &\in \Omega \times \cT, \label{eq:balance_of_force}\\
         &\sigma(y, t) = E(y)\p_y u(y, t) +\nu(y)\partial_{yt}u(y, t), &\quad (y,t) &\in \Omega \times \cT, \label{subeq:map}\\
        &u(0, t) = 0, \; u(1, t) = \epsbar(t), &\; t &\in \cT, \\
        &u(y, 0) = 0, &\; y&\in \Omega,
    \end{align}
\end{subequations}
where $\Omega = [0, 1]$ and $\cT = [0, T]$ and the boundary condition $\epsbar(t)$ satisfies $\epsbar(0) = 0$. This version of the equations is derived in \cite[Lemma 3.12]{bhattacharya2023learning}. The boundary condition $\epsbar$ is suggestively written since the spatially averaged strain $\epsbar(t) = \int_{\Omega} \p_y u(y,t) \dy$ is exactly the boundary condition. In one dimension, $\sigma(y,t) = \sigma(t)$ is not spatially dependent due to the balance of forces in \cref{eq:balance_of_force}. Thus, the spatially averaged stress is $\sigmabar = \sigma$. We remark that the cell problem is not well-defined for all microstructures $(E, \nu) \in L^2(\T; \R^2)$ and hence in our theoretical analysis of this PDE, we constrain the microstructure to be positive, bounded from above and below, and of bounded variation (see \cref{assump:E_nu_eps} below).

A useful procedure to analyze this system is to decompose the solution of our cell problem into a heterogeneous periodic component and a homogeneous nonperiodic component as
\begin{equation}\label{eq:periodic_decomp}
    u(y, t) = p(y, t) + \epsbar(t)y,
\end{equation}
where $p(y, t)$ satisfies the \textit{detrended} cell problem
\begin{subequations}\label{eq:detrended_cell_problem}
    \begin{align}
         \p_y\Big(\nu\p_{yt}p + E\p_y p\Big) &= -\dot{\epsbar}\p_y\nu -
         \epsbar(t)\p_yE(y), &(y,t) &\in \Omega \times \cT, \label{subeq:map}\\
        p(0, t) &= p(1, t) = 0, &t &\in \cT, \\
        p(y, 0) &= 0, &y&\in \Omega.
    \end{align}
\end{subequations}
Note that we refer to $p(y, t)$ as the periodic component of the solution, but more precisely it is the solution to the Dirichlet detrended cell problem above with homogeneous boundary conditions. Our analysis in \cref{subsec:C} and \cref{sec:Lip} proves Lipschitz regularity of the cell problem in \cref{eq:cell_problem} and these arguments rely on integration by parts formulae which are easier to express in terms of $p(y, t)$.

We define a norm under which we can study the magnitude of solutions $u, p$ to the original and detrended cell problems above. Following the notation in~\cite{bhattacharya2023learning}, we define the $\xi$-dependent quadratic form
\begin{equation}
    q_\xi(u, v) := \int_\Omega \xi(y)\p_yu(y)\p_yv(y)\dy
\end{equation}
for arbitrary $\xi \in L^\infty(\Omega; (0, \infty))$. Define the bounds
\begin{equation}\label{eq:xi_bounds}
    \xi_{\max} := \esssup_{x \in \Omega}\xi(x) < \infty, \quad \xi_{\min} := \essinf_{x \in \Omega}\xi(x) > 0.
\end{equation}
Note that $q_\xi(\cdot, \cdot)$ defines an inner product with resulting norm
\begin{equation}\label{eq:H01_weighted}
    \|u\|_{H_0^1, \xi}^2 := q_\xi(u, u).
\end{equation}
In the case that $\xi = \mathds{1}(\cdot)$ is the constant unit function, we write
\begin{equation}
    \|u\|_{H_0^1}^2 := q_\mathds{1}(u, u).
\end{equation}
The norms $\|\cdot\|_{H_0^1, \xi}$ are equivalent for all $\xi \in L^\infty(\Omega; (0, \infty))$ as shown in the following:
\begin{lemma}[Lemma 1.1 in~\cite{bhattacharya2023learning}]\label{lem:norm_equiv}
    For any $\xi, \zeta \in L^\infty(\Omega; (0, \infty))$ satisfying properties~\cref{eq:xi_bounds}, the norms $\|\cdot\|_{H_0^1, \xi}$ and $\|\cdot\|_{H_0^1, \zeta}$ are equivalent in the sense that
    \begin{equation}
        \frac{\zeta_{\min}}{\xi_{\max}}\|u\|_{H_0^1, \xi}^2 \leq \|u\|_{H_0^1, \zeta}^2 \leq \frac{\zeta_{\max}}{\xi_{\min}}\|u\|_{H_0^1, \xi}^2.
    \end{equation}
\end{lemma}
Hence, we can use any inner product $q_\xi(u, v)$ for $\xi$ satisfying~\cref{eq:xi_bounds} since they are all equivalent. 

We can write the \textit{weak form} of our cell problem in~\cref{eq:cell_problem},
seeking solution $u \in C^1\bigl(\cT; H_0^1(\Omega; \R)\bigr)$
satisfying
\begin{subequations}
\label{eq:wf}
\begin{align}
    q_\nu(\p_tu, \varphi) + q_E(u, \varphi) &= 0, \quad \forall \varphi \in H_0^1(\Omega; \R), t \in \cT,\\
    u & =0, \quad t=0.
\end{align}
\end{subequations}
 Finally, we note that the solution to the cell problem $u$ can be interpreted as a function of time that maps into $L^2(\Omega; \R)$ so it lives in $\cZ = L^\infty(\cT; L^2(\Omega; \R))$ equipped with the norm $\|u\|_\cZ = \esssup_{t \in \cT}(\|u(\cdot, t)\|_{L^2})$. Recall the notation for the sets $\cM^B_{\cdot, \cdot}$ and $\cC_{\cdot, \cdot}$ from \cref{subsec:N}.

\subsection{Material Dependence}\label{subsec:mat_dep}
The goal of this and the next section is to study the properties of,
and approximate, the homogenized map $(\epsbar, E, \nu)\mapsto \sigmabar$, where $\epsbar$ and $\sigmabar$ are shorthand for $\{\epsbar(t)\}_{t\in\cT}$ and $\{\sigmabar(t)\}_{t\in \cT}$ respectively, defined by \cref{eq:cell_problem}. This map allows us to study how the average stress $\sigmabar$ depends on the material properties $E, \nu$ and average strain boundary condition $\epsbar$. We  use the assumptions:
\begin{assumptions}\label{assump:E_nu_eps}
    We make the following assumptions on $E, \nu,$ and $\epsbar$ throughout:
    \begin{enumerate}
        \item  Assume that for the constants $0 < E_{\min} \leq E_{\max} < \infty$, $0 < \nu_{\min} \leq \nu_{\max} < \infty$ and $B>0$ we have that $E \in \cM_{E_{\min}, E_{\max}}^B$ and $\nu \in \cM_{\nu_{\min}, \nu_{\max}}^B$.
        \item  Assume that for the constants $0 < \epsbar_{\max}, \dot{\epsbar}_{\max} < \infty$ we have that $\epsbar \in \cC_{\epsbar_{\max}, \dot{\epsbar}_{\max}}.$ 
    \end{enumerate}
\end{assumptions}

\begin{remark} 
By definition of the sets $\cM^B_{\cdot, \cdot}$ and $\cC_{\cdot, \cdot}$ from \cref{subsec:N} we see that the bounds implied by the preceding assumptions hold using the $\esssup$, $\essinf$ over the cell problem domain $y \in \Omega$ or the time domain $t \in \cT$. In the remainder of the paper we will simply write $\sup$ and $\inf$ for notational brevity, but essential suprema and infima are implied.
\hspace*{\fill}$\blacklozenge$
\end{remark}

As proved in the seminal paper of Francfort and Suquet~\cite{francfort1986homogenization}, the map $\Psi_0$ in \cref{eq:cell_problem} takes the form
\begin{equation}\label{eqn:kernelform}
\sigmabar(t) = \Psi_0(\{\epsbar(t)\}_{t\in \cT}; E, \nu) := E'\epsbar(t) + \nu' \dot{\epsbar}(t) - \int_0^t K(t-s) \epsbar(s) \ds.
\end{equation}

For the 1D cell problem~\cref{eq:cell_problem}, the parameters $E', \nu'$ (which we refer to as the \emph{Markovian parameters} in what follows as they do not involve memory) are derived in \cite[Appendix B.1]{bhattacharya2023learning} and shown to take the following form:
\begin{equation}\label{eq:markov_params}
    E' = \frac{\int_0^1\frac{E(y)}{\nu(y)^2}\dy}{\Big(\int_0^1\frac{1}{\nu(y)}\dy\Big)^2}, \quad \nu' = \frac{1}{\int_0^1\frac{1}{\nu(y)}\dy}.
\end{equation}
The memory kernel is given in the Laplace domain as
\begin{equation}
    \mathcal{L}[K](s) = E' + \nu's - \Big(\int_0^1\frac{\dy}{E(y)+\nu(y)s}\Big)^{-1},
\end{equation}
where $\mathcal{L}[K]: \R \to \R$ is the Laplace transform of $K: \R \to \R$. Using detailed complex analysis techniques, the inverse Laplace transform of this formula can be taken, to derive a closed-form expression for $K(t) = \int_\R e^{-\alpha t} \d\mu(\alpha)$ given by a continuum integral over exponential decays. These exponents are weighted by a measure $\mu$ that depends explicitly on the $E, \nu$ microstructure parameters~\cite{darrow2025spectral}.

When the microstructure parameters are piecewise-constant functions, the memory kernel $K(t)$ becomes a finite sum of exponentials, as shown by the following result:
\begin{proposition}[Theorem 3.6 in~\cite{bhattacharya2023learning}]\label{thm:pc_restate}
    Assume $E, \nu$ are piecewise-constant materials with $L$ pieces of positive lengths $\{d_{\ell}\}_{\ell \in [L]}$ where
    \begin{equation}
        E(y) = E_i, \ \nu(y) = \nu_i, \quad y \in \Big[\sum_{l=1}^{i-1}d_l, \sum_{l=1}^id_l\Big)
    \end{equation}
    with $i \in [L]$, and where the piece lengths add up to $\sum_{i=1}^L d_i = 1$. Then the map from $\{\epsbar(t)\}_{t\in \cT}$ to $\{\sigmabar(t)\}_{t\in\cT}$ is given by the integro-differential Volterra equation
    \begin{equation}
        \sigmabar(t) = E_{\pc}'\epsbar(t) + \nu_{\pc}' \dot{\epsbar}(t) - \int_0^t K_{\pc}(t-s) \epsbar(s) \ds
    \end{equation}
    where the memory kernel is given by
    \begin{equation}\label{eq:volterra_constitutive} 
        K_{\pc}(t) = \sum_{l=1}^{L-1} \beta_le^{-\alpha_l t}.
    \end{equation}
    The Markovian parameters~\eqref{eq:markov_params} are given by
    \begin{equation}\label{eqn:PC-E-nu}
        E'_{\pc} = L\frac{\sum_{i=1}^Ld_i\frac{E_i}{\nu_i^2}}{\Big(\sum_{i=1}^L\frac{d_i}{\nu_i}\Big)^2}, \quad \nu_{\pc}' = \frac{1}{\sum_{i=1}^L\frac{d_i}{\nu_i}}.
    \end{equation}
    Defining the two polynomials
    \begin{equation}
        P(s) = \prod_{i=1}^d(E_i - \nu_i s), \quad Q(s) = \sum_{i=1}^Ld_i\prod_{j \neq i}(E_j - \nu_j s).
    \end{equation}
    the exponential decays $\{\alpha_l\}_{l=1}^{L-1}$ of the memory kernel are defined as the roots of $Q(s)$ and the exponent coefficients $\{\beta_l\}_{l=1}^{L-1}$ are given by the residues around the poles of the rational function $P(s)/Q(s)$ which are all positive valued.
    
    Finally, the Volterra equation~\cref{eq:volterra_constitutive} relating strain-to-stress is equivalent to the differential equation model
    \begin{equation}\label{eq:diffeq_constitutive}
    \begin{aligned}
        \sigmabar(t) &= E'_{\pc}\epsbar(t) + \nu'_{\pc}\dot{\epsbar}(t) + \langle \mathds{1}_{L-1}, \xi(t) \rangle, && t \in \cT\\
        \dot{\xi}(t) & = -A\xi(t) + b\epsbar(t), && t \in \cT\\
        \xi(0) & = 0
    \end{aligned}
    \end{equation}
    where the matrix $A$ is diagonal with positive entries $\{\alpha_l\}_{l=1}^{L-1}$ and the coefficient vector $b = \{\beta_l\}_{l=1}^{L-1} \in \R_+^{L-1}$.
\end{proposition}

Next, we build on the preceding proposition to derive the form of the coefficients $\beta_l$ explicitly; and then to conclude that the parameters of the differential equation~\cref{eq:diffeq_constitutive} depend continuously on the piecewise-constant material parameterization.
\begin{lemma}\label{lemma:pc_cont}
    The vector of coefficients $b = \{\beta_l\}_{l=1}^{L-1} \in \R_+^{L-1}$ from~\cref{eq:diffeq_constitutive} are defined by the following closed form
    for the inverse of the components:
    \begin{equation}\label{eq:beta_l}
        \beta_l^{-1} = \sum_{i=1}^L\frac{d_i}{\nu_i} \cdot \frac{1}{\big(\frac{E_i}{\nu_i} - \alpha_l\big)^2}.
    \end{equation}
    From this we conclude $(E'_{\pc}, \nu'_{\pc}, A, b)$ is a continuous functions of  material parameters $\{(d_i, E_i, \nu_i)\}_{i=1}^L$ provided these material parameters are all strictly positive.
\end{lemma}
\begin{proof}
First we establish expression \cref{eq:beta_l}. Taking the polynomials $P(s), Q(s)$ defined in \cref{thm:pc_restate}, and recalling that the $\beta_l$ are the residues around the poles of $P(s)/Q(s)$, we see that
    \begin{align*}
        \beta_l = \lim_{s \to \alpha_l}\frac{P(s)}{Q(s)}(s - \alpha_l) = -\lim_{s \to \alpha_l} \frac{\prod_{i=1}^L (E_i - \nu_is)(\alpha_l - s)}{\sum_{i=1}^Ld_i\prod_{j \neq i}(E_j - \nu_j s)}.
    \end{align*}
    Now applying l'H\^opital's rule, we get that
    \begin{align*}
        \beta_l = -\lim_{s \to \alpha_l} \frac{\alpha_l - s}{\sum_{i=1}^L\frac{d_i}{E_i - \nu_i s}} = \frac{1}{\sum_{i=1}^L\frac{d_i\nu_i}{(E_i - \nu_i\alpha_l)^2}} = \frac{1}{\sum_{i=1}^L\frac{d_i}{\nu_i} \cdot \frac{1}{\big(\frac{E_i}{\nu_i} - \alpha_l\big)^2}}.
    \end{align*}

    Note that the roots $\{\alpha_l\}_{l=1}^{L-1}$ are clearly continuous functions of the materials parameters $\{(d_i, E_i, \nu_i)\}_{i=1}^L$, as long as these parameters are strictly positive, because they depend continuously on the coefficients of the polynomial $Q$. Likewise, the coefficients $\{\beta_l\}_{l=1}^{L-1}$ and the Markovian parameters $E'_{\pc}, \nu'_{\pc}$ are continuous functions of the material parameters when these parameters are strictly positive. The continuity of the coefficients $\beta_l$ is easy to see except at the possible poles of its denominator where $\alpha_l = \frac{E_l}{\nu_l}$; but noting that $\beta_l$ must tend to zero at such poles, as a function of the material parameters, establishes continuity there.
\end{proof}

Examination of \cref{thm:pc_restate} shows that
the solution map $(\epsbar; E, \nu) \mapsto \sigmabar$ is invariant under permutation
of the piecewise-constant material pieces. This allows us to sort the pieces for mathematical convenience, as in the following lemma.

\begin{lemma}\label{lemma:pc_bound}
    Sort the ratios $\{\frac{E_l}{\nu_l}\}_{l=1}^L$ in increasing order.
    Then the roots $\{\alpha_l\}_{l=1}^{L-1}$ may also be sorted in increasing order, and satisfy the bounds
    \begin{equation}
        \frac{E_l}{\nu_l} \leq \alpha_l \leq \frac{E_{l+1}}{\nu_{l+1}}, \quad \ell \in [L-1];
    \end{equation}
equality is achieved if and only if $\frac{E_l}{\nu_l} = \frac{E_{l+1}}{\nu_{l+1}}$.
In this ordering it also follows that
    \begin{equation}
    \label{eq:ub}
        \beta_l \leq \frac{1}{\sum_{i=1}^L\frac{d_i}{\nu_i}}\Big(\frac{E_L}{\nu_L} - \frac{E_1}{\nu_1}\Big)^2
    \end{equation}
    where $\{\frac{E_l}{\nu_l}\}_{l=1}^L$ are ordered increasingly.
\end{lemma}
\begin{proof}
    First, to show the interleaving property of the roots, suppose we order the indices $l \in [L]$ in increasing order of $\frac{E_l}{\nu_l}$ and assume that these ratios are unique with no repetitions so they are strictly increasing. Recall the polynomial
    \begin{equation}
        Q(s) = \sum_{i=1}^Ld_i\prod_{j \neq i}(E_j - \nu_j s)
    \end{equation}
    of which the $\alpha_l$ are roots. Then we have that
    \begin{equation}
        Q\Big(\frac{E_k}{\nu_k}\Big) = \prod_{l=1}^L \nu_l \cdot \sum_{l=1}^L\frac{d_l}{\nu_l}\prod_{j \neq l}\Big(\frac{E_j}{\nu_j} - \frac{E_k}{\nu_k}\Big) = \prod_{l=1}^L \nu_l \cdot \frac{d_k}{\nu_k}\prod_{j \neq k}\Big(\frac{E_j}{\nu_j} - \frac{E_k}{\nu_k}\Big).
    \end{equation}
    which implies that
    \begin{equation}
        \text{sign}\Big[Q\Big(\frac{E_k}{\nu_k}\Big)\Big] = (-1)^{k-1}
    \end{equation}
    Because the polynomial $Q$ has $L$ roots and alternates sign at every $\frac{E_{l}}{\nu_l}$ and must be nonzero at these points, it follows that the roots must lie strictly in between these points. Hence, we have that
    \begin{equation}\label{eq:alpha_discrete}
        \frac{E_l}{\nu_l} < \alpha_l < \frac{E_{l+1}}{\nu_{l+1}}
    \end{equation}
    where the inequalities above are strict. Now suppose again that we have a list of unique ratios $\{\frac{E_l}{\nu_l}\}_{l=1}^K$ sorted in strictly nondecreasing order, but every element in this list is repeated $\mathcal{N}_l$ times such that $\sum_{l=1}^K \mathcal{N}_l = L$. Then, by factoring out the term $\prod_{l=1}^K (E_l/\nu_l - s)^{\mathcal{N}_l-1}$ from $Q(s)$ and combining like terms, we arrive at a new polynomial of the same form as $Q(s)$ with all distinct ratios to which we can apply the previous argument above. This proves the interleaving property of the $\alpha_l$ roots. Using the expression for $\beta_l$ derived in~\cref{eq:beta_l}, we can also immediately establish the upper bound
    \cref{eq:ub} assuming again that the ratios $\frac{E_l}{\nu_l}$ are sorted in increasing order.
\end{proof}

\cref{lemma:pc_cont} and \cref{lemma:pc_bound} show that in the case of piecewise-constant materials, the coefficients of the differential equation constitutive law~\cref{eq:diffeq_constitutive} depend continuously on the collection of materials parameters $\{(d_i, E_i, \nu_i)\}_{i=1}^L$ and are bounded. We use this fact to show that we can approximate the stress-strain dynamics of continuously varying materials by their piecewise-constant discretizations.

\section{Universal Approximation}
\label{sec:A}
The central result of this section is a universal approximation theorem, for the homogenized stress-strain relation arising in one-dimensional Kelvin--Voigt viscoelasticity, within the class of FNM--RNO mappings. To achieve this we first establish Lipschitz properties of the 
cell problem, with respect to its dependence on material properties; see \cref{subsec:C}. We then show that the homogenized constitutive law defined by \eqref{eqn:Psi0} and \cref{eq:cell_problem} may be approximated by the homogenized constitutive law associated with a piecewise-constant approximation of the material, in \cref{subsec:4.2}. This result is then used, in \cref{subsec:neuralop}, to establish a universal approximation theorem for our proposed architecture: for any error $\e>0$ there exists a choice of parameters in FNM--RNO that leads to an $\e-$approximation of the map $\{\epsbar, E, \nu\} \mapsto \sigmabar$, uniformly across a compact set of inputs.

\subsection{Lipschitz Regularity of Cell Problem}
\label{subsec:C}
Here we show the Lipschitz regularity of the Kelvin--Voigt cell problem~\cref{eq:cell_problem}. Let $u_1, u_2$ be solutions corresponding to material parameters $(E_1, \nu_1)$ and $(E_2, \nu_2)$ respectively, both satisfying the conditions in \cref{assump:E_nu_eps}. We can write these cell problems as
\begin{subequations}\label{eq:cell_problem_i}
\begin{align}
     &\p_y\Big(\nu_i(y)\partial_{yt}u_i(y, t) + E_i(y)\p_y u_i(y, t)\Big) = \p_y\sigma_i(y, t) = 0, &\quad (y,t) &\in \Omega \times \cT \label{subeq:map_i}\\
    &u_i(0, t) = 0, \; u_i(1, t) = \epsbar(t), &\; t &\in \cT \\
    &u_i(y, 0) = 0, &\; y&\in \Omega.
\end{align}
\end{subequations}
where the strains corresponding to these stresses are given by
\begin{equation}\label{eq:stress_i}
    \sigma_i(y, t) = \nu_i(y)\p_{yt}u_i(y, t) + E_i(y)\p_y u_i(y, t), \quad (y,t) \in \Omega \times \cT
\end{equation}

Our goal is to bound the difference between the spatially averaged stresses $\sigmabar_1 = \langle\sigma_1, \mathds{1}\rangle$ and $\sigmabar_2 = \langle\sigma_2, \mathds{1}\rangle$ of these two cell problems based on the difference of their material parameters. We do this by first bounding the distance between the solutions $u_1, u_2$ of these two PDEs. Taking the two equations~\cref{subeq:map_i} satisfied by $u_1$ and $u_2$ we can rewrite them as
\begin{align*}
\p_y\Big(\nu_1\p_{yt}u_1 + E_1\p_y u_1\Big) &= 0\\
\p_y\Big(\nu_1\p_{yt}u_2 + E_1\p_y u_2\Big) &= \p_y\Big((\nu_1 - \nu_2)\partial_{yt}u_2 + (E_1 - E_2)\p_y u_2\Big).
\end{align*}
Defining the difference functions
\begin{equation}
    \gamma = u_1 - u_2, \quad \Delta E = E_1 - E_2, \quad \Delta\nu = \nu_1 - \nu_2, \quad g = \Delta\nu\partial_{yt}u_2 + \Delta E\p_yu_2,
\end{equation}
we can subtract the equations above to get
\begin{align*}
    \p_y\Big(\nu_1\p_{yt}\gamma + E_1\p_y\gamma\Big) = -\p_yg.
\end{align*}
Choosing any test function $\varphi \in H_0^1(\Omega)$ we can write the weak form of this PDE as
\begin{equation}\label{eq:gamma_weak}
    q_{\nu_1}(\p_t\gamma, \varphi) + q_{E_1}(\gamma, \varphi) = -\langle g, \p_y\varphi\rangle.
\end{equation}
Now, we are ready to state the following Lipschitz bound on the difference between $u_1$ and $u_2$. Both cell problem solutions $u_1, u_2$ can be viewed as functions of time that map into $L^2(\Omega; \R)$ so they live in the function space $\cZ = L^\infty(\cT; L^2(\Omega; \R))$. We will measure their difference under the norm $\|u\|_\cZ = \esssup_{t \in \cT}(\|u(\cdot, t)\|_{L^2})$.

\begin{lemma}\label{lem:u1_u2}
Let $u_i$ be the solution to the cell problem~\cref{eq:cell_problem_i} associated with material properties $E_i, \nu_i$ for $i = 1, 2$ and a time-varying boundary condition $\epsbar(t)$ satisfying \cref{assump:E_nu_eps}. Then we have the Lipschitz bound
\begin{equation}
    \|u_1 - u_2\|_\cZ \leq C_1\|\nu_1 - \nu_2\|_{L^2} + C_2\|E_1 - E_2\|_{L^2}
\end{equation}
where the constants $C_1, C_2 > 0$ depend only on $E_{\min}, E_{\max}, \nu_{\min}, \nu_{\max}$ and $\epsbar_{\max}, \dot{\epsbar}_{\max}$.
\end{lemma}

\begin{proof}
In \cref{prop:gamma_bounds}, we establish that
\begin{equation}
    \sup_{t \in \cT}\|\gamma\|_{H_0^1, \nu_1} \leq \frac{\nu_{\max}}{E_{\min}}\frac{1}{\sqrt{\nu_{\min}}}\|g\|_\cZ.
\end{equation}
Combining this with \cref{lem:norm_equiv} gives us
\begin{align*}
    \sup_{t \in \cT}\|\gamma\|_{H_0^1} \leq \frac{\nu_{\max}}{E_{\min}\nu_{\min}}\|g\|_\cZ.
\end{align*}
By the Poincar\'e inequality and the definition of the norm on $\cZ$ we have that $\|\gamma\|_\cZ \leq C_p\sup_{t \in \cT}\|\gamma\|_{H_0^1}$ for some constant $C_p > 0$ and hence,
\begin{equation}\label{eq:u1_u2}
    \|u_1 - u_2\|_\cZ = \|\gamma\|_\cZ \leq C_p\frac{\nu_{\max}}{E_{\min}\nu_{\min}}\|g\|_\cZ.
\end{equation}
Thus we focus on  bounding the norm of $g$. By Cauchy-Schwarz we write
\begin{align*}
    \|g\|_\cZ &= \|\Delta\nu\partial_{yt}u_2 + \Delta E\p_yu_2\|_\cZ\\
    &\leq \sup_{t \in \cT}\|\partial_{yt}u_2\|_{L^2}\|\Delta\nu\|_{L^2} + \sup_{t \in \cT}\|\p_yu_2\|_{L^2}\|\Delta E\|_{L^2}\\
    &\leq \sup_{t \in \cT}\|\p_tu_2\|_{H_0^1}\|\Delta\nu\|_{L^2} + \sup_{t \in \cT}\|u_2\|_{H_0^1}\|\Delta E\|_{L^2}\\
    &\leq \frac{1}{\sqrt{\nu_{\min}}}\Big(\sup_{t \in \cT}\|\p_tu_2\|_{H_0^1, \nu_2}\|\Delta\nu\|_{L^2} + \sup_{t \in \cT}\|u_2\|_{H_0^1, \nu_2}\|\Delta E\|_{L^2}\Big)
\end{align*}
where the last line follows again from \cref{lem:norm_equiv}. In \cref{cor:u_bounds}, we bound the solution of the cell problem
to show that $\sup_{t \in \cT}\|u_2\|_{H_0^1, \nu_2}$ and $\sup_{t \in \cT}\|\p_tu_2\|_{H_0^1, \nu_2}$ are finite, which implies that
\begin{align*}
    \|g\|_\cZ \leq C_1'\|\nu_1 - \nu_2\|_{L^2} + C_2'\|E_1 - E_2\|_{L^2}
\end{align*}
for constants $C_1', C_2' > 0$ that depend only on $E_{\min}, E_{\max}, \nu_{\min}, \nu_{\max}$ and $\epsbar_{\max}, \dot{\epsbar}_{\max}$. Finally, combining this with~\cref{eq:u1_u2}, gives us
\begin{equation}
    \|u_1 - u_2\|_\cZ \leq C_1\|\nu_1 - \nu_2\|_{L^2} + C_2\|E_1 - E_2\|_{L^2}
\end{equation}
where constants $C_1, C_2 > 0$ depend only on $E_{\min}, E_{\max}, \nu_{\min}, \nu_{\max}$ and $\epsbar_{\max}, \dot{\epsbar}_{\max}$.
\end{proof}

We now use the lemma above to show that the stress resulting from the cell problem also satisfies Lipschitz regularity with respect to the material parameters.

\begin{lemma}\label{lem:sigma1_sigma2}
Let $\sigma_i$ be the stress~\cref{eq:stress_i} resulting from the solution $u_i$ of the cell problem~\cref{eq:cell_problem_i} associated with material properties $E_i, \nu_i$ for $i = 1, 2$ and a time-varying boundary condition $\epsbar(t)$ satisfying \cref{assump:E_nu_eps}. Then we have the Lipschitz bound
\begin{equation}
    \|\sigma_1 - \sigma_2\|_\cZ \leq C_1\|\nu_1 - \nu_2\|_{L^2} + C_2\|E_1 - E_2\|_{L^2}
\end{equation}
where constants $C_1, C_2 > 0$ depend only on $E_{\min}, E_{\max}, \nu_{\min}, \nu_{\max}$ and $\epsbar_{\max}, \dot{\epsbar}_{\max}$.
Define the spatial averages of the two stresses as $\sigmabar_1 = \langle\sigma_1, \mathds{1}\rangle$ and $\sigmabar_2 = \langle\sigma_2, \mathds{1}\rangle$ where $\mathds{1}$ is the constant function taking value one in $\Omega$. Then
\begin{equation}
        \|\sigmabar_1 - \sigmabar_2\|_{L^\infty} \leq C_1\|\nu_1 - \nu_2\|_{L^2} + C_2\|E_1 - E_2\|_{L^2}.
    \end{equation}
\end{lemma}
\begin{proof}
We define $\gamma, \Delta E, \Delta\nu$ and $g$ as before and note that
    \begin{align*}
        \|\sigma_1 - \sigma_2\|_\cZ &= \|\nu_1\p_{yt}\gamma + E_1\p_y\gamma\|_\cZ + \|g\|_\cZ\\
        &\leq \nu_{\max}\|\p_{yt}\gamma\|_\cZ + E_{\max}\|\p_y\gamma\|_\cZ + \|g\|_\cZ\\
        &= \nu_{\max}\sup_{t \in \cT}\|\p_t\gamma\|_{H_0^1} + E_{\max}\sup_{t \in \cT}\|\gamma\|_{H_0^1} + \|g\|_\cZ.
    \end{align*}
    We prove in \cref{prop:gamma_bounds} that $\sup_{t \in \cT}\|\p_t\gamma\|_{H_0^1}$ and $\sup_{t \in \cT}\|\gamma\|_{H_0^1}$ can both be bounded by constant multiples of $\|g\|_\cZ$. Hence, we can bound
    \begin{equation}
        \|\sigma_1 - \sigma_2\|_\cZ \leq C\|g\|_\cZ
    \end{equation}
    for a constant $C > 0$ that depends only on $E_{\min}, E_{\max}, \nu_{\min}, \nu_{\max}$ and $\epsbar_{\max}, \dot{\epsbar}_{\max}$. Finally, using the bound we derived on $\|g\|_\cZ$ in the proof of \cref{lem:u1_u2}, this shows that
    \begin{equation}
        \|\sigma_1 - \sigma_2\|_\cZ \leq C_1\|\nu_1 - \nu_2\|_{L^2} + C_2\|E_1 - E_2\|_{L^2}
    \end{equation}
    where $C_1, C_2 > 0$ are constants that depend only on $E_{\min}, E_{\max}, \nu_{\min}, \nu_{\max}$ and $\epsbar_{\max}, \dot{\epsbar}_{\max}$. The desired result about spatial averages follows by noting that
    $\sigma_i$ is in fact constant in $D$.
\end{proof}

Finally, these results prove that the strain-to-stress map $\Psi_0$ of a one-dimensional homogenized Kelvin-Voigt material is Lipschitz with respect to the material parameters.

\begin{corollary}\label{cor:Psi0_lip}
    For any material parameters $E_1, \nu_1$ and $E_1, \nu_2$ along with average strain input $\epsbar$ satisfying \cref{assump:E_nu_eps}, we have that
    \begin{equation}
        \|\Psi_0(\epsbar; E_1, \nu_1) - \Psi_0(\epsbar; E_2, \nu_2)\|_{L^\infty} \leq C_1\|\nu_1 - \nu_2\|_{L^2} + C_2\|E_1 - E_2\|_{L^2}
    \end{equation}
    where $C_1, C_2 > 0$ depend only on $E_{\min}, E_{\max}, \nu_{\min}, \nu_{\max}$ and $\epsbar_{\max}, \dot{\epsbar}_{\max}$.
\end{corollary}

\subsection{Approximation of the PDE by Piecewise-Constant Problems}
\label{subsec:4.2}
In this section, we use the Lipschitz property of the cell problem derived above to show that the homogenized Kelvin--Voigt constitutive model can be well-approximated by a strain-to-stress map that takes the material microstructure parameters $E, \nu$, discretizes them into piecewise-constant functions, and then applies the exact analytical formula for the piecewise-constant strain-to-stress map derived in Section~\ref{subsec:mat_dep}. This strain-to-stress map is denoted by
\begin{equation}\label{eqn:Psipc}
\begin{gathered}
        \Psi^{\pc}: C^1(\cT; \R) \times \cM_{E_{\min}, E_{\max}}^B \times \cM_{\nu_{\min}, \nu_{\max}}^B \to C(\cT; \R),\\
    (\{\epsbar(t)\}_{t\in \cT}, E, \nu) \mapsto \{\sigmabar_{\pc}(t)\}_{t\in \cT},
\end{gathered}
\end{equation}
and is given by a differential equation model
\begin{equation}\label{eq:pc_model}
\begin{aligned}
\sigmabar_{\pc}(t) &= F_{\pc}(\epsbar(t), \dot{\epsbar}(t), \xi(t); E, \nu)\\
\dot{\xi}_{\pc}(t) &= G_{\pc}(\epsbar(t), \xi(t); E, \nu)\\
\xi_{\pc}(0) &= 0.
\end{aligned}
\end{equation}
Note that the resulting definition of $\Psi^{\pc}$ is then causal (see \cref{rem:causal}).
The average strain and strain rate are $\epsbar(t), \dot{\epsbar}(t) \in \R$ and the ODE evolves a set of $n$-dimensional internal variables $\xi_{\pc}(t) \in \R^n$. We define the maps
\begin{equation}
\begin{aligned}
    &F_{\pc}: \R^{n+2} \times \cM_{E_{\min}, E_{\max}}^B \times \cM_{\nu_{\min}, \nu_{\max}}^B \to \R\\
    &G_{\pc}: \R^{n+1} \times \cM_{E_{\min}, E_{\max}}^B \times \cM_{\nu_{\min}, \nu_{\max}}^B \to \R
\end{aligned}
\end{equation}
which have the following semi-linear form:
\begin{subequations}\label{eqn:FGpc}
\begin{align}
F_{\pc}(\epsbar(t), \dot{\epsbar}(t), \xi(t); E, \nu) & = E_{\pc}'(E, \nu)\epsbar(t) + \nu_{\pc}'(E, \nu)\dot{\epsbar}(t) + \ip{\mathds{1}}{\xi(t)} \\
G_{\pc}(\epsbar(t), \xi(t); E, \nu) & = -A(E, \nu)\xi(t) + b(E, \nu)\epsbar(t).
\end{align}
\end{subequations}
The coefficients of $F_{\pc}, G_{\pc}$ act linearly on the strain $\epsbar$, strain rate $\dot{\epsbar}$, and hidden state $\xi_{\pc}$. Because the entries of $A(E, \nu), b(E, \nu)$ are bounded (see \cref{eq:Ab_bound}), the dynamics of $\xi_{\pc}$ in \cref{eqn:FGpc} are continuous and hence $\sigmabar_{\pc} \in C(\cT; \R)$ as expected.

For any microstructure $E, \nu: \T \to \R$ define its \textit{piecewise-constant discretization} $E_{\pc}, \nu_{\pc}: \T \to \R$ with $n+1$ pieces as
\begin{equation}
\begin{aligned}
    E_{\pc}(y) &= {E_{\pc}}_i := (n+1)\int_{y_i}^{y_{i+1}}E(z)\dz, \\
    \nu_{\pc}(y) &= {\nu_{\pc}}_i = (n+1)\int_{y_i}^{y_{i+1}}\nu(z)\dz, \quad y \in [y_i, y_{i+1})
\end{aligned}
\end{equation}
for $i \in [n+1]$ where $y_i = \tfrac{i-1}{n+1}$. The coefficient functions of the ODE model~\cref{eqn:FGpc} are
\begin{equation}
    \begin{aligned}
        A&: \cM_{E_{\min}, E_{\max}}^B \times \cM_{\nu_{\min}, \nu_{\max}}^B \to \D_+^n\\
        b&: \cM_{E_{\min}, E_{\max}}^B \times \cM_{\nu_{\min}, \nu_{\max}}^B \to \R_+^n\\
        E'_{\pc}, \nu'_{\pc}&: \cM_{E_{\min}, E_{\max}}^B \times \cM_{\nu_{\min}, \nu_{\max}}^B \to \R_+,
    \end{aligned}
\end{equation}
defined in Proposition~\ref{thm:pc_restate} and Lemma~\ref{lemma:pc_cont} as continuous functions of the piecewise constants $\{{E_{\pc}}_i, {\nu_{\pc}}_i\}_{i=1}^{n+1}$, and hence implicitly, as continuous functions of the original microstructure $(E, \nu) \in \cM_{E_{\min}, E_{\max}}^B \times \cM_{\nu_{\min}, \nu_{\max}}^B$. Since the diagonal of $A$ is nonnegative, the dynamics of~\cref{eq:pc_model} are stable. In fact, by \cref{lemma:pc_bound} the diagonal entries of the matrix $A$ and the entries of the vector $b$ are bounded by
\begin{equation}\label{eq:Ab_bound}
    \frac{E_{\min}}{\nu_{\max}} \leq \text{diag}(A) \leq \frac{E_{\max}}{\nu_{\min}}, \quad 0 \leq b \leq \nu_{\max}\big(\frac{E_{\max}}{\nu_{\min}} - \frac{E_{\min}}{\nu_{\max}}\big)^2.
\end{equation}

Note that if $E, \nu$ are piecewise-constant materials with $n+1$ equisized pieces, in other words if $E = E_{\pc}$ and $\nu = \nu_{\pc}$, then the piecewise constant and true constitutive laws agree exactly
\begin{align*}
    \Psi^{\pc}(\epsbar; E_{\pc}, \nu_{\pc}) = \Psi_0(\epsbar; E_{\pc}, \nu_{\pc})
\end{align*}
as a result of \cref{thm:pc_restate,lemma:pc_cont}. For a general microstructure however, $\Psi^{\pc}$ will only serve as an approximation to the true constitutive law $\Psi_0$ of the material.

The basic idea behind using $\Psi^{\pc}$ to approximate the constitutive law
of any material is as follows. The homogenized constitutive model for any reasonable choice of material properties $E,\nu$ can be approximated by the constitutive model arising from making a piecewise-constant approximation $E_{\pc}, \nu_{\pc}$ of the material properties. In fact, we show that $\Psi^{\pc}$ approximates $\Psi_0$ \textit{uniformly} over all materials $E, \nu \in \cM_{E_{\min}, E_{\max}}^B \times \cM_{\nu_{\min}, \nu_{\max}}^B$ to arbitrary accuracy $\e$, for choice of $n$ sufficiently large.

\begin{proposition}\label{prop:pc_approx}
    For any material parameters $E, \nu$ and average strain input $\epsbar$ satisfying \cref{assump:E_nu_eps} and any $\e > 0$, there exists a differential equation strain-to-stress map $\Psi^{\pc}$ defined by \cref{eqn:Psipc} with dimension $n = n(\e)$ such that it uniformly approximates the true strain-to-stress map $\Psi_0$ defined by \cref{eq:cell_problem} of the homogenized material to accuracy $\e$, in the following sense:
    \begin{equation}
        \sup_{\substack{E \in \cM_{E_{\min}, E_{\max}}^B\\ \nu \in \cM_{\nu_{\min}, \nu_{\max}}^B}}\sup_{\epsbar \in \cC_{\epsbar_{\max}, \dot{\epsbar}_{\max}}} \|\Psi^{\pc}(\epsbar; E, \nu) - \Psi_0(\epsbar; E, \nu)\|_{L^\infty} < \e.
    \end{equation}
\end{proposition}
\begin{proof}
    We begin by taking any $E \in \cM_{E_{\min}, E_{\max}}^B$, $\nu \in \cM_{\nu_{\min}, \nu_{\max}}^B$ and $\epsbar \in \cC_{\epsbar_{\max}, \dot{\epsbar}_{\max}}$. Assume that the differential equation model $\Psi^{\pc}$ has hidden state dimension $n = n(\e)$ where we will specify the dependence of $n$ on $\e$ below. Because $\Psi^{\pc}(\epsbar; E_{\pc}, \nu_{\pc}) = \Psi_0(\epsbar; E_{\pc}, \nu_{\pc})$, using \cref{cor:Psi0_lip}, we now have that
    \begin{equation}\label{eq:sigma_bound}
        \|\Psi^{\pc}(\epsbar; E_{\pc}, \nu_{\pc}) - \Psi_0(\epsbar; E, \nu)\|_{L^\infty} \leq C_1\|\nu_{\pc} - \nu\|_{L^2} + C_2\|E_{\pc} - E\|_{L^2}
    \end{equation}
    where $C_1, C_2 > 0$ are constants that depend on $E_{\min}, E_{\max}, \nu_{\min}, \nu_{\max}$ and $\epsbar_{\max}, \dot{\epsbar}_{\max}$. Note that $E, E_{\pc}$ and $\nu, \nu_{\pc}$ are bounded so in particular
    \begin{align*}
        |\nu_{\pc}(x) - \nu(x)| \leq \nu_{\max} - \nu_{\min}, \quad |E_{\pc}(x) - E(x)| \leq E_{\max} - E_{\min}
    \end{align*}
    for almost every $x \in \Omega$. This implies, by $L^1$-$L^\infty$ interpolation of $L^2$, that
    \begin{equation}\label{eq:L2_L1}
    \begin{aligned}
        \|\nu_{\pc} - \nu\|_{L^2} &\leq (\nu_{\max} - \nu_{\min})^\frac{1}{2}\|\nu_{\pc} - \nu\|_{L^1}^\frac{1}{2},\\
        \|E_{\pc} - E\|_{L^2} &\leq (E_{\max} - E_{\min})^\frac{1}{2}\|E_{\pc} - E\|_{L^1}^\frac{1}{2}.
    \end{aligned}
    \end{equation}
    Combining~\cref{eq:sigma_bound} with~\cref{eq:L2_L1} gives us
    \begin{align*}
        &\|\Psi^{\pc}(\epsbar; E_{\pc}, \nu_{\pc}) - \Psi_0(\epsbar; E, \nu)\|_{L^\infty}\\
        &\quad\leq C_1(\nu_{\max} - \nu_{\min})^\frac{1}{2}\|\nu_{\pc} - \nu\|_{L^1}^\frac{1}{2} + C_2(E_{\max} - E_{\min})^\frac{1}{2}\|E_{\pc} - E\|_{L^1}^\frac{1}{2}.
    \end{align*}
    
    As proven in \cref{lem:bv_l1_approx}, piecewise constant functions with $n+1$ pieces can uniformly approximate functions of bounded variation with total variation at most $B$ with error
    \begin{align*}
        \|\nu_{\pc} - \nu\|_{L^1} &\leq \frac{B}{L+1} < \frac{1}{\nu_{\max} - \nu_{\min}}\Big(\frac{\e}{2C_1}\Big)^2, \\
        \|E_{\pc} - E\|_{L^1} &\leq \frac{B}{L+1} < \frac{1}{E_{\max} - E_{\min}}\Big(\frac{\e}{2C_2}\Big)^2
    \end{align*}
    assuming we set $n = n(\e) = C'\frac{B}{\e^2} - 1$. Here $C'$ is a constant that depends on $E_{\min}, E_{\max}, \nu_{\min}, \nu_{\max}$ and $\epsbar_{\max}, \dot{\epsbar}_{\max}$. Combining these results together gives us that
    \begin{equation}
        \|\Psi^{\pc}(\epsbar; E, \nu) - \Psi_0(\epsbar; E, \nu)\|_{L^\infty} < \e.
    \end{equation}
    Because we proved this statement for any $E \in \cM_{E_{\min}, E_{\max}}^B$, $\nu \in \cM_{\nu_{\min}, \nu_{\max}}^B$ and $\epsbar \in \cC_{\epsbar_{\max}, \dot{\epsbar}_{\max}}$, our bound holds uniformly over this class.
\end{proof}

The fact that $(E, \nu) \in \cM_{E_{\min}, E_{\max}}^B \times \cM_{\nu_{\min}, \nu_{\max}}^B$ are functions of bounded variation $B$ is necessary to show that our approximation guarantees hold uniformly over this class of functions as shown in the proof of \cref{prop:pc_approx}. In the next section, we will show that the differential equation right-hand sides in~\cref{eqn:FGpc} are well approximated by Fourer neural mapping architectures. This relies on approximation theory results requiring that the set of material parameters $E, \nu$ entering the theorem statement is compact in $L^2(\Omega)$; this is ensured by working in a subset of $L^2(\Omega)$ in which we have bounded variation.

\begin{remark}
    The Lipschitz bounds above are uniform in length of time-interval $|\cT| = T$. This is due to the stability of the dynamics of the cell problem~\cref{eq:cell_problem}. However, approximation of the dynamics of~\cref{eq:pc_model} by our proposed FNM-RNO model from \eqref{def:FNMRNO}, which we study in the next subsection, will lead to error constants that grow with $T$. It is likely that such results can be improved, by establishing stability properties of the FNM-RNO itself, but doing so is outside the scope of the present paper. The numerical simulations shown in \cref{sec:N} confirm that our RNO models are indeed stable.
    \hspace*{\fill}$\blacklozenge$
\end{remark}

\subsection{Approximation Through Neural Operators}\label{subsec:neuralop}
In this section, we combine the results from the previous section on piecewise-constant approximation, with approximation guarantees of Fourier neural operators, to prove that the homogenized constitutive law $\Psi_0$ of a multiscale KV material from~\cref{eq:cell_problem} is well approximated by an FNM--RNO model $\Psi^{\FNMRNO}$.

Recall from \cref{def:FNMRNO} that an FNM--RNO model
\begin{gather*}
    \Psi^{\FNMRNO}: C^1(\cT; \R) \times L^2(\T; \R^2) \to C(\cT; \R),\\
    (\{\epsbar(t)\}_{t \in \cT}; E, \nu) \mapsto \{\sigmabar_{\RNO}(t)\}_{t \in \cT}
\end{gather*}
with a one-dimensional Kelvin--Voigt material microstructure $M = (E, \nu) \in L^2(\T; \R^2)$ is given by
\begin{equation}\label{eq:FNMRNO_model}
\begin{aligned}
\sigmabar_{\RNO}(t) &= F_{\FNM}(\epsbar(t), \dot{\epsbar}(t), \xi_{\RNO}(t); E, \nu)\\
\dot{\xi}_{\RNO}(t) &= G_{\FNM}(\epsbar(t), \xi_{\RNO}(t); E, \nu),\quad \xi_{\RNO}(0) = 0,
\end{aligned}
\end{equation}
where $F_{\FNM}$ and $G_{\FNM}$ are two FNMs as defined in \cref{def:FNM}.

Note that the differential equation model $\Psi^{\pc}$ \cref{eq:pc_model} is of the same form as the FNM--RNO model $\Psi^{\FNMRNO}$ \cref{eq:FNMRNO_model}, and the right-hand sides $F_{\pc}, G_{\pc}$ of this ODE model can be approximated by $F_{\FNM}, G_{\FNM}$ if designed appropriately. To show this formally, we use the universal approximation theorem for Fourier neural mappings, an important result in neural operator theory~\cite{kovachki2021universal, kovachki2023neural, huang2024operator}. We rely on the compact embedding of the space of bounded $\BV$ functions in $L^2$, used in prior works on approximation theory of FNMs for elliptic homogenzation problems~\cite[Lemma C.1]{bhattacharya2024learning}. We combine these results here to obtain the universal approximation result for our FNM architecture, proven in detail in \cref{prop:UA-hom}.

\begin{proposition}\label{prop:UA-hom_restate}
    Under \cref{assump:E_nu_eps}, choosing any $\epsbar_{\max}, \dot{\epsbar}_{\max}, \xi_{\max} > 0$ and $\e_F, \e_G > 0$, there exist FNMs $F_{\FNM}$ and $G_{\FNM}$, in 
    \cref{eqn:PsiRNO}, that approximate $F_{\pc}$ and $G_{\pc}$ of \cref{eqn:FGpc} such that 
    \begin{gather}
        \sup_{\substack{|z_1| \leq \epsbar_{\max}, |z_2| \leq \dot{\epsbar}_{\max}, \|z_3\|\leq \xi_{\max},\\ E \in \cM_{E_{\min}, E_{\max}}^B, \nu \in \cM_{\nu_{\min}, \nu_{\max}}^B}} |F_{\FNM}(z_1, z_2, z_3; E, \nu) - F_{\pc}(z_1, z_2, z_3; E, \nu)| < \e_F\\
        \sup_{\substack{|z_1| \leq \epsbar_{\max}, \|z_3\|\leq \xi_{\max},\\ E \in \cM_{E_{\min}, E_{\max}}^B, \nu \in \cM_{\nu_{\min}, \nu_{\max}}^B}} \|G_{\FNM}(z_1, z_3; E, \nu) - G_{\pc}(z_1, z_3; E, \nu)\| < \e_G.
    \end{gather}
\end{proposition}

Now we are ready to prove that for a general class of material microstructures $(E, \nu)$ and strain histories $\epsbar$, the true constitutive law $\Psi_0$ from~\cref{eq:cell_problem} can be approximated by $\Psi^{\pc}$ which can in turn be approximated by $\Psi^{\FNMRNO}$.

\begin{theorem}[FNM-RNO Universal Approximation]\label{thm:main_ua_theorem}
    For any materials $E, \nu$ and strain $\epsbar$ satisfying \cref{assump:E_nu_eps} and any $\e > 0$, there exist FNMs $F_{\FNM}$ and $G_{\FNM}$ such that the map
    \begin{align*}
        \Psi^{\FNMRNO}: \left(\{\epsbar(t)\}_{t \in \cT}; E, \nu\right) \mapsto \{\sigmabar_{\RNO}(t)\}_{t \in \cT}
    \end{align*}
    defined by \cref{eqn:PsiRNO} uniformly approximates the map
    \begin{align*}
        \Psi_0: \left(\{\epsbar(t)\}_{t\in\cT}; E, \nu\right) \mapsto \{\sigmabar(t)\}_{t \in \cT}
    \end{align*}
    defined by \cref{eq:cell_problem} to accuracy $\e$, in the following sense:
    \begin{equation}
        \sup_{\substack{E \in \cM_{E_{\min}, E_{\max}}^B\\ \nu \in \cM_{\nu_{\min}, \nu_{\max}}^B}}\sup_{\epsbar \in \cC_{\epsbar_{\max}, \dot{\epsbar}_{\max}}} \|\Psi^{\FNMRNO}(\epsbar; E, \nu) - \Psi_0(\epsbar; E, \nu)\|_{L^\infty(\cT)} < \e.
    \end{equation}
\end{theorem}
\begin{proof}
    Assume throughout this proof that $E \in \cM_{E_{\min}, E_{\max}}^B$, $\nu \in \cM_{\nu_{\min}, \nu_{\max}}^B$ and $\epsbar \in \cC_{\epsbar_{\max}, \dot{\epsbar}_{\max}}$. The bounds derived below hold uniformly over all functions in these classes. Denote the true strain-to-stress map defined by the cell problem~\cref{eq:cell_problem} as $\Psi_0: (\{\epsbar(t)\}_{t\in\cT}; E, \nu) \mapsto \{\sigmabar(t)\}_{t \in \cT}$. We know by \cref{prop:pc_approx} that
    \begin{equation}\label{eq:psi_pc_dagger_diff}
        \|\Psi^{\pc}(\epsbar; E, \nu) - \Psi_0(\epsbar; E, \nu)\|_{L^\infty} = \|\sigmabar_{\pc} - \sigmabar\|_{L^\infty} < \frac{\e}{2}
    \end{equation}
    as long as the dimension $n = n(\e/2)$ in~\cref{eq:pc_model} is taken sufficiently large.
    
    Now we study the error between $\Psi^{\FNMRNO}$ and $\Psi^{\pc}$. To do this, we first need to show that the trajectories of the hidden variables $\xi_{\pc}$ in~\cref{eq:pc_model} and $\xi_{\RNO}$ in~\cref{eq:FNMRNO_model} stay in a bounded domain so that we can apply well-established FNM universal approximation results. First studying $\xi_{\pc}$ note that
    \begin{align*}
        \dot{\xi}_{\pc}(t) \leq - A(E, \nu)\xi_{\pc}(t) + b(E, \nu)\epsbar_{\max}.
    \end{align*}
    Since $A$ is a diagonal matrix with strictly positive entries, using~\cref{eq:Ab_bound} we can apply Gronwall's inequality to each entry of $\xi_{\pc}$, using that $\xi_{\pc}(0)=0$, to write
    \begin{align*}
        \xi_{\pc}(t) \leq A(E, \nu)^{-1}b(E, \nu)\epsbar_{\max} \leq \epsbar_{\max}\frac{\nu_{\max}^2}{E_{\min}}\Big(\frac{E_{\max}}{\nu_{\min}} - \frac{E_{\min}}{\nu_{\max}}\Big)^2
    \end{align*}
    where the bound above is interpreted element-wise. In the preceding inequality, we have used the lower and upper bounds, derived in \cref{prop:pc_approx}, on $A$ and $b$ respectively. We can derive the same bound for $-\xi_{\pc}$ through a similar application of Gronwall's inequality which proves that
    \begin{equation}\label{eq:xi_pc_bound}
        \sup_{t \in \cT}\|\xi_{\pc}(t)\| \leq \sqrt{L}\epsbar_{\max}\frac{\nu_{\max}^2}{E_{\min}}\Big(\frac{E_{\max}}{\nu_{\min}} - \frac{E_{\min}}{\nu_{\max}}\Big)^2.
    \end{equation}
    The next step is to bound the difference between $\xi_{\RNO}$ and $\xi_{\pc}$. Because the trajectory of $\xi_{\pc}$ is bounded, by FNM universal approximation result~\cref{prop:UA-hom_restate}, there exists a FNM $G_{\FNM}$ such that
    \begin{align*}
        \|G_{\pc}(\epsbar, \xi_{\pc}; E, \nu) - G_{\FNM}(\epsbar, \xi_{\pc}; E, \nu)\|_2 < \e_G
    \end{align*}
    for any small $\e_G > 0$. 
    Next, we apply the triangle inequality
    \begin{align*}
        \|\dot{\xi}_{\pc} - \dot{\xi}_{\RNO}\|_2 &\leq \|G_{\pc}(\epsbar, \xi_{\pc}; E, \nu) - G_{\FNM}(\epsbar, \xi_{\pc}; E, \nu)\|_2\\
        &\quad\quad\quad +  \|G_{\FNM}(\epsbar, \xi_{\pc}; E, \nu)-G_{\FNM}(\epsbar, \xi_{\RNO}; E, \nu)\|_2\\
        &< \e_{G} + L_G\|\xi_{\pc} - \xi_{\RNO}\|_2
    \end{align*}
    where $L_G$ is the Lipschitz constant of the Fourier neural mapping $G_{\FNM}$ in the variable $\xi$. We explicitly derive the form of this Lipschitz constant in \cref{lem:FNO_lip}. Now note that
    \begin{equation}
        \frac{\d}{\dt}\|\xi_{\pc} - \xi_{\RNO}\|_2 = \Big\langle\dot{\xi}_{\pc} - \dot{\xi}_{\RNO}, \frac{\xi_{\pc} - \xi_{\RNO}}{\|\xi_{\pc} - \xi_{\RNO}\|_2}\Big\rangle_2 \leq \|\dot{\xi}_{\pc} - \dot{\xi}_{\RNO}\|_2
    \end{equation}
    and hence, we have that
    \begin{align*}
        \frac{\d}{\dt}\|\xi_{\pc} - \xi_{\RNO}\|_2 < \e_{G} + L_G\|\xi_{\pc} - \xi_{\RNO}\|_2.
    \end{align*}
    By Gronwall's inequality, we get that
    \begin{equation}\label{eq:xi_pc_RNO_diff}
        \|\xi_{\pc}(t) - \xi_{\RNO}(t)\|_2 < \frac{\e_G}{L_G}e^{L_Gt} \leq \frac{\e_G}{L_G}e^{L_GT}
    \end{equation}
    assuming that $t \in \cT = [0, T]$. Finally, combining~\cref{eq:xi_pc_bound} with~\cref{eq:xi_pc_RNO_diff} we get that
    \begin{equation}
        \sup_{t \in \cT}\|\xi_{\pc}(t)\|_2, \sup_{t \in \cT}\|\xi_{\RNO}(t)\|_2 < \frac{\e_G}{L_G}e^{L_GT} + \sqrt{L}\epsbar_{\max}\frac{\nu_{\max}^2}{E_{\min}}\Big(\frac{E_{\max}}{\nu_{\min}} - \frac{E_{\min}}{\nu_{\max}}\Big)^2.
    \end{equation}
    Lastly, again invoking FNM universal approximation \cref{prop:UA-hom_restate} there exists an FNM $F_{\FNM}$ such that 
    \begin{align*}
        \|F_{\FNM}(\epsbar, \dot{\epsbar}, \xi_{\RNO}; E, \nu) - F_{\pc}(\epsbar, \dot{\epsbar}, \xi_{\RNO}; E, \nu)\|_2 < \e_F
    \end{align*}
    for any small $\e_F > 0$. Because $F_{\pc}$ is linear in $\xi$, we further have by Cauchy-Schwarz that
    \begin{align*}
        \|F_{\pc}(\epsbar, \dot{\epsbar}, \xi_{\RNO}; E, \nu) - F_{\pc}(\epsbar, \dot{\epsbar}, \xi_{\pc}; E, \nu)\|_2 &= \|\ip{\mathds{1}}{\xi_{\RNO} - \xi_{\pc}}\|_2 \\
        &\leq \sqrt{L}\|\xi_{\RNO} - \xi_{\pc}\|_2 \leq \sqrt{L}\frac{\e_G}{L_G}e^{L_GT}
    \end{align*}
    where in the last inequality we used the bound derived in~\cref{eq:xi_pc_RNO_diff}. By the triangle inequality, we can write
    \begin{align*}
        |\sigmabar_{\RNO}&(t) - \sigmabar_{\pc}(t)|\\
        &\leq \|F_{\FNM}(\epsbar, \dot{\epsbar}, \xi_{\RNO}; E, \nu) - F_{\pc}(\epsbar, \dot{\epsbar}, \xi_{\RNO}; E, \nu)\|_2 \\
        &\qquad+ \|F_{\pc}(\epsbar, \dot{\epsbar}, \xi_{\RNO}; E, \nu) - F_{\pc}(\epsbar, \dot{\epsbar}, \xi_{\pc}; E, \nu)\|_2\\
        &\leq \e_F + \sqrt{L}\frac{\e_G}{L_G}e^{L_GT}.
    \end{align*}
    which proves the bound
    \begin{equation}\label{eq:psi_RNO_pc_diff}
        \|\Psi^{\FNMRNO}(\epsbar; E, \nu) - \Psi^{\pc}(\epsbar; E, \nu)\|_{L^\infty} = \|\sigmabar_{\RNO} - \sigmabar_{\pc}\|_{L^\infty} < \frac{\e}{2}
    \end{equation}
    by choosing $\e_F, \e_G$ sufficiently small. Finally, by combining~\cref{eq:psi_pc_dagger_diff} and~\cref{eq:psi_RNO_pc_diff} through a triangle inequality we get the desired bound
    \begin{equation}
        \|\Psi^{\FNMRNO}(\epsbar; E, \nu) - \Psi_0(\epsbar; E, \nu)\|_{L^\infty} < \e.
    \end{equation}
\end{proof}
This proves the main theoretical result of our paper, namely that the homogenized constitutive law of the one-dimensional Kelvin--Voigt model can be approximated by an FNM--RNO architecture uniformly over a large class of material microstructures and strain inputs. The theorem justifies the consideration of the FNM--RNO more generally, beyond the specifics of linear Kelvin--Voigt viscoelasticity, an avenue we pursue further in the next section on numerical experiments.

\section{Numerical Experiments}\label{sec:N}
In this section, we apply our proposed recurrent neural operator architecture to learn
and deploy homogenized constitutive laws of viscoelastic and elasto-viscoplastic materials. We initially consider linear viscoelasticity with piecewise-constant microstructures, with varying numbers of pieces, and then design and study high-memory continuous microstructures. We first discuss, in \cref{subsec:data_set}, our data generation procedure for sampling these microstructures and for our choice of strain trajectories used to force the cell problem at the boundary. In \cref{subsec:architecture}, we give further details of our FNM--RNO architecture and model training. We then demonstrate that the ability of our neural operator to encode memory in the strain-to-stress relationship allows us to improve significantly over memoryless models. Our numerical results are shown in \cref{subsec:numerical}, where our architecture is tested on the multiscale Kelvin--Voigt cell problem and is then used within homogenized macroscale simulation. In \cref{subsec:plastic}, we show that the same model can be used to learn the constitutive law of elasto-viscoplastic materials. Taken together, the experiments demonstrate that our approach applies to different constitutive models and generalizes across a wide array of material microstructures and strain inputs.

\subsection{Dataset Generation}\label{subsec:data_set}
The dataset for our FNM--RNO architecture consists of $n_s$ material microstructures, averaged strain trajectories, and averaged stress trajectories $\{E^{(j)}, \nu^{(j)}, \epsbar^{(j)}, \sigmabar^{(j)}\}_{j=1}^{n_s}$. We consider two different ways of producing joint samples of $E$ and $\nu$ that lead to piecewise-constant (PC) random materials and high-memory continuous (HMC) random materials. The strain trajectories $\epsbar^{(j)}$ are independently sampled following the procedure in~\cite{bhattacharya2023learning, liu2023learning, zhang2024iterated} which is detailed below. Given these samples, we solve the cell problem in \cref{eq:cell_problem} for the averaged stress $\sigmabar^{(j)}$ on a uniform grid using linear Lagrange finite elements with $501$ spatial degrees of freedom (DoFs) and a 4th order explicit Runge--Kutta method with $5,001$ temporal DoFs.

We first detail the sampling procedures for piecewise-constant and high-memory continuous microstructures and then describe the construction of the average strains; taken together, these define the data sets used later for training and testing. Further
testing of generalization with respect to strain trajectories is implicit in the macroscale calculations, also presented later, since these generate strains that are not in our training set.~\\

\paragraph{\bf Piecewise-Constant Random Materials} The piecewise constant random functions ($E$, $\nu$) are generated to be spatially periodic and share the same set of jump discontinuities. The number of constant pieces $L$ is selected uniformly at random from $5$--$20$. Locations of the discontinuities are drawn at random from the finite set $\{0.02k\}_{k=0}^{50}$, with equal probability and with replacement; this leads to a minimum length of $0.02$ for each piece. The values of $E$ and $\nu$ in each piece are sampled from a uniform distribution on $[0.1, 1]^2$.~\\

\paragraph{\bf High-Memory Continuous Random Materials} We construct High-Memory Continuous (HMC) materials ($E$, $\nu$) to again be spatially-periodic. They are designed by taking samples from a periodic random mean shift $m \in C_{\text{per}}(\Omega;\R^2)$ and periodic centered Gaussian random function $g \in C_{\text{per}}(\Omega;\R^2)$,
\begin{equation}
    \begin{bmatrix}E(y)\\ \nu(y)\end{bmatrix} = 0.45\times \Big(\text{erf}\Big(m(y) + g(y)\Big) + 1\Big) + 0.1,
\end{equation}
where $\text{erf}$ is the error function that smoothly enforces $E$, $\nu$ to be bounded in $[0.1, 1]$.

We design $m: \Omega \to \R^2$ using a piecewise-constant material that has large contrasts between the magnitude of its first and second coordinates, hence resulting in a large contrast between the elasticity $E$ and viscosity $\nu$. The choice of this mean function typically corresponds to viscoelastic materials with large memory kernels, as it leads to large exponential weights $\beta_l$, as given 
by~\cref{eq:beta_l}, of the memory kernel derived in \cref{lemma:pc_cont}. Specifically, $m$ is generated by sampling its two piecewise-constant pieces from a Gaussian mixture distribution with two equally-weighted modes centered at $[-1, 1]$ and $[1, -1]$ with a small covariance $0.06 I$. Lastly, $m$ is turned into a continuous function by applying a spatial Gaussian convolution with a standard deviation of $0.01$ to each component.

The random perturbation $g: \Omega \to \R^2$ is sampled from a centered Gaussian distribution with a diagonal covariance matrix with entries $\rho^{(k)}{\sigma^{(k)}}^2(1 - {\rho^{(k)}}^2\partial_x^2)^{-2}$, where $k = 1, 2$ indicates the $E$ or $\nu$ component and $\rho^{(k)}$ and $\sigma^{(k)}$ represent correlation length and pointwise standard deviation. The statistics $\rho^{(k)}, \sigma^{(k)}$ are sampled i.i.d. from a reciprocal distribution on $[0.01, 0.3]$ and a uniform distribution on $[0.1, 0.3]$ respectively.~\\

\paragraph{\bf Averaged Strain Trajectories} The averaged strain trajectories $\epsbar$ are generated by first randomly picking a total number of time points $3\leq n\leq 21$ with $0=t_1<\dots<t_{n}=1$, where the internal time points $\{t_k\}_{k=2}^{n-1}$ are uniformly randomly placed in $[0,1]$. At each time point, we assign its averaged strain value $\epsbar(t_{k})$ by first randomly picking a sign $v_k\in\{-1, 1\}$ and then taking
\begin{equation}
    \epsbar(t_{k}) = \epsbar(t_{k-1}) + (0.5-\epsbar(t_{k-1}))v_k\sqrt{t_{k}-t_{k-1}}
\end{equation}
where we initialize $\epsbar(0) = 0$. We use piecewise-cubic Hermite interpolating polynomials (PCHIP) to create the averaged strain trajectories from these points.

Samples of the piecewise-constant (PC) material dataset and the high-memory continuous (HMC) material dataset are provided in \cref{fig:data_set_visual}. We visualize the averaged stress response with and without memory effects. The stress response without the memory effects is given by \cref{eqn:kernelform} with $K\equiv 0$. In \cref{fig:data_set_visual}, the stress response without memory for PC sample \#2 and HMC sample \#3 show large discrepancies in comparison to the stress response with memory, demonstrating the importance of modeling memory effects for those material and strain trajectory inputs.
\begin{figure}[htbp]
    \centering
    \scalebox{0.84}{
    \addtolength{\tabcolsep}{-6pt}
    \begin{tabular}{E G G}
        & \makecell{\bf Piecewise-constant material\\ \bf dataset (PC)} & \makecell{\bf High-memory continuous material\\\bf dataset (HMC)} \\
        \rotatebox{90}{\bf Sample \#1} & \includegraphics[width=\linewidth]{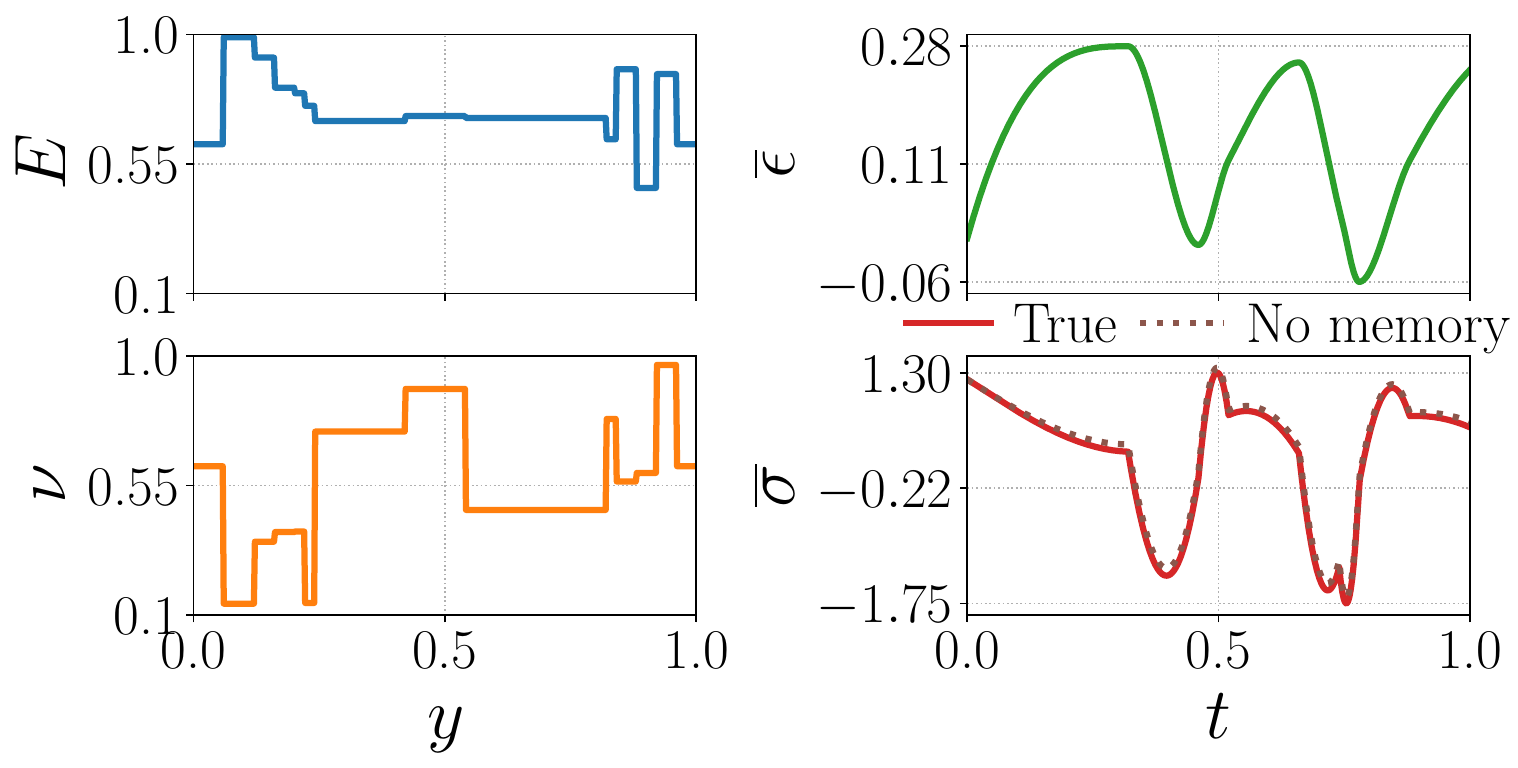}  & \includegraphics[width=\linewidth]{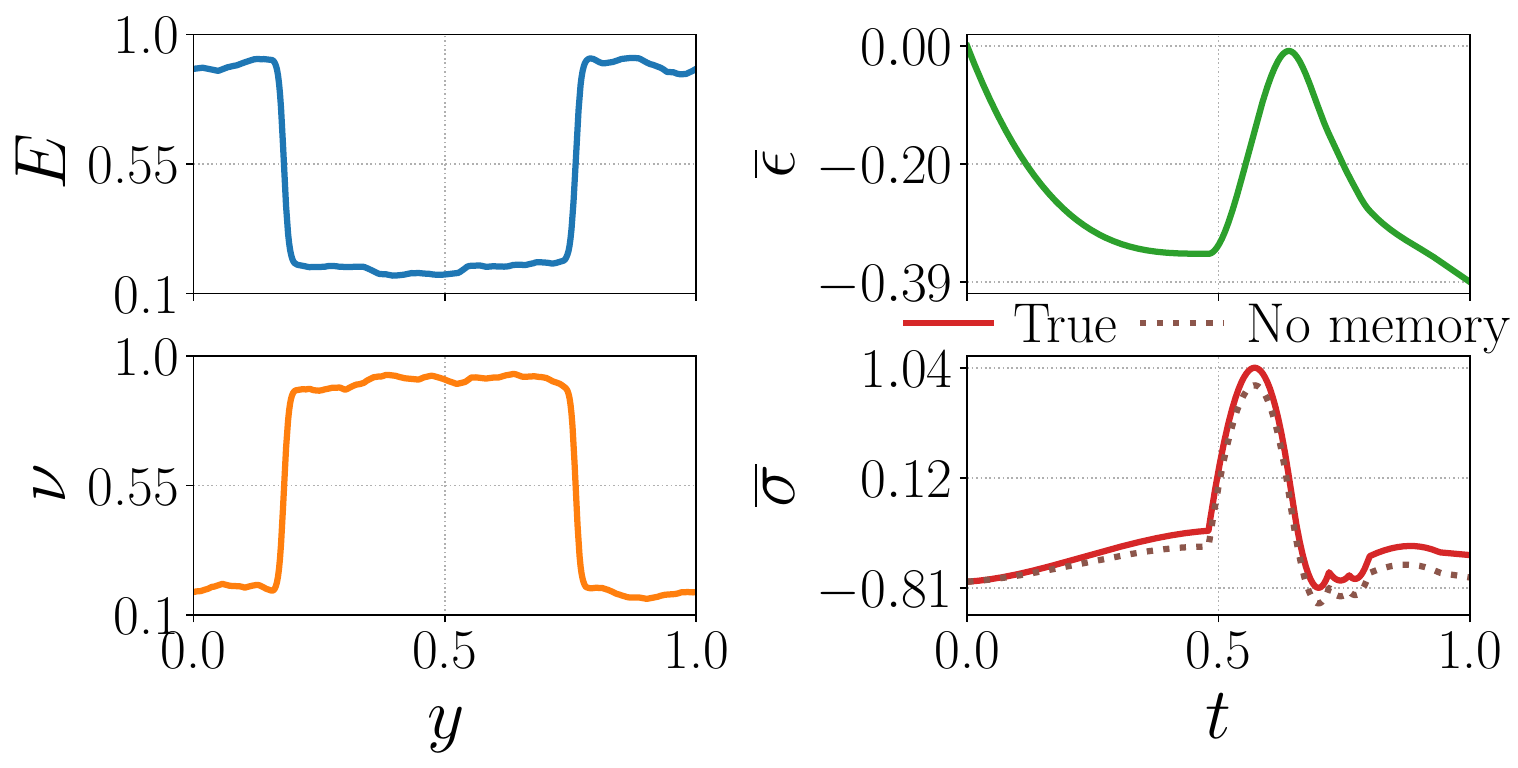} \\
        \rotatebox{90}{\bf Sample \#2}& \includegraphics[width=\linewidth]{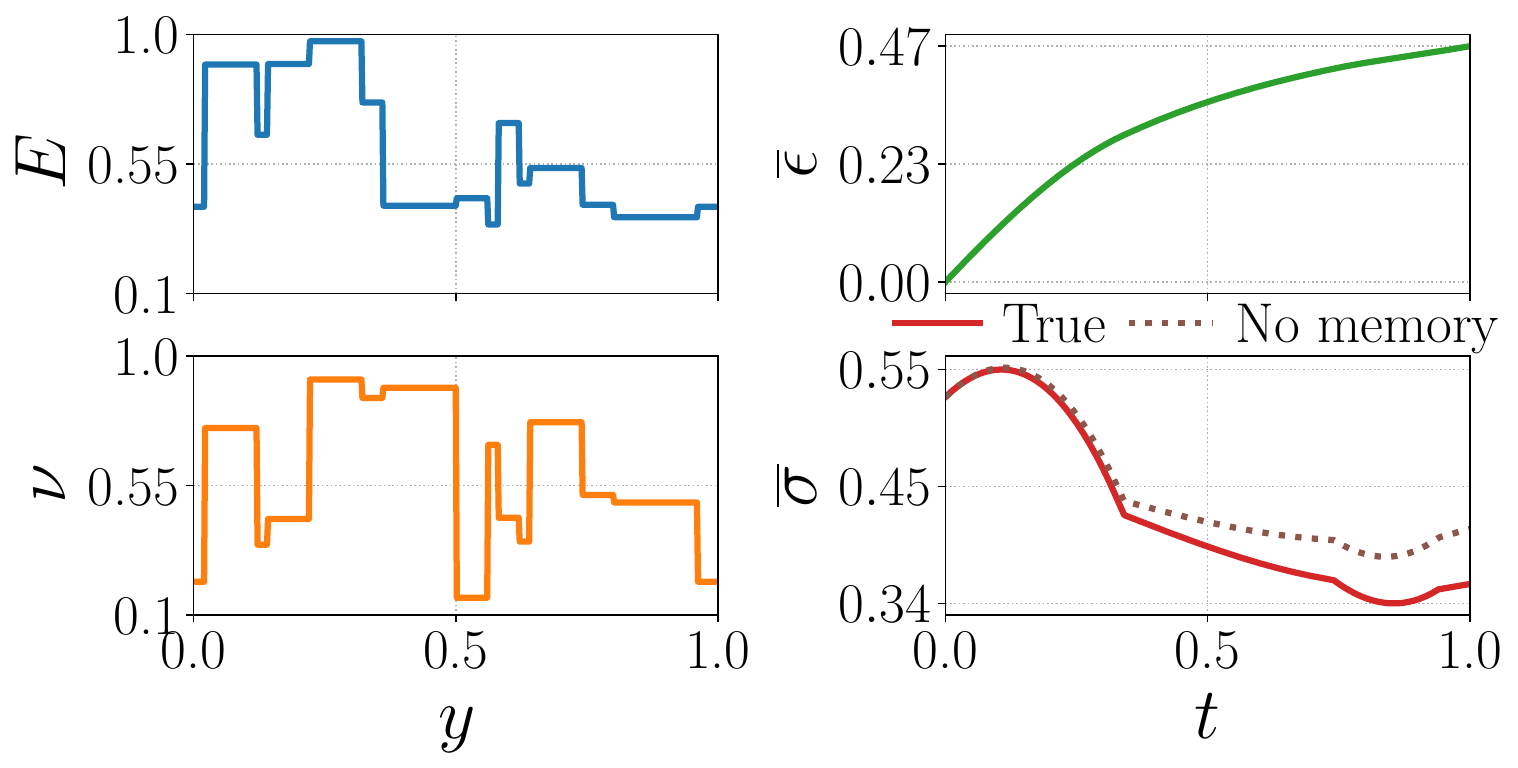}  & \includegraphics[width=\linewidth]{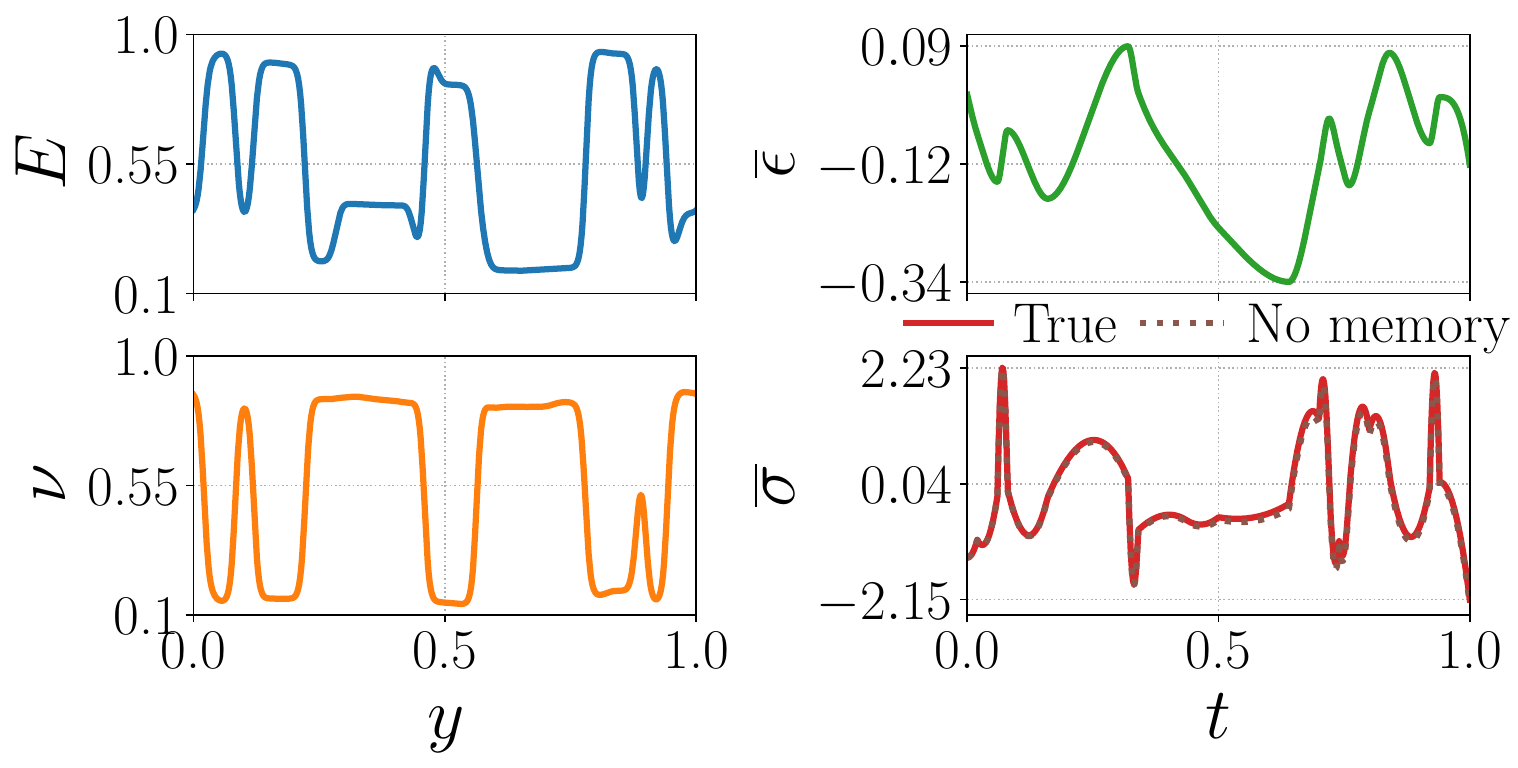} \\
        \rotatebox{90}{\bf Sample \#3} & \includegraphics[width=\linewidth]{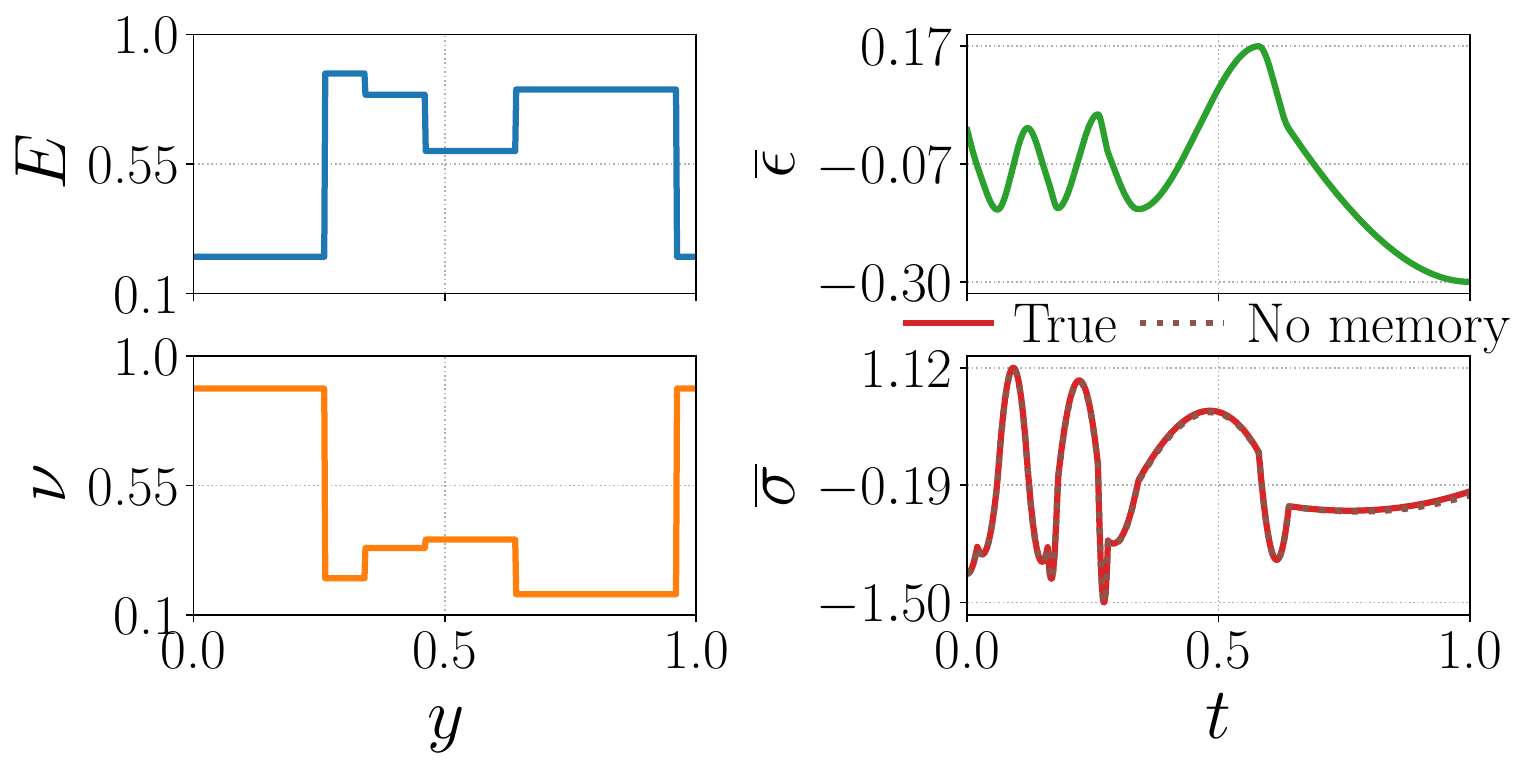}  & \includegraphics[width=\linewidth]{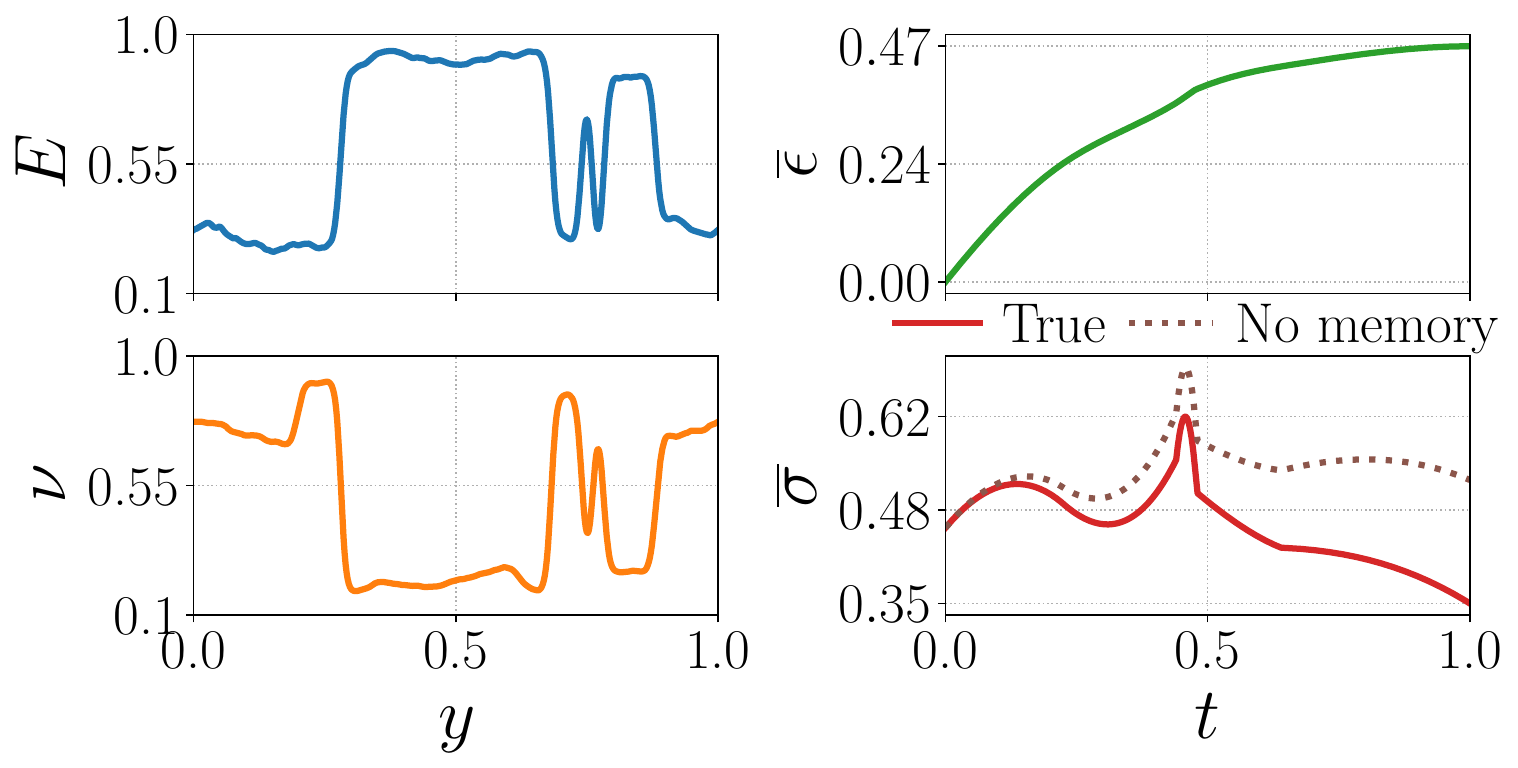} 
    \end{tabular}
    \addtolength{\tabcolsep}{6pt}
    }
    \caption{Visualization of samples from the two datasets: piecewise-constant material (PC) and high-memory continuous material (HMC); see \cref{subsec:data_set}. Each dataset consists of material samples $(E^{(j)}, \nu^{(j)})$, averaged strain trajectory samples $\epsbar^{(j)}$, and the averaged stress trajectory samples $\sigmabar^{(j)}$. We visualize the averaged stress response with (\textit{solid lines}) and without (\textit{dotted lines}) memory effects.}
    \label{fig:data_set_visual}
\end{figure}

\subsection{Architecture and Training}\label{subsec:architecture}
We now describe the architecture of the FNM--RNO from \cref{subsec:neuralop} introduced in \cref{def:FNMRNO} which is used in the following experiments. The internal variable size $L$ of our RNO differential equation is varied through all values in the set $\{1, 3, 5, 10, 15, 20, 25\}$. We simplify the architecture by letting our vector-to-function map $\sD$ and the vector lifting function $S_v$ be the identity. We use spatial coordinates as an additional channel in our function input besides the material microstructures $(E, \nu)$, a form of positional encoding. Hence, the number of functional inputs into our architecture is $d_{\inn}^f = 3$. The function lifting layer leads to hidden channels of size $d_0=32$. We use 3 Fourier layers with the same hidden channel size $d_t=32$ for $t=0,1,2.$ Each layer uses Gaussian error linear unit activations and 4 Fourier modes to parameterize all the convolution operators. The output dimension of the functional layers is $d_{\proj}^{fv} = 64$. 

We consider the loss function given by a squared relative $L^2$ error with a penalty term. Let $\sigmabar^{(j)}_{\text{RNO}}$ denote the FNM--RNO prediction of the averaged stress trajectories for the data sample $E^{(j)}, \nu^{(j)}, \epsbar^{(j)}$ and $\sigmabar^{(j)}$. The prediction depends on the learnable parameters of $F_{\FNM}$ and $G_{\FNM}$, and these parameters are found by minimizing the following loss function:
\begin{equation}\label{eq:loss_function}
\begin{aligned}
    \text{Loss}(\{\sigmabar^{(j)}\}_{j=1}^{n_\text{train}}, \{\sigmabar^{(j)}_{\text{RNO}}\}_{j=1}^{n_\text{train}}) &= \frac{1}{n_{\text{train}}}\sum_{j=1}^{n_\text{train}}\Big(
    \mathcal{E}(\sigmabar^{(j)}, \sigmabar^{(j)}_{\text{RNO}})^2 \\
    &\qquad+ \| G_{\text{FNM}}(0,0;E^{(j)},\nu^{(j)})\|^2\Big),
\end{aligned}
\end{equation}
where $\mathcal{E}$ returns the relative $L^2$ error in the average stress
\begin{equation}
     \mathcal{E}(\sigmabar^{\dagger}, \sigmabar) = \Big(\frac{\int_{0}^{1}\left|\sigmabar^{\dagger}(t)- \sigmabar(t) \right|^2dt}{\int_{0}^ {1}\left|\sigmabar^{\dagger}(t) \right|^2dt}\Big)^{1/2}.
\end{equation}
The penalty term is included in the loss function because we find that the physical constraint $G_{\text{FNM}}(0,0;E^{(j)},\nu^{(j)}) = 0$ is usually not learned without the penalty, which leads to a large error in averaged stress predictions when the material has no deformation history; see \cref{appsec:penalty}. We use $n_{\text{train}}=2,049$ samples from the PC dataset to train 7 FNM--RNOs with varying numbers of the internal variables $L$, with resolutions of the training data reduced to $251$ spatial and $501$ temporal DoFs. The PC dataset is used for training as it leads to better model generalization. We use the HMC dataset for testing only. To evaluate the loss function, we use the forward Euler scheme to estimate $\sigma^{(j)}_{\text{RNO}}$ and the trapezoidal rule to estimate the temporal integration. We use the Adam optimizer with a learning rate of $10^{-3}$, a batch size of 32, total epochs of $500$, and cosine annealing of the learning rate that tends towards $10^{-5}$.

\subsection{Numerical Results}\label{subsec:numerical}

\paragraph{\bf High Generalization Accuracy} We evaluate the trained FNM--RNOs on 2,500 testing samples from both the PC and HMC datasets. In \cref{fig:testing_error_distribution}, we present the distributions of the relative $L^2$ error for the FNM--RNO stress response alongside the error for the linear stress response excluding memory effects. Additionally, we visualize in \cref{fig:testing_visual} the testing samples with the largest and median errors for the FNM--RNO using five internal variables.

The FNM--RNOs achieve consistently low relative $L^2$ testing errors on both datasets, with mean errors of 0.7\%--1.2\% for the PC dataset and 0.9\%--1.9\% for the HMC dataset. On average, the FNM--RNO stress response considerably outperforms the linear stress response without memory effects, which exhibits mean errors of 4.5\% and 7.3\% for the PC and HMC datasets, respectively. For testing samples with large FNM--RNO stress response errors, the stress response without memory typically shows a much greater discrepancy from the true stress response in comparison; see, e.g., \cref{fig:testing_visual}.

Furthermore, the generalization accuracy of FNM--RNO improves with the inclusion of more than one internal variable; however, there is no significant enhancement when including more than three internal variables. We note that the analytical form of the RNO for the PC data set requires at least 20 internal variables due to \cref{thm:pc_restate}, which is not discovered using the FNM--RNO architecture, likely because the training is harder with more internal variables.~\\

\begin{figure}[htb]
    \centering
    \scalebox{0.9}{
    \begin{tabular}[t]{c c}
        \hspace{0.09\linewidth}\makecell{\bf Relative $L^2$ testing error\\\bf
        on the PC dataset} & \hspace{0.09\linewidth}\makecell{\bf Relative $L^2$ testing error\\\bf
        on the HMC dataset}\\
        \includegraphics[width=0.5\linewidth]{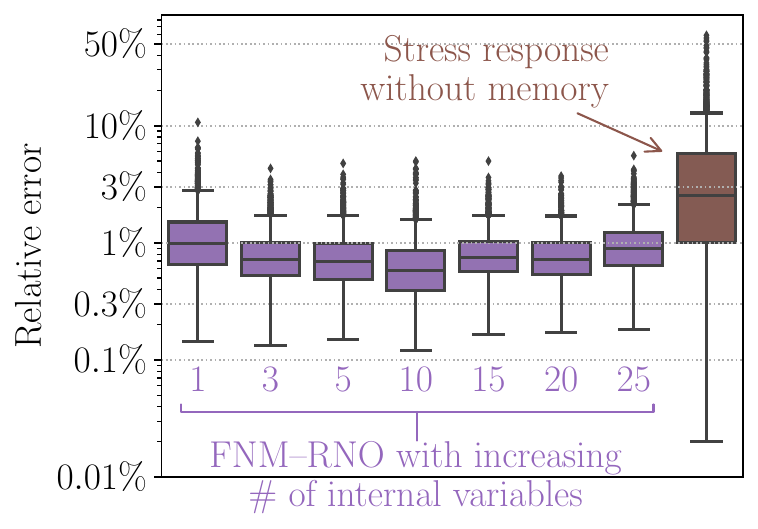} & \includegraphics[width=0.5\linewidth]{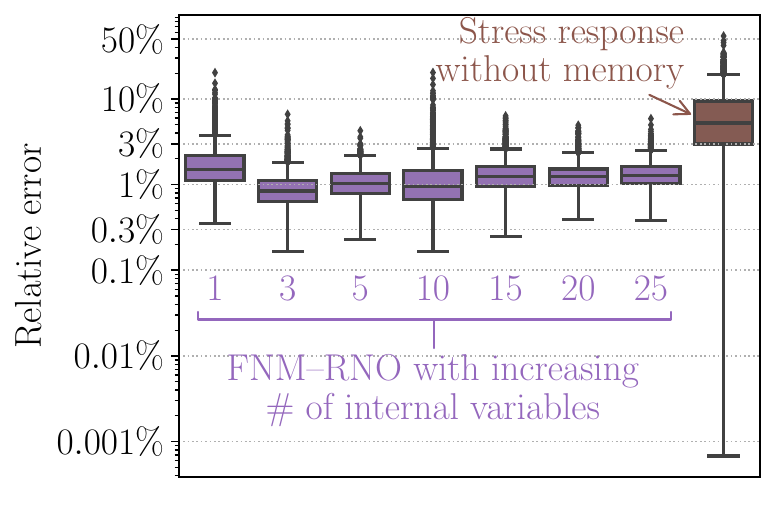}
    \end{tabular}
    }
    \caption{The distributions of the relative $L^2$ error on 2,500 testing samples from the PC dataset (\textit{left}) and the HMC dataset (\textit{right}). We visualize the errors in FNM--RNOs predictions where the trained FNM--RNOs have a varying number of internal variables. We also visualize the distribution of error given by the linear stress response without memory effects, where the response function is obtained using \cref{eqn:kernelform} with $K\equiv0$.} 
    \label{fig:testing_error_distribution}
\end{figure}

\begin{figure}[htbp]
    \centering
    \scalebox{0.85}{
    \addtolength{\tabcolsep}{-6pt}
    \begin{tabular}{E F F F F}
    & \multicolumn{2}{c}{\makecell{\bf Piecewise-constant material\\ \bf dataset (PC)}} & \multicolumn{2}{c}{\makecell{\bf High-memory continuous material\\\bf dataset (HMC)}} \\
     \multirow{2}{*}[2em]{\rotatebox{90}{\makecell{\bf Largest testing error}}}   & \includegraphics[width=\linewidth]{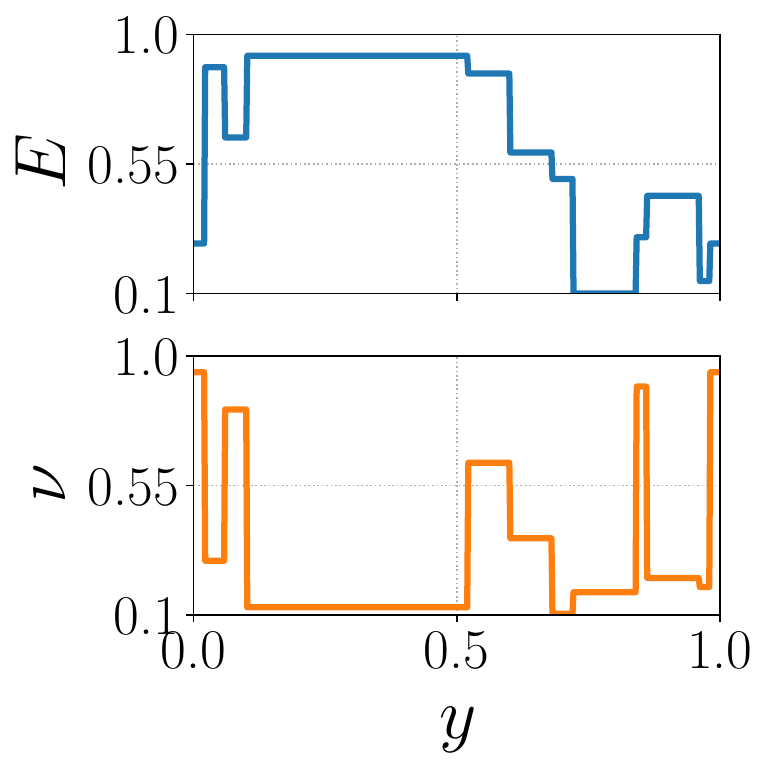} & \multirow{2}{*}[4.5em]{\includegraphics[width=\linewidth]{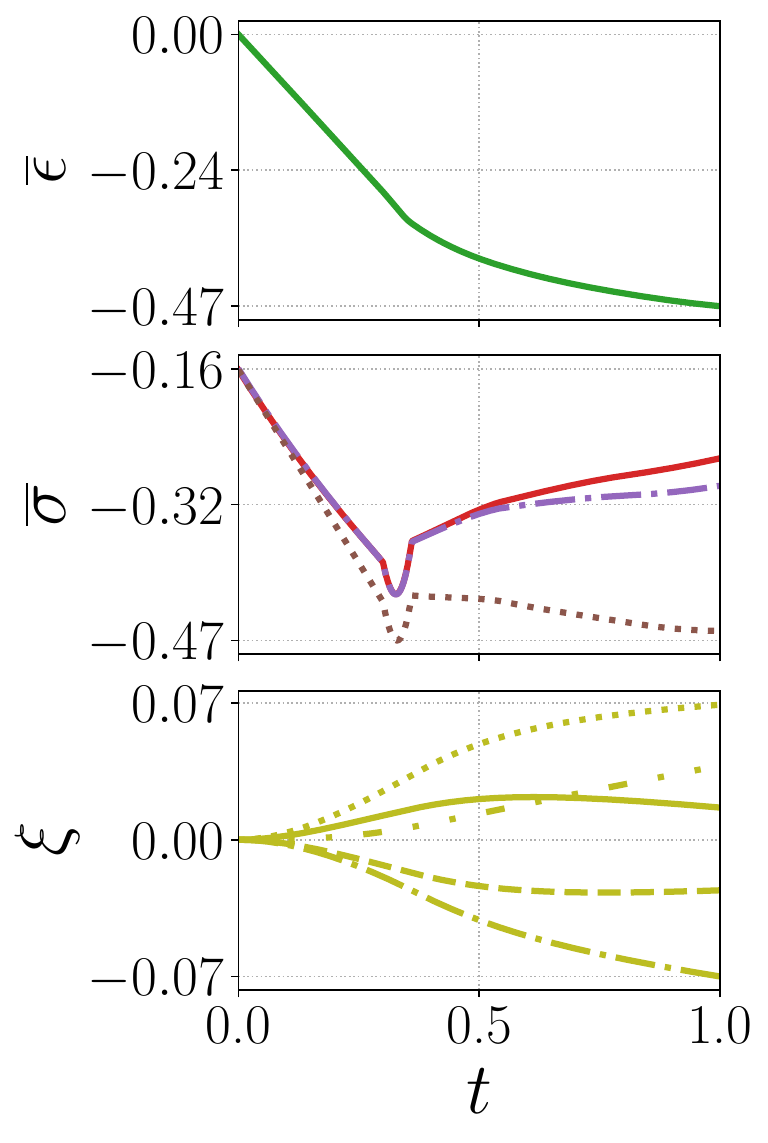}} & \includegraphics[width=\linewidth]{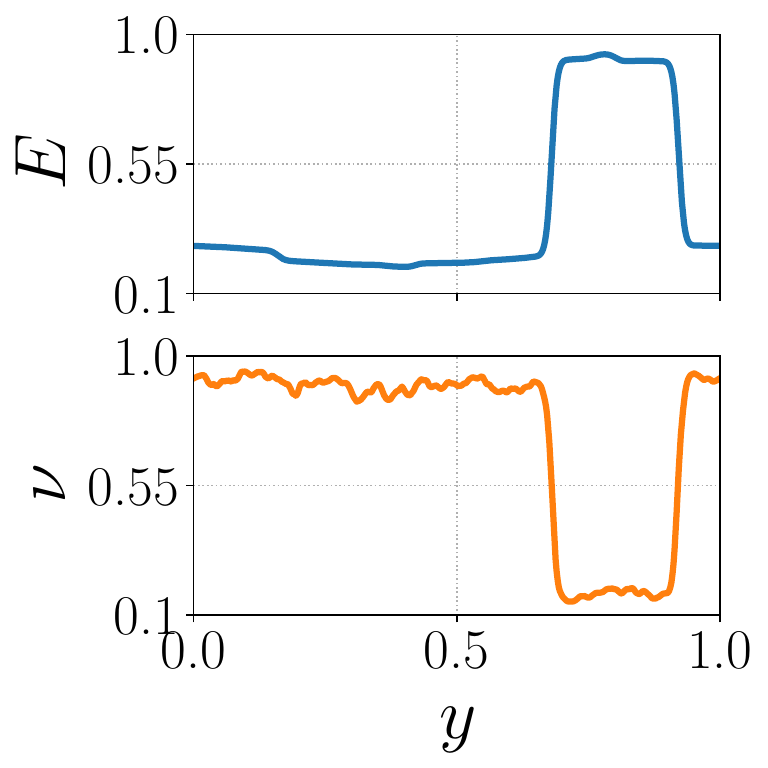} & \multirow{2}{*}[4.5em]{\includegraphics[width=\linewidth]{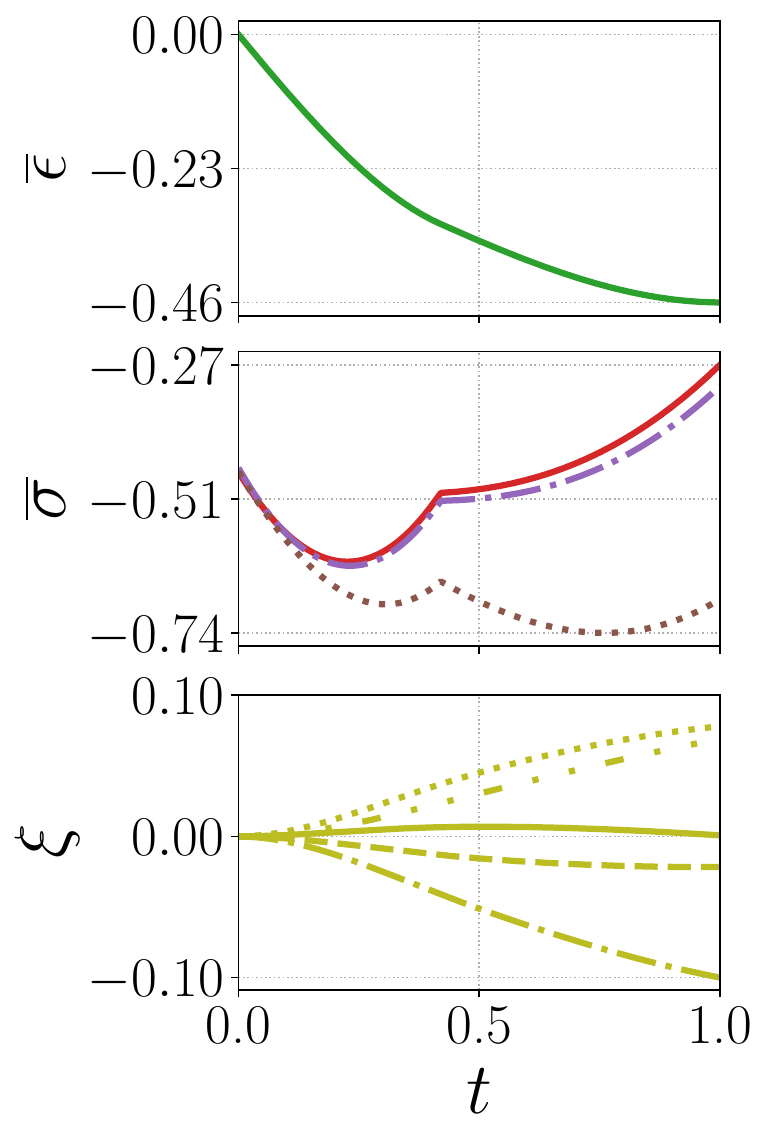}}  \\
    & \quad\makecell[l]{Stress predictions:\\ \includegraphics[width=0.15\linewidth]{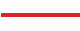} True\\\includegraphics[width=0.15\linewidth]{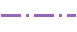} FNM--RNO\\\includegraphics[width=0.15\linewidth]{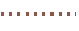} No memory} & & \quad\makecell[l]{Stress predictions:\\ \includegraphics[width=0.15\linewidth]{true_legend.pdf} True\\\includegraphics[width=0.15\linewidth]{rno_legend.pdf} FNM--RNO\\\includegraphics[width=0.15\linewidth] {no_memory_legend.pdf} No memory} & \\
     \multirow{2}{*}[2em]{\rotatebox{90}{\makecell{\bf Median testing error}}} & \includegraphics[width=\linewidth]{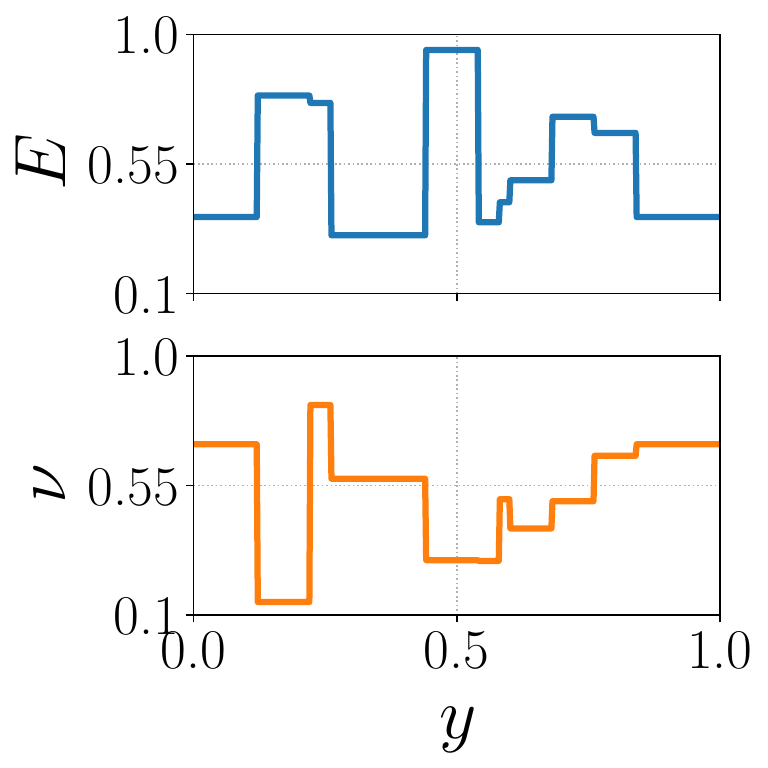} & \multirow{2}{*}[4.5em]{\includegraphics[width=\linewidth]{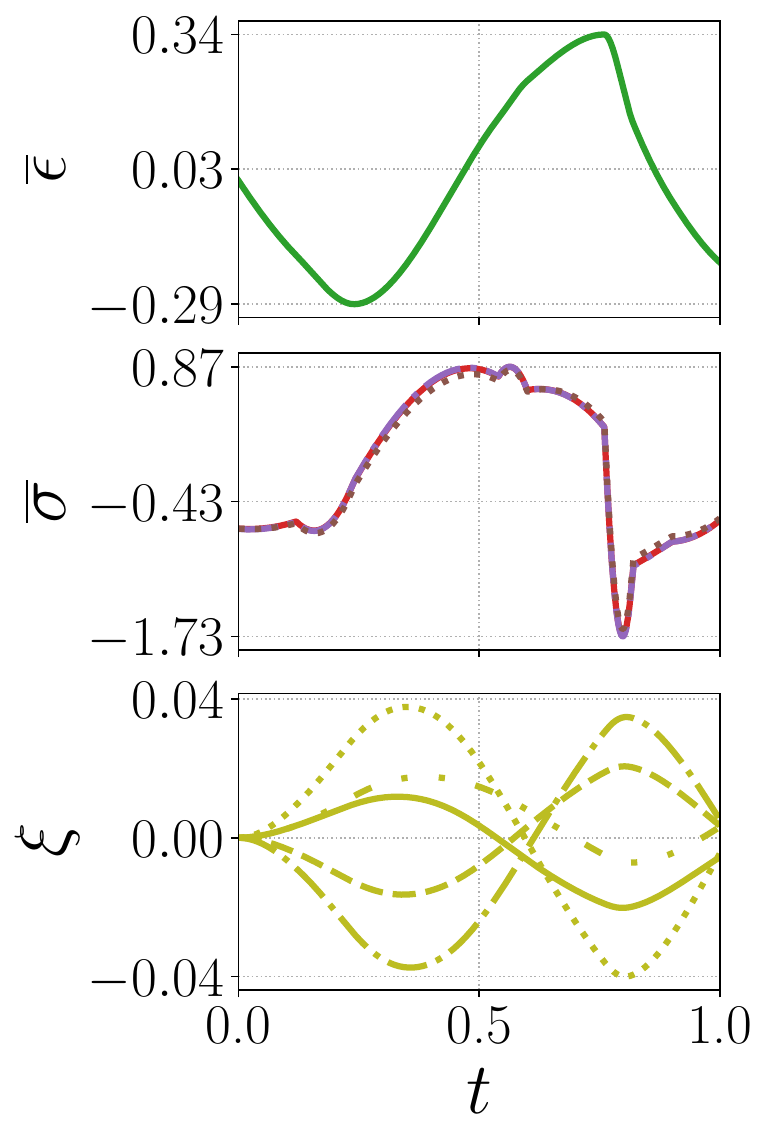}} & \includegraphics[width=\linewidth]{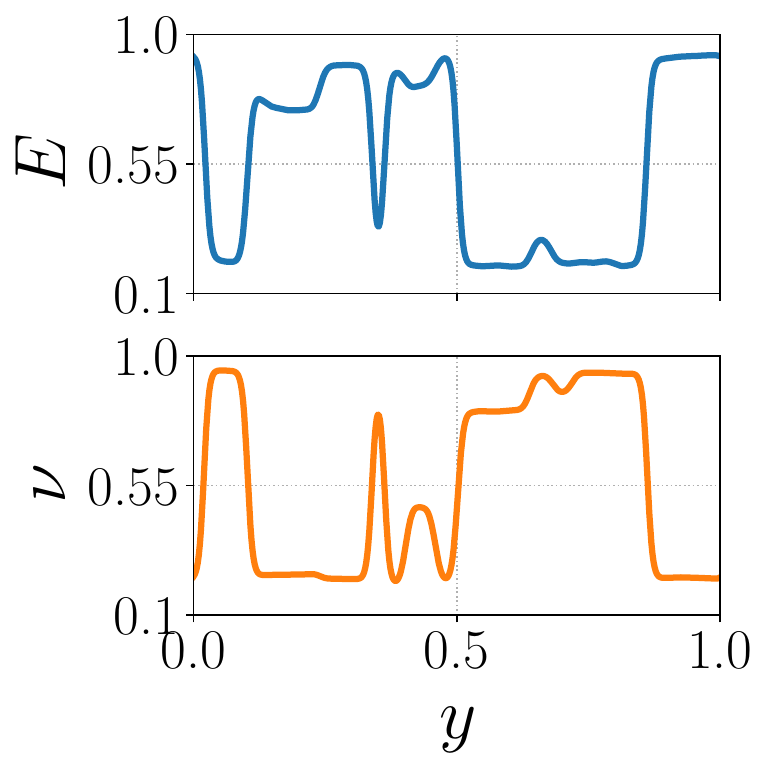} & \multirow{2}{*}[4.5em]{\includegraphics[width=\linewidth]{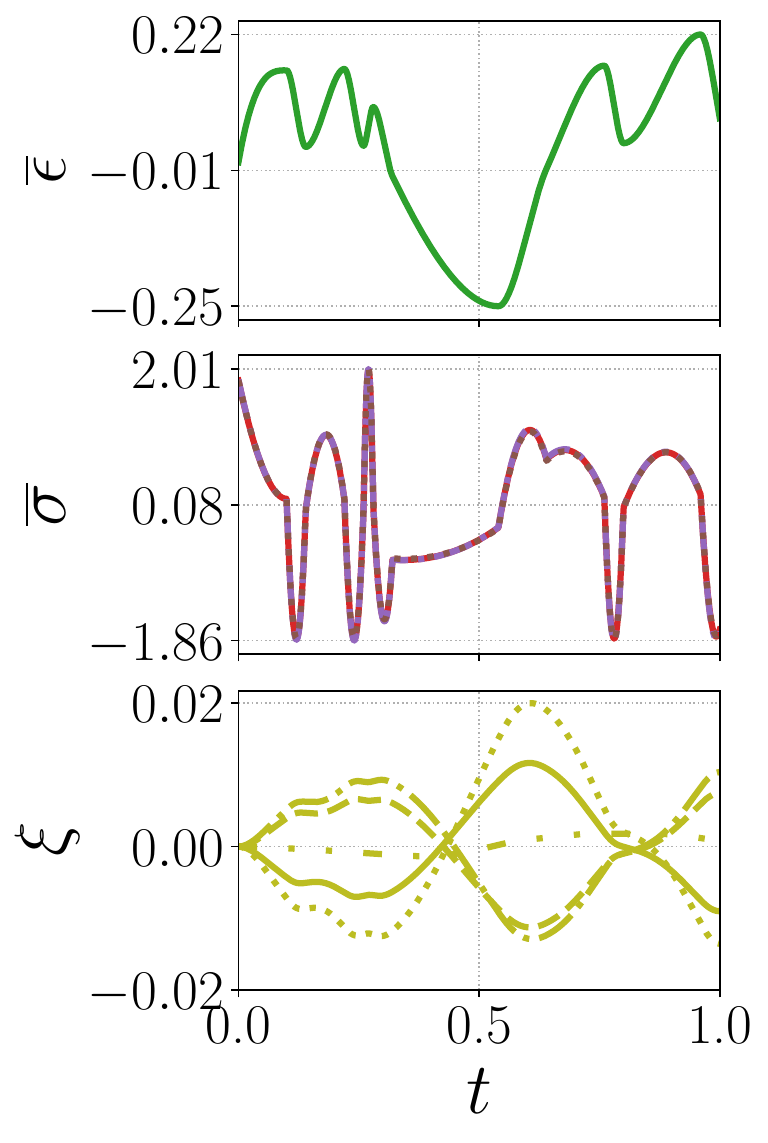}}\\
& \quad\makecell[l]{Stress predictions:\\ \includegraphics[width=0.15\linewidth]{true_legend.pdf} True\\\includegraphics[width=0.15\linewidth]{rno_legend.pdf} FNM--RNO\\\includegraphics[width=0.15\linewidth] {no_memory_legend.pdf} No memory} & & \quad\makecell[l]{Stress predictions:\\ \includegraphics[width=0.15\linewidth]{true_legend.pdf} True\\\includegraphics[width=0.15\linewidth]{rno_legend.pdf} FNM--RNO\\\includegraphics[width=0.15\linewidth]{no_memory_legend.pdf} No memory} &
    \end{tabular}
    \addtolength{\tabcolsep}{6pt}
    }
    \caption{Visualization of testing samples and FNM--RNO predictions (averaged stress $\sigmabar$ and internal variables $\xi$) with the largest and median relative $L^2$ testing error. We consider the trained FNM--RNO with 5 internal variables. We visualize the averaged stress response with (\textit{solid lines}) and without (\textit{dotted lines}) memory effects along with the FNM--RNO prediction (\textit{dash-dot lines}).}
    \label{fig:testing_visual}
\end{figure}

\paragraph{\bf Discretization Agnostic} The FNM--RNO constitutive model can be learned on data with one set of spatial and temporal resolutions and used to make predictions on another. To illustrate this property, we evaluate the trained FNM--RNO model on testing samples with varying spatial and temporal resolutions, and the resulting mean relative $L^2$ testing errors are visualized in \cref{fig:multiresolution_kv}. Testing samples with different resolutions for the HMC materials and the averaged stress trajectories are generated using linear interpolation. 

The results show that the accuracy of FNM--RNO predictions is relatively sensitive to changes in temporal resolution, primarily due to truncation errors in estimating the evolution of internal variables. For the PC dataset, the testing error shows less sensitivity to spatial resolution changes, as the discontinuity points are located on a low-resolution grid. In contrast, for the HMC dataset, the testing error increases with decreased spatial resolution when the temporal resolution is high. Overall, the FNM--RNO can predict stress response on finer spatial and temporal resolutions than those used for training without significant deterioration in accuracy.~\\

\begin{figure}[htbp]
    \centering
    \scalebox{0.9}{
    \begin{tabular}{c c}
        \hspace{0.06\linewidth}\makecell{\bf Mean relative $L^2$ testing error \\ \bf on the PC dataset\bf \\\bf  at different testing resolutions} & \hspace{0.06\linewidth}\makecell{\bf Mean relative $L^2$ testing error \\ \bf on the HMC dataset \\\bf at different testing resolutions}   \\
        \includegraphics[width=0.42\linewidth]{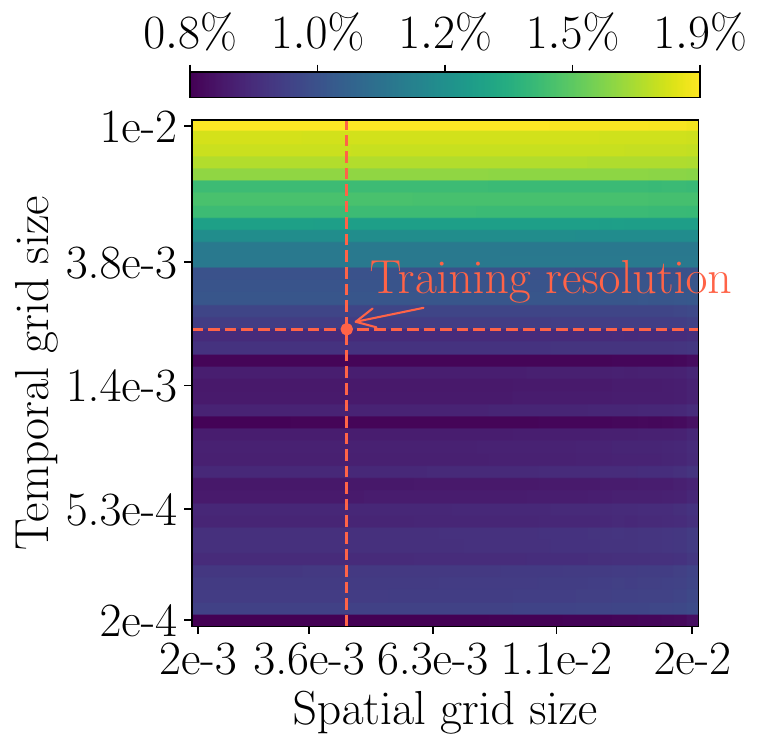} & \includegraphics[width=0.42\linewidth]{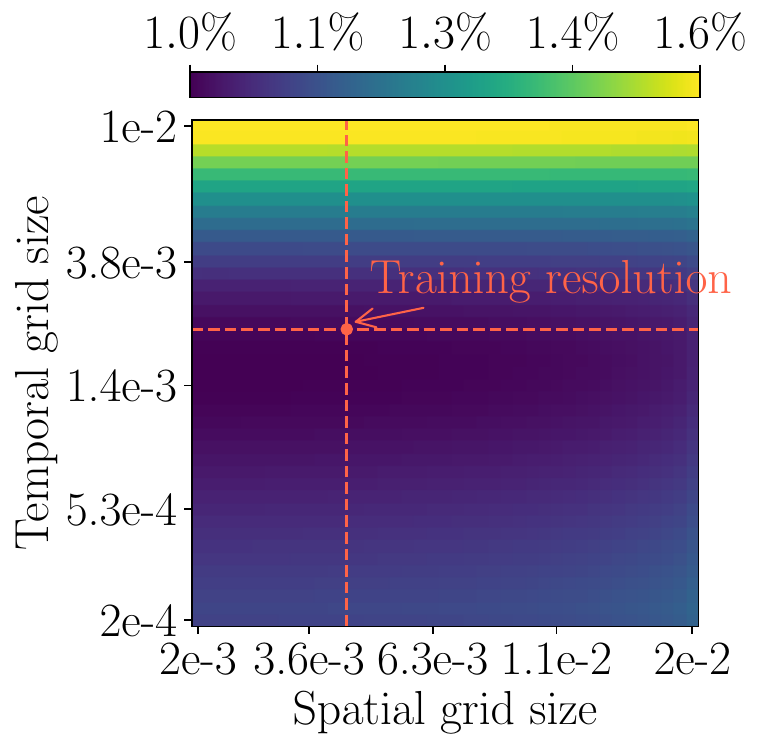}
    \end{tabular}
    }
    \caption{The mean relative $L^2$ error on the PC (\textit{left}) and HMC (\textit{right}) testing samples at different testing resolutions for the trained FNM--RNO with 5 internal variables.}
    \label{fig:multiresolution_kv}
\end{figure}

\paragraph{\bf Deployment in Macroscale Simulations}
    We consider a macroscale problem on domain $\cD \times \cT=[0,1] \times [0,1]$ and employ a body forcing term $f(x,t)= 100\sin(8\pi(x + t))$. This leads to a stress evolution given by $\sigma(x,t) = -800\cos(8\pi(x+t))$. Our goal is to solve for the displacement $u(x,t)$ with the boundary conditions $u(0, t)=u(1,t)=0$ for all $t\in[0, 1]$ and the initial condition $u(x, 0)=0$ for all $x\in[0,1]$.
    We use four different constitutive models to solve for the displacement:
    \vspace{5pt}
    \begin{enumerate}[label=(\roman*)]
    \item Homogenized stress response using a memory kernel form as in \cref{eqn:kernelform},
    \item Multiscale stress response with $\varepsilon^{-1} = 5, 10, 20, 40$ or $80$,
    \item Homogenized stress response without memory effects; $K\equiv0$ in \cref{eqn:kernelform},
    \item FNM--RNO stress response with $5$ internal variables.
    \end{enumerate}
    \vspace{5pt}
    We expect that the multiscale solution converges to the homogenized solution in $L^2$ as $\varepsilon\to 0$ for each material microstructure, while FNM--RNO stress response and stress response without memory leads to biased macroscale solutions compared to the homogenized stress response.

    In \cref{fig:macro_error}, we visualize the distributions of the relative $L^2$ error in macroscale solutions using 800 material microstructure samples from the HMC dataset, where the macroscale solutions obtained by the homogenized stress response are used as the reference. The multiscale solutions linearly converge to the homogenized solutions as $\varepsilon \to 0$, and the FNM--RNO solutions have error distribution similar to that arising from the multiscale solutions at $\varepsilon^{-1}=20$. The stress response without memory leads to macroscale solution error higher on average than FNM--RNO, and lower on average than multiscale simulation with $\varepsilon^{-1}=10$. 
    
    In \cref{fig:macro_solution_sample}, we visualize a material microstructure and its macroscale solutions. This sample corresponds to the median of the FNM--RNO macroscale solution error distribution shown in \cref{fig:macro_error}. We also visualize the pointwise absolute error of macroscale solutions compared to the one obtained via homogenized stress response. The multiscale stress response does not exhibit noticeable accumulation of error in time. The FNM--RNO solution accumulates errors in time but the overall error is less than that resulting from use of linear stress response without memory effects.

    \begin{figure}[htbp]
    \centering
        \scalebox{0.9}{
    \begin{tabular}{c}
        \hspace{0.09\linewidth}\makecell{\bf Relative $L^2$ error in macroscale solutions\\
        \bf using HMC materials}\\
        \includegraphics[width=0.5\linewidth]{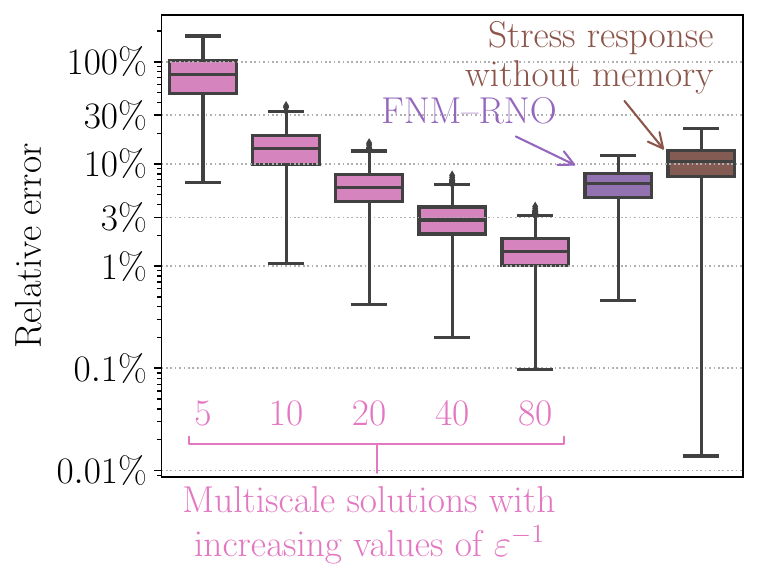}
        \end{tabular}}
    \caption{The distributions of the relative $L^2$ error in macroscale solutions obtained by different constitutive models using 800 material microstructure samples from the HMC dataset. The error is computed relative to the macroscale solutions obtained by the homogenized stress response using memory kernels.}
    \label{fig:macro_error}
    \end{figure}

    \begin{figure}[htbp]
        \centering
        \scalebox{0.84}{
        \addtolength{\tabcolsep}{-6pt}
        \begin{tabular}{c}
            \hspace{0.03\linewidth}\makecell{\bf Reference solution \\\bf using homogenized \\\bf stress response}\\
            \includegraphics[width=0.25\linewidth]{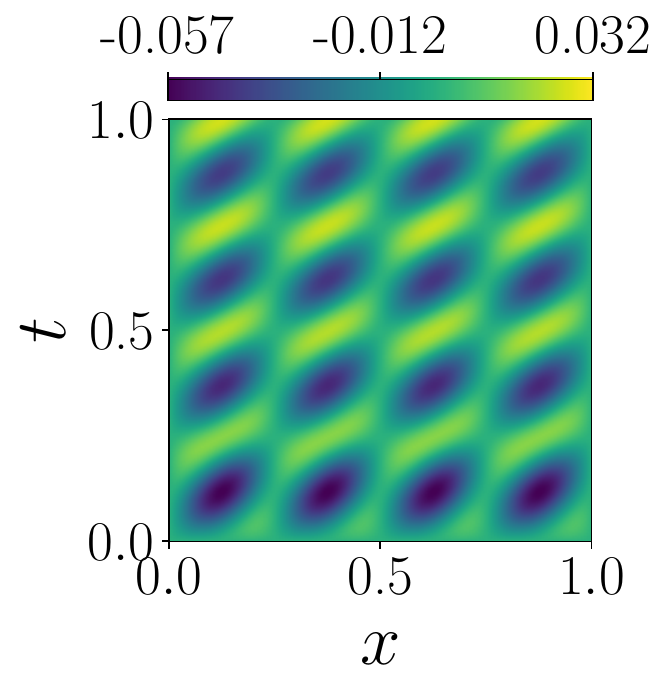}\\
            \hspace{0.045\linewidth}\makecell{\bf Material\\ \bf microstructure} \\ 
            \includegraphics[width=0.25\linewidth]{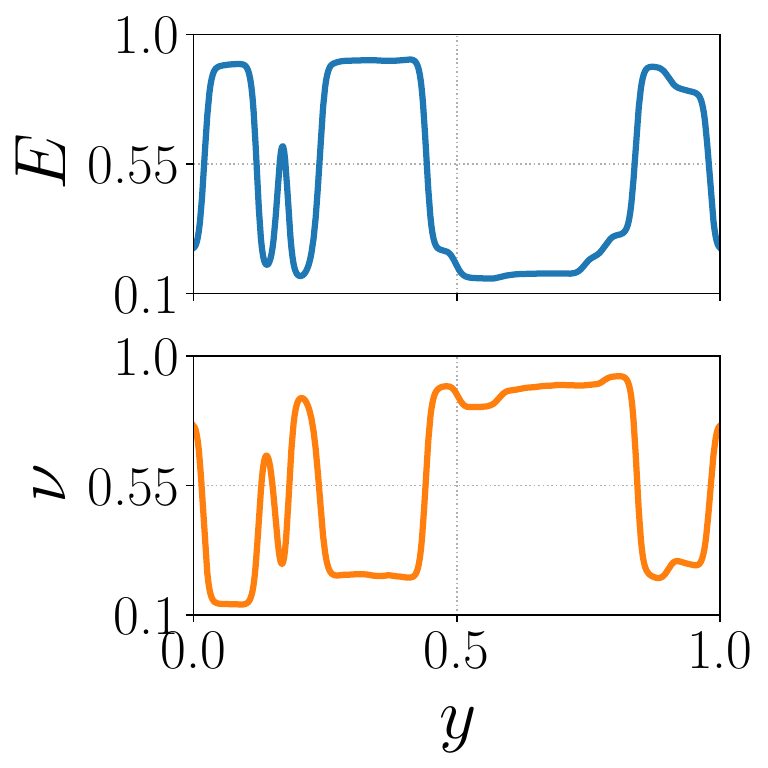}
        \end{tabular}
        \begin{tabular}{C D D D}
         & \hspace{0.12\linewidth}\makecell{\bf Multiscale\\\bf stress response} & \makecell{\bf \hspace{0.1\linewidth}Stress response\\\bf \hspace{0.1\linewidth}without memory} & \hspace{0.12\linewidth}\makecell{\bf FNM--RNO\\\bf stress response}\\
            \rotatebox{90}{\makecell{\bf Macroscale\\\bf solution}}&
             \includegraphics[width=\linewidth]{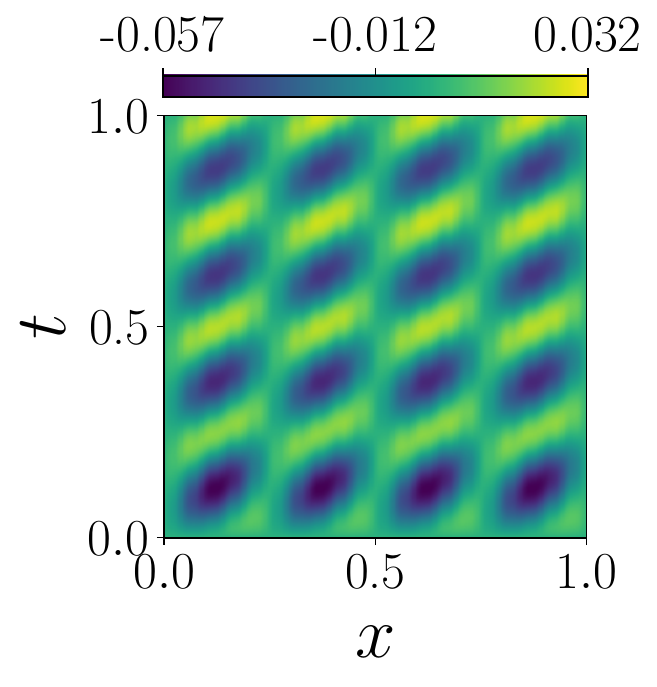}& \includegraphics[width=\linewidth]{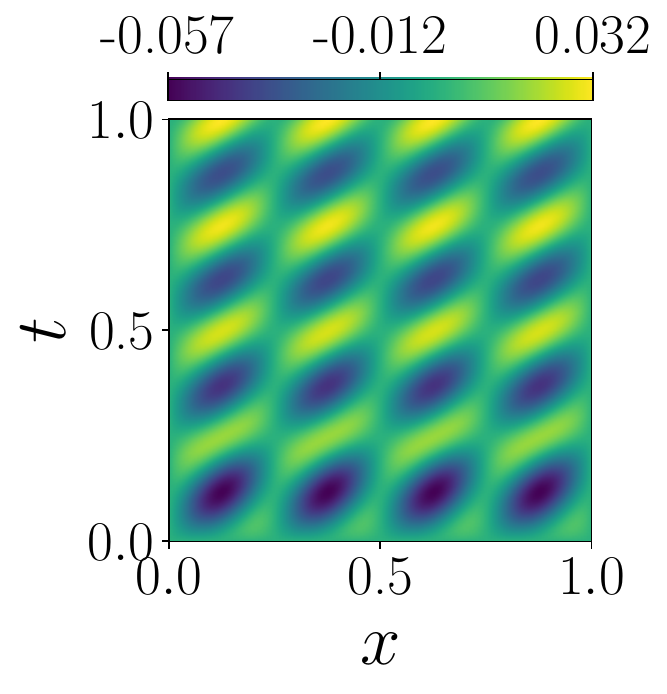} & \includegraphics[width=\linewidth]{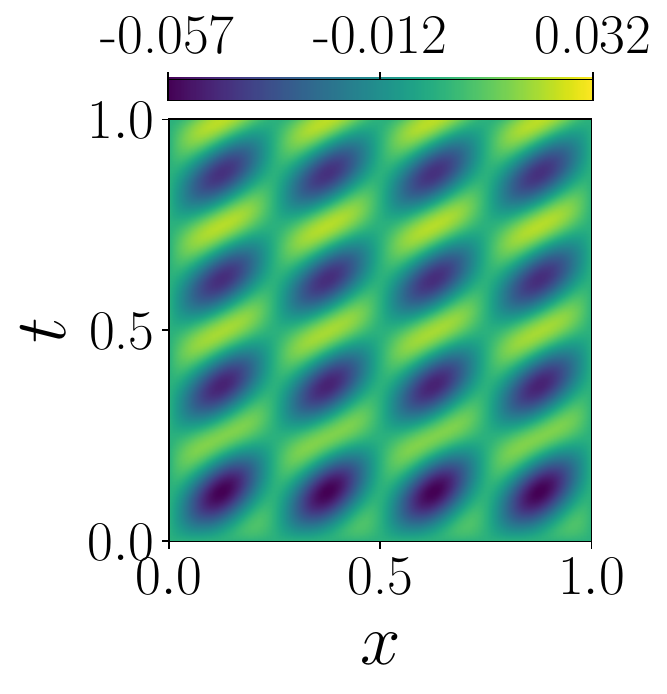}  \\
             & \makecell{$8.5\%$ relative\\ $L^2$ error} & \makecell{$13\%$ relative\\ $L^2$ error} & \makecell{$6.6\%$ relative\\ $L^2$ error}\\
             \rotatebox{90}{\makecell{\bf Pointwise\\ \bf absolute error}}
             & \includegraphics[width=\linewidth]{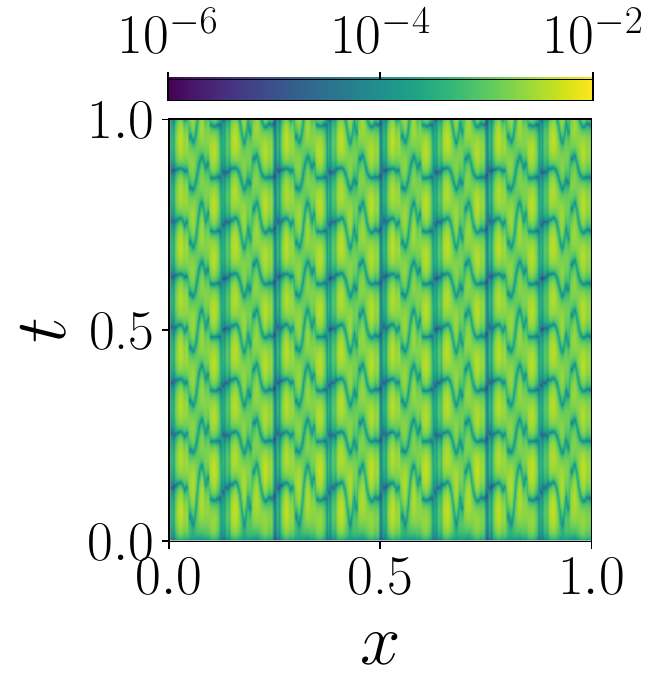}& \includegraphics[width=\linewidth]{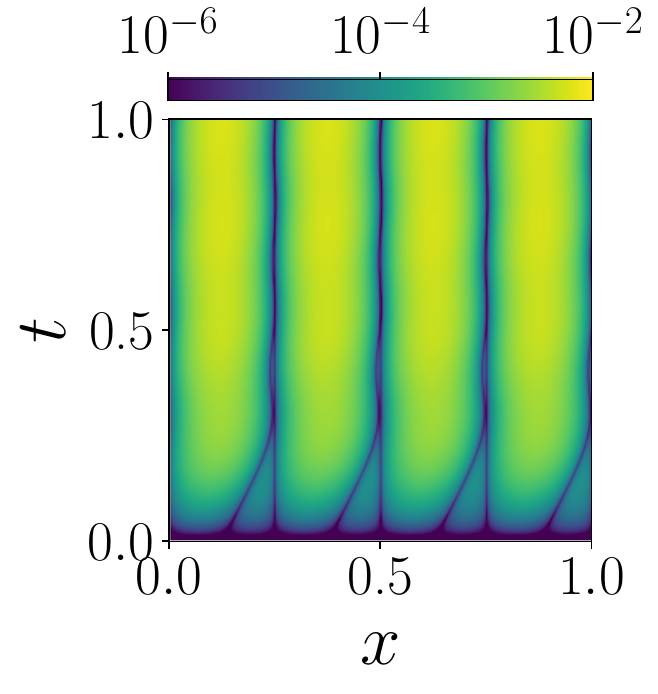} & \includegraphics[width=\linewidth]{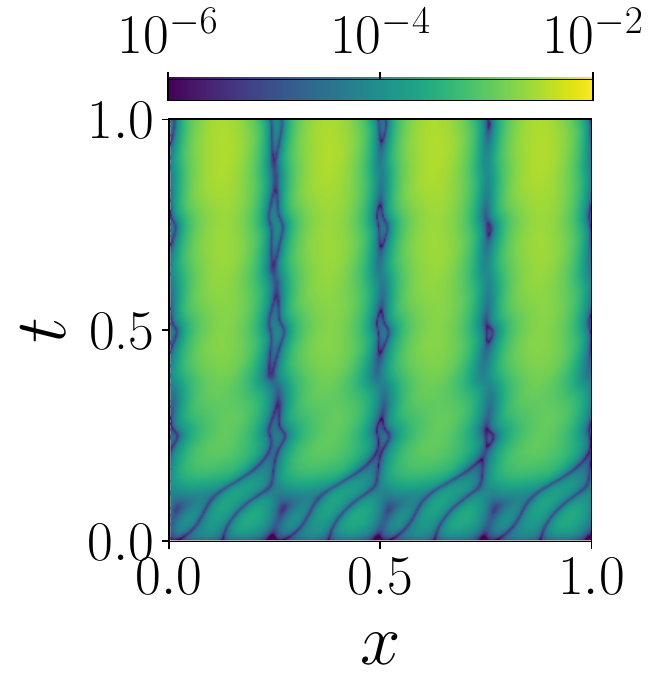}
        \end{tabular}\addtolength{\tabcolsep}{6pt}}\\
        \scalebox{0.79}{\addtolength{\tabcolsep}{-6pt}
        \begin{tabular}{c c c c c}
        \multicolumn{5}{c}{\bf FNM--RNO internal variables ($\mathbf{L}$ = 1 -- 5)}\\
\includegraphics[width=0.25\linewidth]{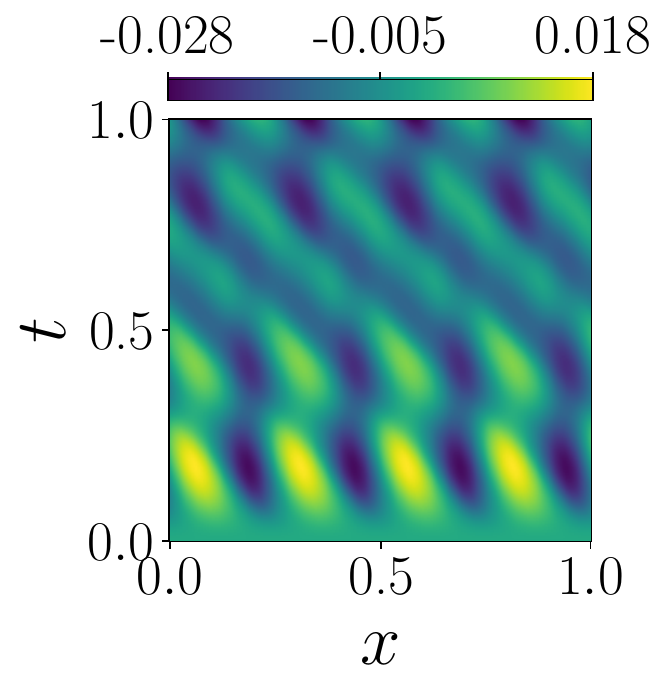} &        \includegraphics[width=0.25\linewidth]{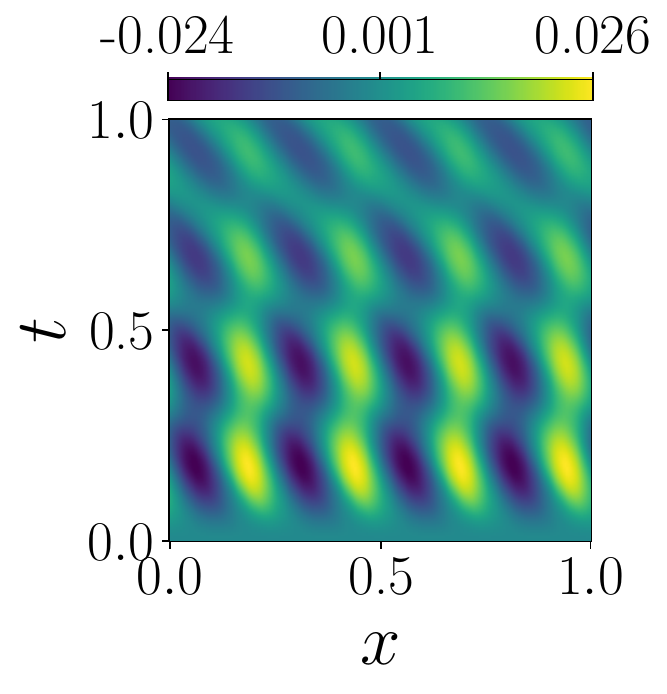} &        \includegraphics[width=0.25\linewidth]{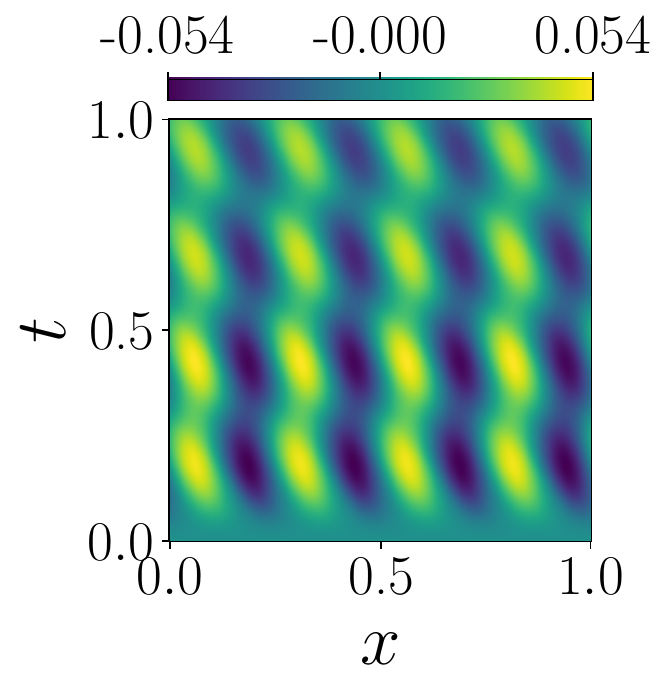} &       \includegraphics[width=0.25\linewidth]{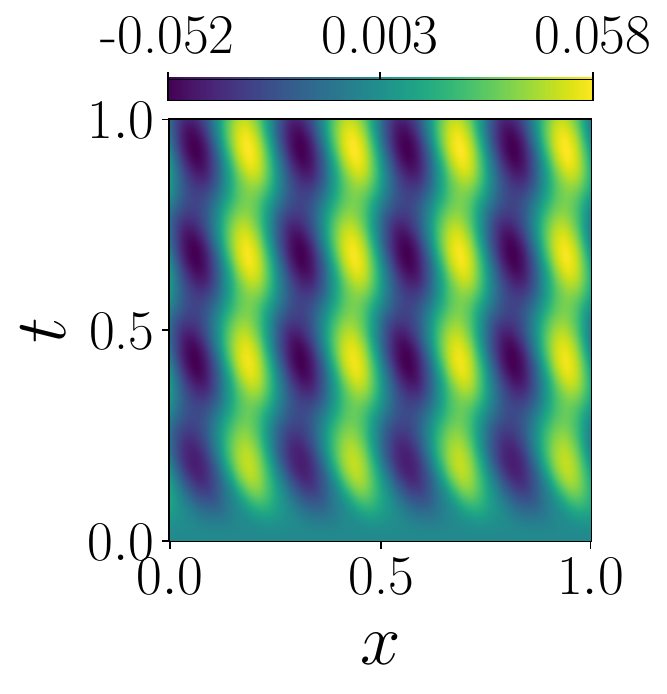} &    \includegraphics[width=0.25\linewidth]{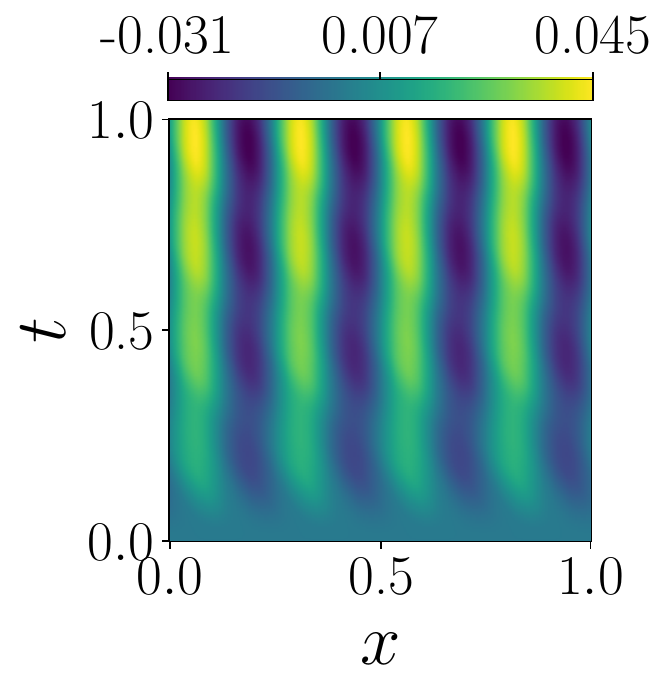}
        \end{tabular}\addtolength{\tabcolsep}{6pt}}
        \caption{Visualization of macroscale solutions using different constitutive models for an HMC material microstructure sample. We visualize the multiscale solution with $\varepsilon^{-1}=20$. This sample corresponds to the median of the FNM--RNO macroscale solution relative error distribution in \cref{fig:macro_error}.}
        \label{fig:macro_solution_sample}
    \end{figure}

\subsection{Application to Elasto-Viscoplasticity}\label{subsec:plastic}

We consider elasto-viscoplastic composites in one dimension. The cell problem is given by
\begin{subequations}\label{eq:evp_cell_problem}
    \begin{align}
        &-\p_y\sigma(y, t) = 0, &\quad (y,t) &\in \Omega \times \cT, \\
         &\sigma(y, t) = E(y)(\p_y u(y, t) - \epsilon_p(y,t)), &\quad (y,t) &\in \Omega \times \cT, \\
        &\partial_t \epsilon_p(y, t) = \dot{\epsilon}_{p0}(y)\text{sign}(\sigma(y,t))\left(\frac{|\sigma(y,t)|}{\sigma_Y(y)}\right)^{n(y)}, &\quad (y,t) &\in \Omega \times \cT, \\
        &u(0, t) = 0, \; u(1, t) = \epsbar(t), &\; t &\in \cT, \\
        &u(y, 0) = 0,\; \epsilon_p(y, 0) = 0, &\; y&\in \Omega.
    \end{align}
\end{subequations}
Here $\epsilon_p$ is the plastic strain, $E$ is the Young's modulus, $\dot{\epsilon}_{p0}$ is the rate constant, $\sigma_Y$ is the yield stress, and $n$ is the rate exponent. We assume that these four material parameters $(E,\dot{\epsilon}_{p0},\sigma_Y,n)$ vary spatially in the unit cell. Our goal is to learn the constitutive model $\{\{\epsbar(t)\}_{t\in\cT}, E, \dot{\epsilon}_{p0}, \sigma_Y,n\}\mapsto\{\sigmabar(t)\}_{t\in\cT}$, where $\epsbar(t)=\int_{\Omega}\partial_y u(y,t) dy$ and $\sigmabar(t)=\int_{\Omega}\sigma(y,t) dy$. We highlight the fact that the constitutive model can be expressed using the averaged plastic strain $\overline{\epsilon}_p(t) = \int_{\Omega}\epsilon_p(y, t) dy$ as an internal variable; see~\cite[Eq. 11]{liu2023learning}

We generate two datasets following a strategy similar to that described in \cref{subsec:data_set}. The piecewise-constant random material (PC-EVP) uses a uniform distribution on $[1, 10]\times[0.5, 2.0]\times[0.1, 1.0]\times[1, 20]$, independently drawn in each of the four components of the materials property vector $(E^{(j)}, \dot{\epsilon}_{p0}^{(j)}, \sigma_Y^{(j)}, n^{(j)})$, for each piece with
label $j$, and drawn i.i.d.\ with respect to $j$. The continuous random materials (C-EVP) take the spatially smooth piecewise-constant random material as the random mean function. The sampling procedure for the mean is similar to HMC materials in \cref{subsec:data_set}, except that the values taken in each piece for the four materials are independently distributed.

We consider an FNM--RNO architecture given by
\begin{subequations}\label{eqn:RNOevp}
\begin{align}
    \sigmabar(t) &= F_{\FNM}(\epsbar(t), \xi(t); E, \dot{\epsilon}_{p0}, \sigma_Y,n)\\
    \dot{\xi}(t) &= G_{\FNM}(\epsbar(t), \xi(t); E, \dot{\epsilon}_{p0}, \sigma_Y,n)\\
    \xi(0) &= 0.
\end{align}
\end{subequations}
Note that, comparing with \cref{eqn:FNMRNO}, we have suppressed dependence
on $\dot{\epsbar}(t)$ in $F_{\FNM}$, motivated by the analysis in \cite{liu2023learning}. The same work motivates an expectation that the internal variable should be scalar, and indeed we expect it to follow $\overline{\epsilon}_p(t) \approx c(\epsbar, E, \dot{\epsilon}_{p0}, \sigma_Y,n)\xi(t)$, where $c$ is a scalar-valued function that can be found numerically for each set of materials and averaged strain trajectory. We use a similar FNM architecture and training procedure as in \cref{subsec:architecture}, except that: (i) 2 Fourier modes are used to parameterize all the convolution operators; and (ii) the penalty term in the loss function is not included.

We evaluate the trained FNM--RNO on 2,500 testing samples from the PC-EVP and C-EVP datasets. The distributions of the relative $L^2$ testing error are shown in \cref{fig:evp_error_distribution}. For the PC-EVP dataset, the FNM--RNO achieves mean relative $L^2$ errors of $3.4\%$ in predicting the averaged stress and $1.4\%$ in predicting the averaged plastic strain up to a constant. For the C-EVP dataset, the FNM--RNO achieves mean relative $L^2$ errors of $2.8\%$ in predicting the averaged stress and $1.4\%$ in predicting the averaged plastic strain up to a constant. In \cref{fig:testing_visual_EVP}, we visualize the testing samples in the PC-EVP and C-EVP datasets with the largest and median relative $L^2$ error in the averaged stress.

\begin{figure}[htbp]
    \centering
    \scalebox{0.9}{
    \begin{tabular}{c}
         \hspace{0.1\linewidth}\makecell{\bf Relative $L^2$ testing error on \vspace{0.2cm}}\\
         \hspace{0.09\linewidth}\makecell{\bf PC-EVP\\\bf dataset} \hspace{-0.02\linewidth}\makecell{\bf C-EVP\\\bf dataset}\\
         \includegraphics[width=0.5\linewidth]{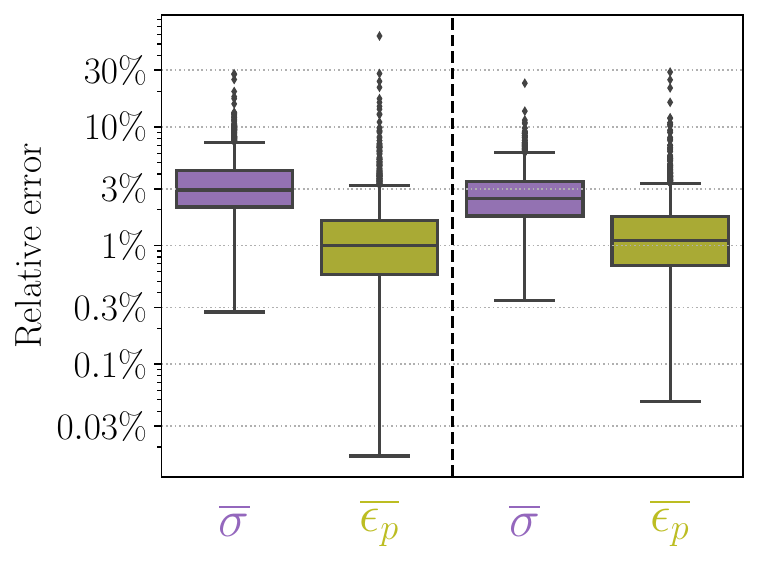} 
    \end{tabular}
    
    }
    \caption{The distributions of the relative $L^2$ error on 2,500 testing samples from the PC-EVP dataset (\textit{left}) and the C-EVP dataset (\textit{right}) for elasto-viscoplasticity. We consider the error in the the trained FNM--RNO predictions of (i) the averaged stress $\sigmabar$ and (ii) the averaged plastic strain $\overline{\epsilon}_p$ up to a multiplicative constant via the internal variable $\xi$.}
    \label{fig:evp_error_distribution}
\end{figure}

\begin{figure}[htbp]
    \centering
    \scalebox{0.8}{
    \addtolength{\tabcolsep}{-6pt}
    \begin{tabular}{E F F F F}
    & \multicolumn{2}{c}{\makecell{\bf Piecewise-constant material\\ \bf dataset (PC-EVP)}} & \multicolumn{2}{c}{\makecell{\bf Continuous material\\\bf dataset (C-EVP)}} \\
     \rotatebox{90}{\makecell{\bf Largest testing error}}   & \includegraphics[width=\linewidth]{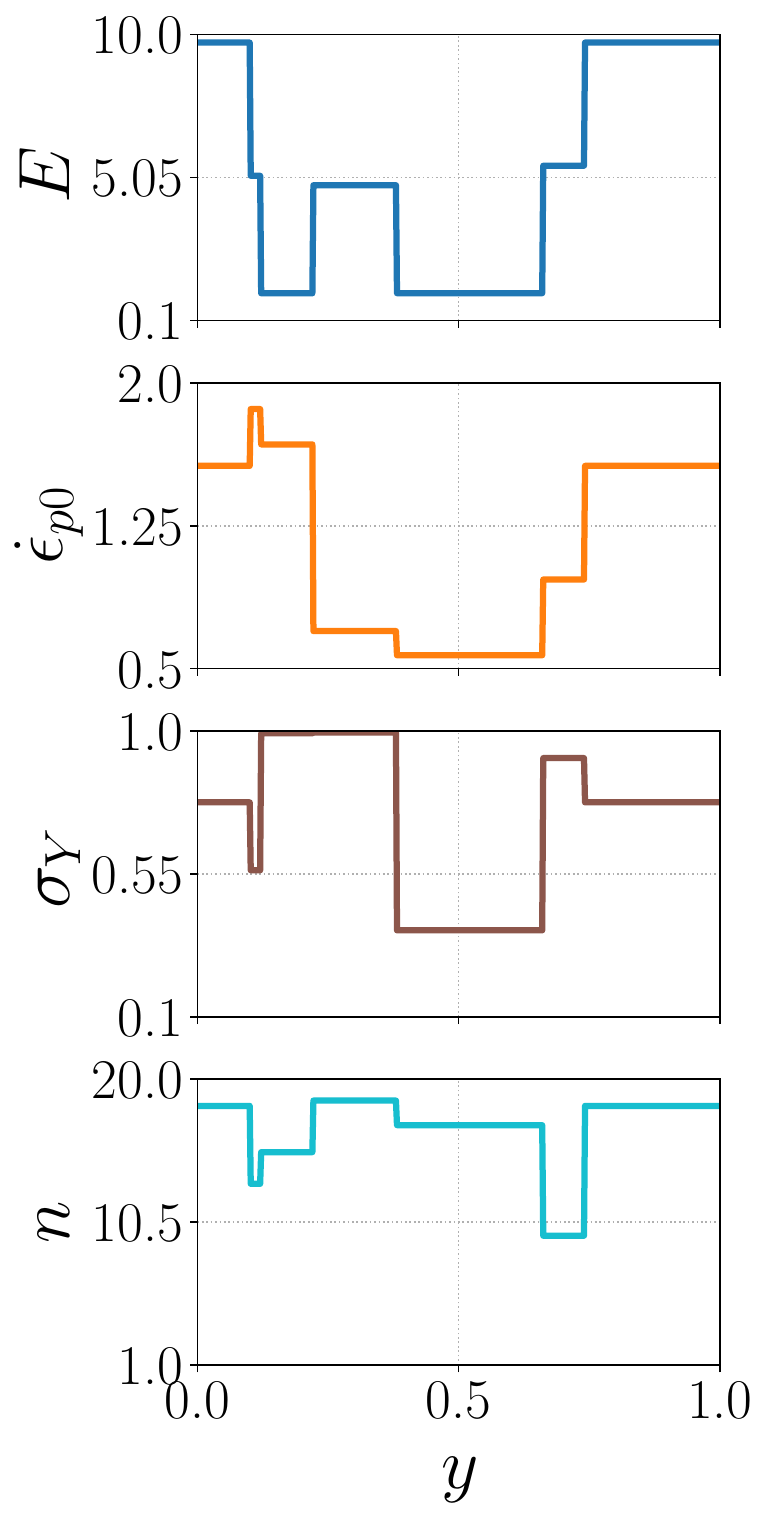} & \includegraphics[width=\linewidth]{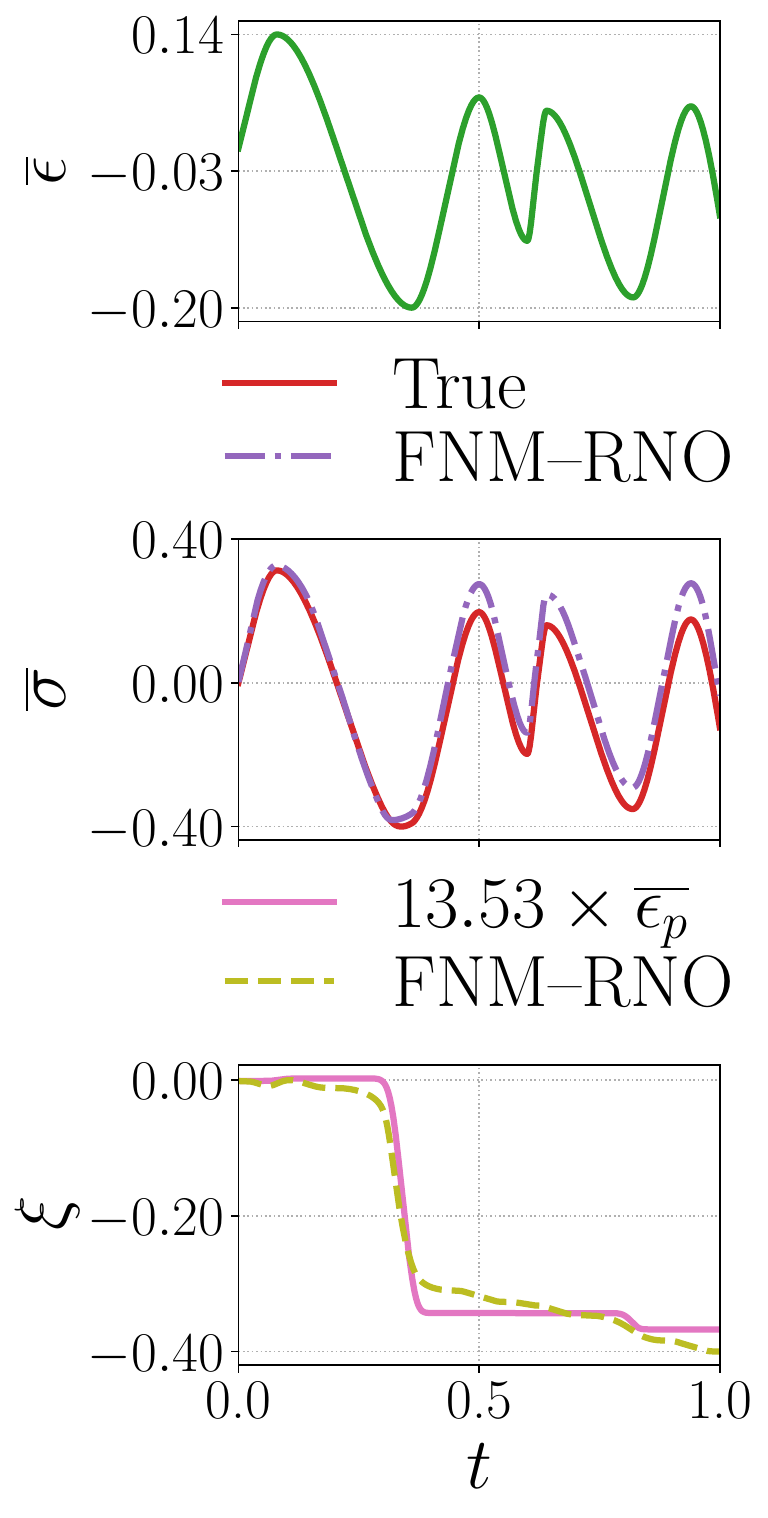} & \includegraphics[width=\linewidth]{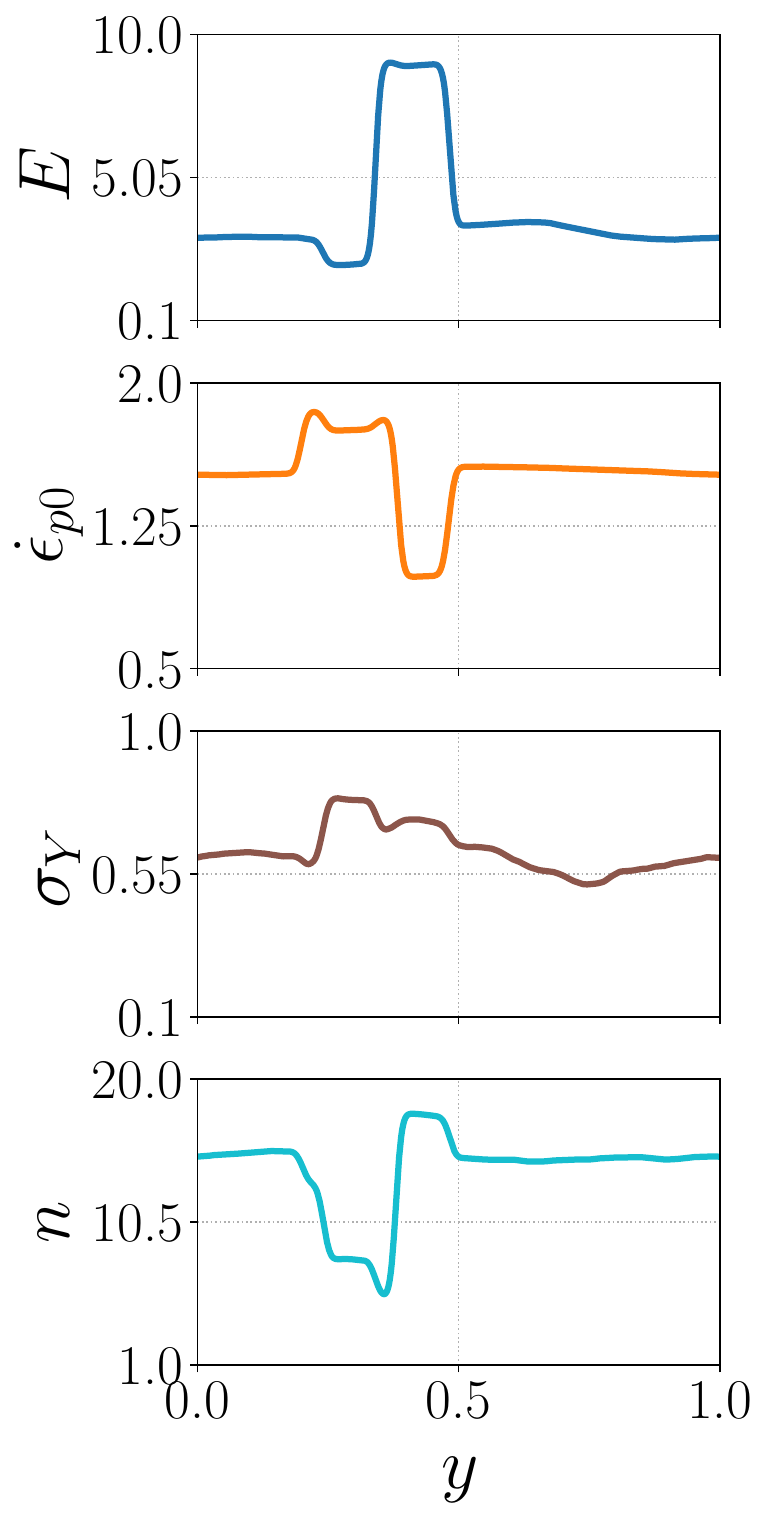} & \includegraphics[width=\linewidth]{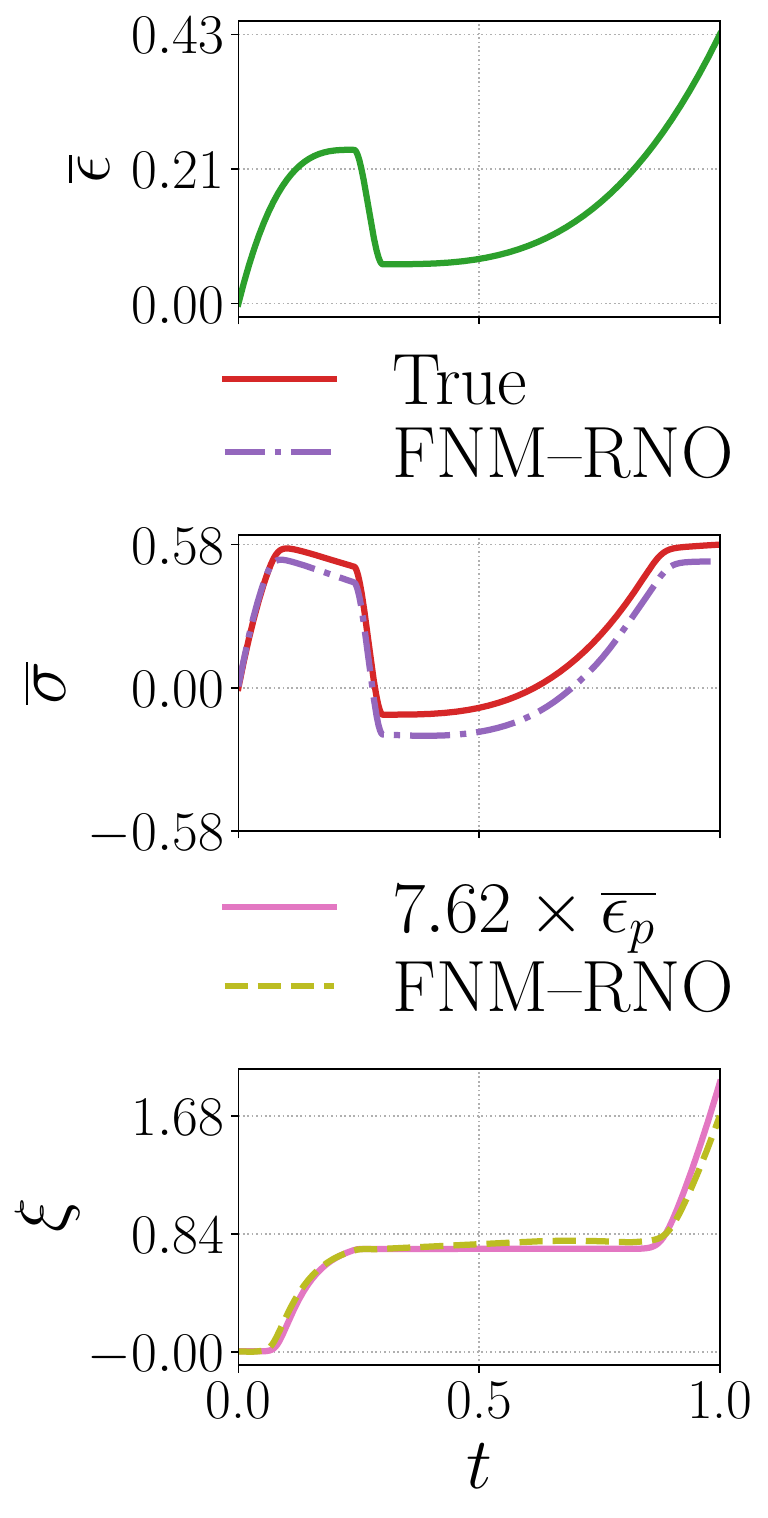}  \\
     \rotatebox{90}{\makecell{\bf Median testing error}} & \includegraphics[width=\linewidth]{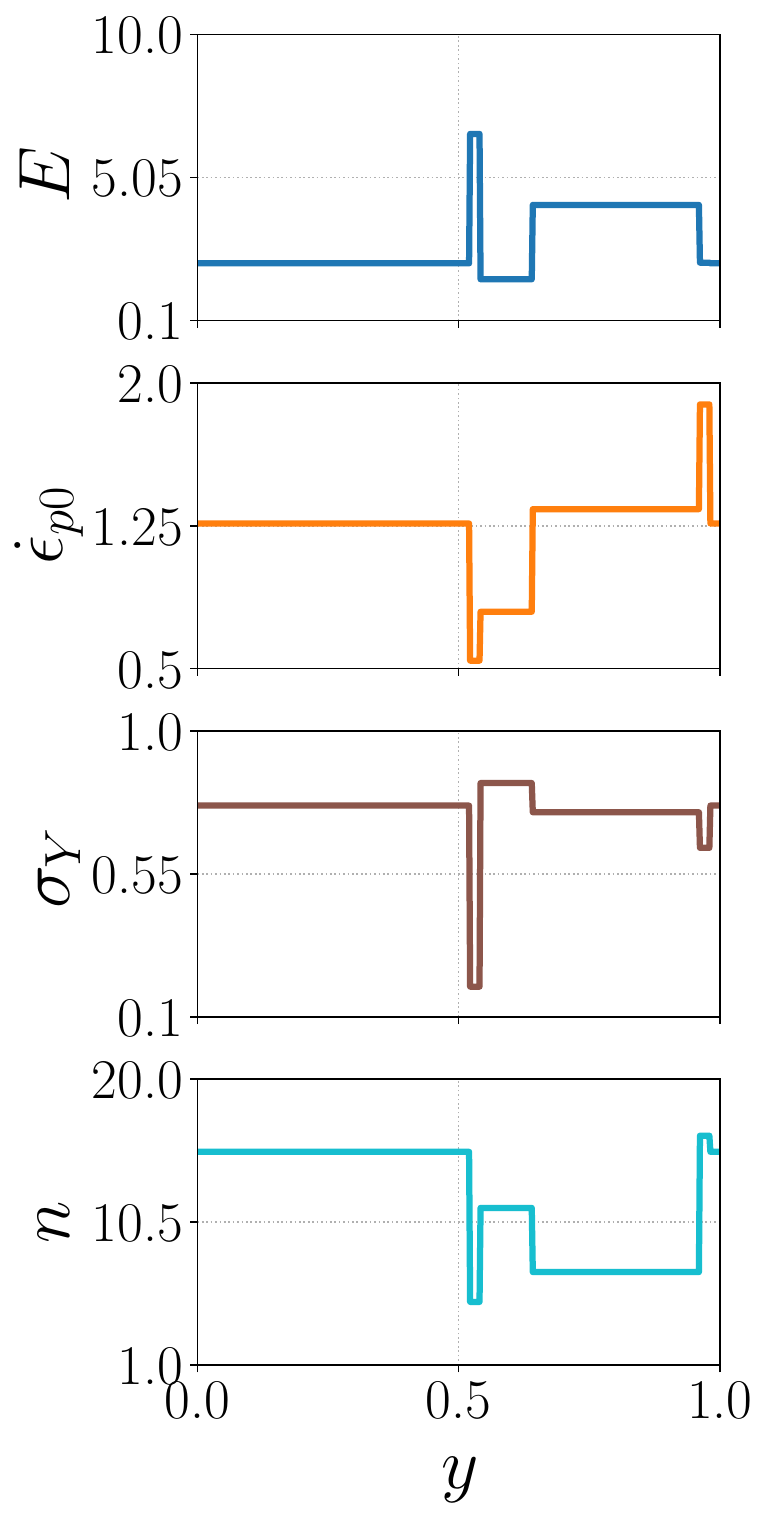} & \includegraphics[width=\linewidth]{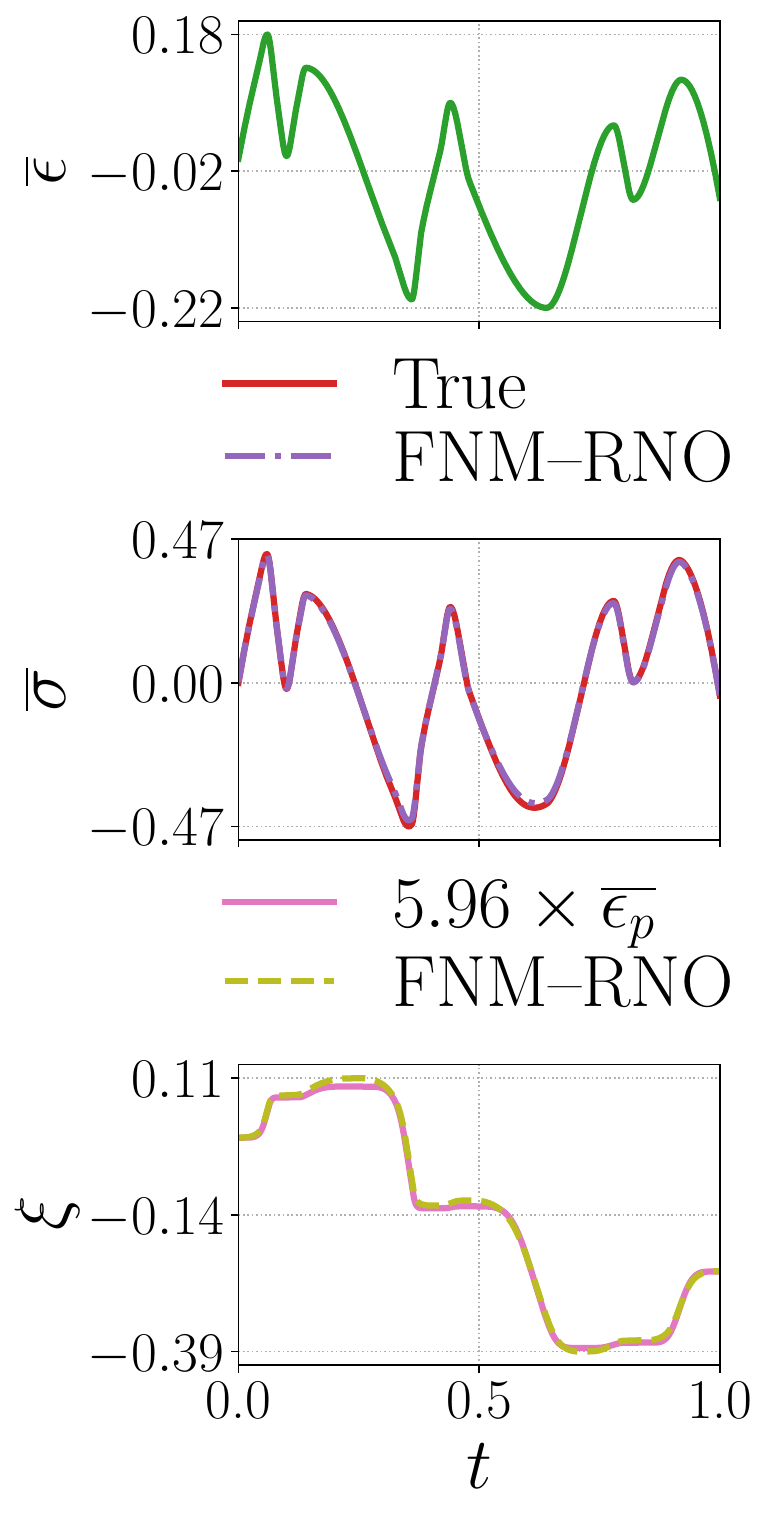} & \includegraphics[width=\linewidth]{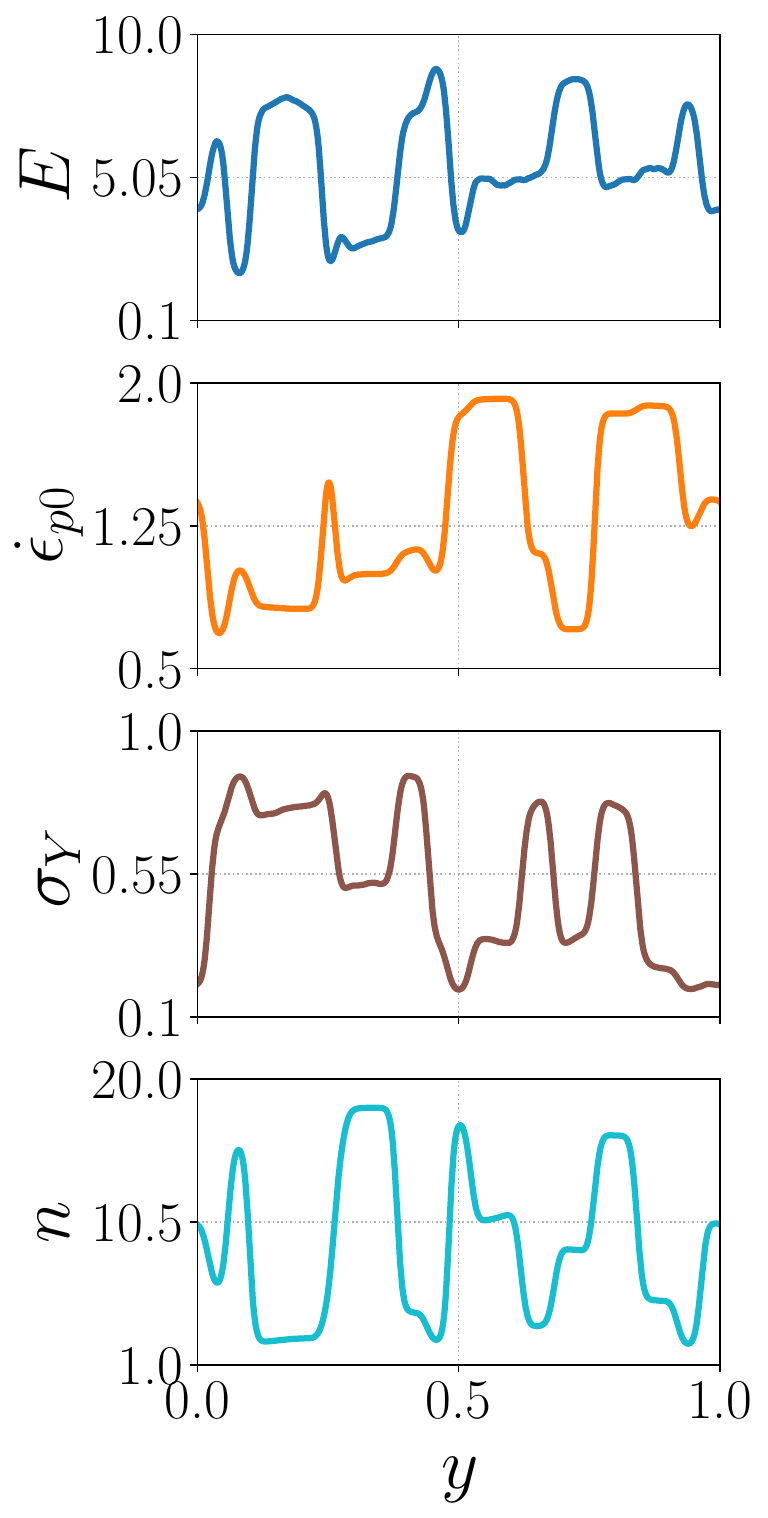} & \includegraphics[width=\linewidth]{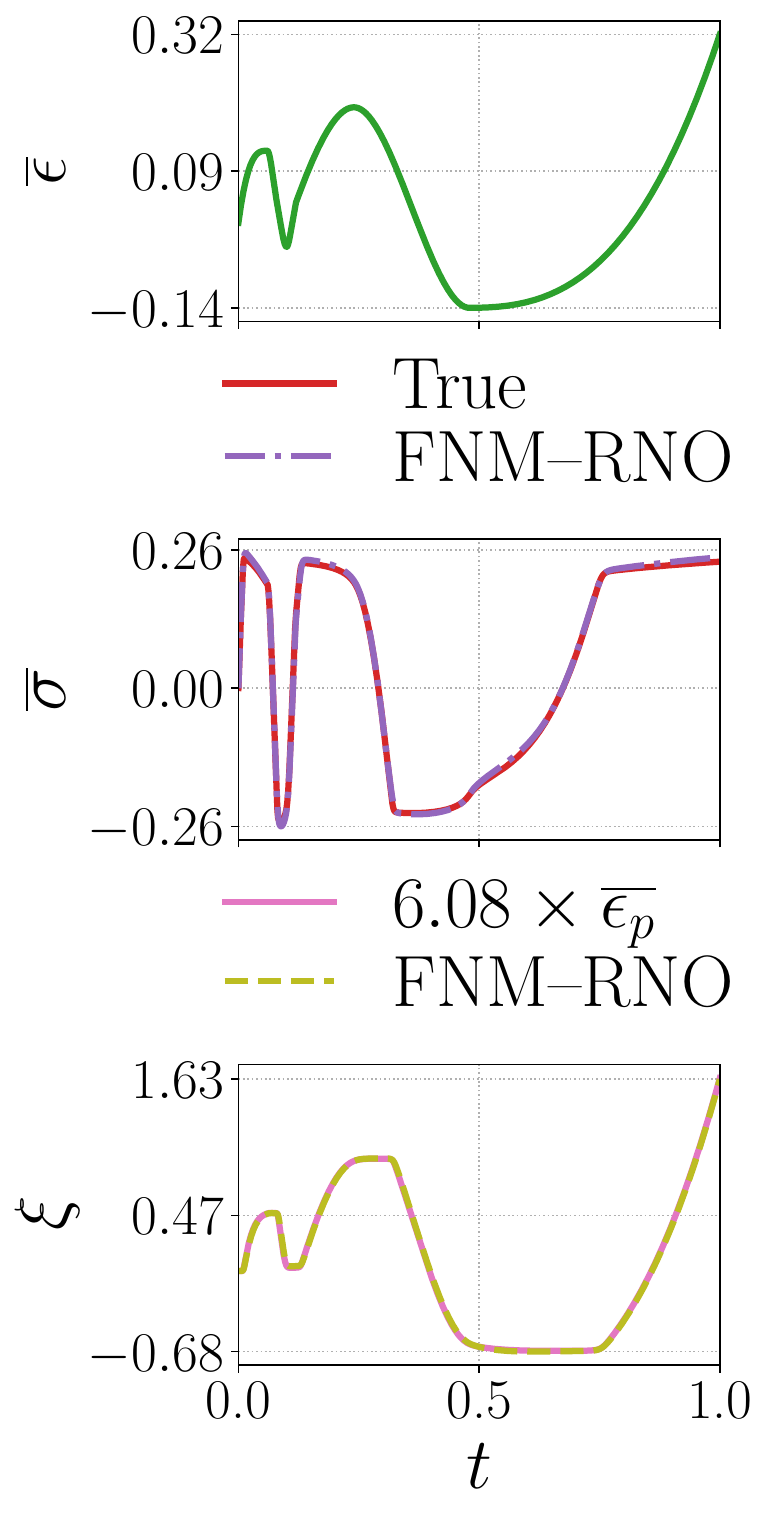}
    \end{tabular}
    \addtolength{\tabcolsep}{6pt}
    }
    \caption{Visualization of testing samples and FNM--RNO predictions (averaged stress $\sigmabar$ and internal variables $\xi$) with the largest and median relative $L^2$ error in stress response for elasto-viscoplasticity. We visualize the internal variable $\xi$ (\textit{dashed line}) along with the plastic strain $\overline{\epsilon}_p$ scaled by a constant (\textit{solid line}), where the constant is determined through minimizing the $L^2$ distance between the internal variable and the scaled averaged plastic strain.}
    \label{fig:testing_visual_EVP}
\end{figure}

\section{Conclusion}
In this paper, we present a novel recurrent neural operator architecture capable of predicting the memory-dependent constitutive laws of homogenized multiscale materials over a wide array of microstructures. Our architecture is designed as a neural differential equation with a Fourier neural mapping on the right-hand side, making it agnostic to the level of discretization or sampling of the material microstructure. Guided by the theory of homogenization in the one-dimensional Kelvin--Voigt model, we derive Lipschitz properties of the cell problem and, to our knowledge, present the first universal approximation guarantees of a data-driven model for predicting a viscoelastic constitutive law as a function of the microstructure. Our numerical experiments confirm, in the context of the multiscale Kelvin--Voigt linear viscoelastic model, that the neural operator accurately predicts the homogenized dynamics of the material and is able to generalize from training on piecewise constant to testing on continuous microstructures. We also show that our architecture can be applied beyond the specifics of linear viscoelasticity: we demonstrate its efficacy in learning the nonlinear constitutive model for homogenized viscoplastic materials.

Our objective in this work was to introduce a novel data-driven modeling technique for predicting microstructure-dependent constitutive laws, and to study this approach in the well-defined setting of one-dimensional Kelvin--Voigt materials where theoretical guarantees could be established. For future work, we aim to apply this idea of combining neural ODEs and mesh invariant neural operators to the simulation of 2D and 3D homogenized materials in viscoelasticity and viscoplasticity, which is outside the scope of the current paper. We also aim to understand in what cases memory and fractional derivatives are necessary to model constitutive laws of microstructures in higher dimensions, as this is an open mathematical and experimental problem~\cite{ostoja2018does}. More generally, learning constitutive models that are microstructure dependent now allows us to investigate which microstructure properties lead to these distinct features of homogenized constitutive laws.

%% file: supplement_text.tex
\section{Equivalence between Cell Problems}\label{sec:cell_problem_equiv}
Here, we show that our original cell problem~\cref{eq:cell_problem} can be derived from the classical cell problem of a viscoelastic material~\cite{bhattacharya2023learning, pavliotis2008multiscale} with a periodic boundary when the microstructure $E$ and $\nu$ are periodic functions. The classical periodic cell problem is given by
\begin{subequations}\label{eq:periodic_cell_problem}
    \begin{align}
        -\p_y\big((E(y) + s\nu(y))\p_y\chi(y)\big) &= \p_y(E(y) + s\nu(y)), &\quad y &\in \Omega, \\
         \chi \text{ is 1-periodic}, \quad \int_{\Omega} \chi(y)dy &= 0.
    \end{align}
\end{subequations}
where $s \in \R$ is a Laplace variable which the solution $\chi$ will depend on. Hence, we can write the solution to this cell problem more explicitly as $\chi(y, s)$. Now take any forcing trajectory $\epsbar(t)$ and take its Laplace transform $\widehat{\epsbar}(s)$. Then we can define $\widehat{u}(y, s) = \widehat{\epsbar}(s)(\chi(y, s) + y)$ and rewrite the cell problem above as
\begin{subequations}
    \begin{align}
        -\p_y\widehat{\sigma}(y, s) &= 0, &\quad y &\in \Omega,\\
        \widehat{\sigma}(y, s) &= (E(y) + s\nu(y))\p_y\widehat{u}(y, s), &\quad y &\in \Omega,\\
        \widehat{u}(1, s) &= \widehat{u}(0, s)+\widehat{\epsbar}(s),\\
        \p_y \widehat{u}(1, s) &= \p_y \widehat{u}(0, s),\\
        \int_{\Omega} \widehat{u}(y, s)dy &= \frac{\widehat{\epsbar}(s)}{2}.
    \end{align}
\end{subequations}
Now converting back into the time domain from the Laplace domain for $t \in \cT := [0, T]$ we have
\begin{subequations}
    \begin{align}
        -\p_y\sigma(y, t) &= 0, &\quad (y, t) &\in \Omega \times \cT,\\
        \sigma(y, t) &= E(y)\p_yu(y, t) + \nu(y)\p_{yt}u(y, t), &\quad (y, t) &\in \Omega \times \cT,\\
        u(1, t) &= u(0, t)+\epsbar(t), &t&\in\cT,\\
        \p_y u(1, t) &= \p_y u(0, t), &t&\in\cT,\\
        \int_{\Omega} u(y, t)dy &= \frac{\epsbar(t)}{2}, &t&\in\cT.
    \end{align}
\end{subequations}
Omitting temporarily the last integral constraint, the solution $u(y, t)$ of the PDE above still remains a valid solution for $u(y, t) + c(t)$ for any trajectory $c(t) \in \R$. Hence, we can shift our solution $u(y, t) \to u(y, t) - u(0, t)$ and it will now satisfy the Dirichlet problem
\begin{subequations}
    \begin{align}
        -\p_y\sigma(y, t) &= 0, &\quad (y, t) &\in \Omega \times \cT,\\
        \sigma(y, t) &= E(y)\p_yu(y, t) + \nu(y)\p_{yt}u(y, t), &\quad (y, t) &\in \Omega \times \cT,\\
        u(0, t) &= 0, \quad u(1, t) = \epsbar(t), &t&\in\cT,\\
        \p_y u(1, t) &= \p_y u(0, t), &t&\in\cT,\\
        u(y, 0) &= 0, &y&\in\Omega
    \end{align}
\end{subequations}
which is precisely the cell problem~\cref{eq:cell_problem} we began our discussion from in \cref{subsec:KV_homog}.

\section{Lipschitz Continuity Proofs}\label{sec:Lip}
This section proves several bounds on the solution of the cell problem which we rely on in the proofs of \cref{lem:u1_u2,lem:sigma1_sigma2}. The proof of these bounds follow the proof of Lemma 3.5 in~\cite{bhattacharya2023learning} with the difference that these bounds depend on the minimum and maximum value of the cell problem microstructure $E_{\min}, E_{\max}, \nu_{\min}$, $\nu_{\max}$ \textit{as well as} the maximum absolute value of the strain input and its derivative $\epsbar_{\max}, \dot{\epsbar}_{\max}$. This allows us to uniformly control the variation in cell problem solutions across a class of microstructures and strain inputs. Also, our result derives Lipschitz bounds on the solution of the cell problem whereas Lemma 3.5 in~\cite{bhattacharya2023learning} works at the level of the original multiscale problem which does not suit our purposes here.

\begin{proposition}\label{prop:periodic_bounds}
    Under \cref{assump:E_nu_eps}, for the solutions $p$ of the periodic cell problem~\cref{eq:detrended_cell_problem}, the following bounds hold
    \begin{enumerate}[label=(\alph*)]
        \item $\sup_{t \in \mathcal{T}}\|p\|_{H_0^1, \nu} \leq \cfrac{\nu_{\max}(\dot{\epsbar}_{\max}\nu_{\max} + \epsbar_{\max}E_{\max})}{E_{\min}\sqrt{\nu_{\min}}}$.
        \item $\sup_{t \in \mathcal{T}}\|\p_tp\|_{H_0^1, \nu} \leq \cfrac{(E_{\max}\nu_{\max} + E_{\min}\nu_{\min})(\dot{\epsbar}_{\max}\nu_{\max} + \epsbar_{\max}E_{\max})}{E_{\min}\nu_{\min}^\frac{3}{2}}$.
    \end{enumerate}
\end{proposition}
\begin{proof}
    We show the first bound by choosing a test function $\varphi = p$ and writing the weak form of our periodic cell problem~\cref{eq:wf} as
    \begin{equation}
        q_\nu(\p_t p, p) + q_E(p, p) = -\dot{\epsbar}(t)\langle\nu, \p_y p\rangle - \epsbar(t)\langle E, \p_y p\rangle
    \end{equation}
    and hence by Cauchy-Schwarz and the definition of the weighted $H^1_0$ norm in~\cref{eq:H01_weighted} we get
    \begin{align*}
        \frac{1}{2}\frac{\d}{\dt}\|p\|_{H_0^1, \nu}^2 + \|p\|_{H_0^1, E}^2 \leq \Big(|\dot{\epsbar}(t)|\|\nu\|_{L^2} + |\epsbar(t)|\|E\|_{L^2}\Big)\|p\|_{H_0^1} \leq C\|p\|_{H_0^1}
    \end{align*}
    for the constant $C = \dot{\epsbar}_{\max}\nu_{\max} + \epsbar_{\max}E_{\max}$. Applying \cref{lem:norm_equiv} we have
    \begin{align*}
        \frac{\d}{\dt}\|p\|_{H_0^1, \nu}^2 + 2\frac{E_{\min}}{\nu_{\max}}\|p\|_{H_0^1, \nu}^2 \leq 2\frac{C}{\sqrt{\nu_{\min}}}\|p\|_{H_0^1, \nu}
    \end{align*}
    which by Young's inequality for $\delta > 0$ gives
    \begin{align*}
        \frac{\d}{\dt}\|p\|_{H_0^1, \nu}^2 + 2\frac{E_{\min}}{\nu_{\max}}\|p\|_{H_0^1, \nu}^2 \leq \frac{C^2}{\nu_{\min}\delta^2} + \delta^2\|p\|_{H_0^1, \nu}^2.
    \end{align*}
    Setting $\delta^2 = \frac{E_{\min}}{\nu_{\max}}$ gives us
    \begin{align*}
        \frac{\d}{\dt}\|p\|_{H_0^1, \nu}^2 + \frac{E_{\min}}{\nu_{\max}}\|p\|_{H_0^1, \nu}^2 \leq \frac{C^2\nu_{\max}}{E_{\min}\nu_{\min}}
    \end{align*}
    which by Gronwall's inequality yields
    \begin{equation}\label{eq:p_bound}
        \sup_{t \in \mathcal{T}}\|p\|_{H_0^1, \nu}^2 \leq \Big(\frac{\nu_{\max}}{E_{\min}}\Big)^2\frac{C^2}{\nu_{\min}} = \Big(\frac{\nu_{\max}(\dot{\epsbar}_{\max}\nu_{\max} + \epsbar_{\max}E_{\max})}{E_{\min}\sqrt{\nu_{\min}}}\Big)^2,
    \end{equation}
    so the first bound is proved.

    To prove the second bound, we take a test function $\varphi = \p_tp \in H_0^1(\Omega)$ and write the weak form of the periodic cell problem
    \begin{equation}
        q_\nu(\p_tp, \p_tp) + q_E(\p_tp, p) = -\dot{\epsbar}(t)\langle\nu, \p_{yt}p\rangle - \epsbar(t)\langle E, \p_{yt}p\rangle
    \end{equation}
    which by Cauchy--Schwarz gives us
    \begin{align*}
        \|\p_tp\|_{H_0^1, \nu}^2 \leq \|p\|_{H_0^1, E}\|\p_tp\|_{H_0^1, E} + C\|\p_tp\|_{H_0^1}
    \end{align*}
    for the same constant $C = \dot{\epsbar}_{\max}\nu_{\max} + \epsbar_{\max}E_{\max}$. Then by \cref{lem:norm_equiv} we get
    \begin{align*}
        \|\p_tp\|_{H_0^1, \nu}^2 \leq \Big(\frac{E_{\max}}{\nu_{\min}}\|p\|_{H_0^1, \nu} + \frac{C}{\sqrt{\nu_{\min}}}\Big)\|\p_tp\|_{H_0^1, \nu}
    \end{align*}
    and therefore using our first bound on $\|p\|_{H_0^1, \nu}^2$ in~\cref{eq:p_bound} we get
    \begin{equation*}
        \|\p_tp\|_{H_0^1, \nu} \leq (\dot{\epsbar}_{\max}\nu_{\max} + \epsbar_{\max}E_{\max})\left(\frac{E_{\max}\nu_{\max}}{E_{\min}\nu_{\min}^{3/2}} + \frac{1}{\sqrt{\nu_{\min}}}\right).
    \end{equation*}
\end{proof}

Now recall that a solution $u(y, t)$ to the original cell problem~\cref{eq:cell_problem} can be decomposed into its periodic and nonperiodic parts as $u(y, t) = p(y, t) + \epsbar(t)y$ where $p$ is a solution to the periodic cell problem~\cref{eq:detrended_cell_problem}. This allows us to bound by the triangle inequality and \cref{lem:norm_equiv}
\begin{equation}
\begin{gathered}
    \sup_{t \in \mathcal{T}}\|u\|_{H_0^1, \nu} \leq \sup_{t \in \mathcal{T}}\|p\|_{H_0^1, \nu} + \nu_{\max}\sup_{t \in \mathcal{T}}|\epsbar(t)|,\\
    \sup_{t \in \mathcal{T}}\|\p_tu\|_{H_0^1, \nu} \leq \sup_{t \in \mathcal{T}}\|\p_tp\|_{H_0^1, \nu} + \nu_{\max}\sup_{t \in \mathcal{T}}|\dot{\epsbar}(t)|.
\end{gathered}
\end{equation}
Combining these bounds with \cref{prop:periodic_bounds} immediately leads to the following corollary.
\begin{corollary}\label{cor:u_bounds}
    Under \cref{assump:E_nu_eps}, for the solution $u$ of the cell\linebreak problem~\cref{eq:cell_problem}, the following bounds hold
    \begin{enumerate}[label=(\alph*)]
        \item $\sup_{t \in \mathcal{T}}\|u\|_{H_0^1, \nu} \leq \cfrac{\nu_{\max}(\dot{\epsbar}_{\max}\nu_{\max} + \epsbar_{\max}E_{\max})}{E_{\min}\nu_{\min}} + \nu_{\max}\epsbar_{\max}$
        \item $\sup_{t \in \mathcal{T}}\|\p_tu\|_{H_0^1, \nu} \leq \tfrac{(E_{\max}\nu_{\max} + E_{\min}\nu_{\min})(\dot{\epsbar}_{\max}\nu_{\max} + \epsbar_{\max}E_{\max})}{E_{\min}\nu_{\min}^\frac{3}{2}} + \nu_{\max}\dot{\epsbar}_{\max}.$
    \end{enumerate}
\end{corollary}

\begin{proposition}\label{prop:gamma_bounds}
    Under \cref{assump:E_nu_eps}, for all solutions $\gamma$ of~\cref{eq:gamma_weak}, the following bounds hold
    \begin{enumerate}[label=(\alph*)]
        \item $\sup_{t \in \mathcal{T}}\|\gamma\|_{H_0^1, \nu_1} \leq \cfrac{\nu_{\max}}{E_{\min}\sqrt{\nu_{\min}}}\|g\|_\cZ$
        \item $\sup_{t \in \mathcal{T}}\|\p_t\gamma\|_{H_0^1, \nu_1} \leq \cfrac{E_{\max}\nu_{\max} + E_{\min}\nu_{\min}}{E_{\min}\nu_{\min}^\frac{3}{2}}\|g\|_\cZ$.
    \end{enumerate}
\end{proposition}
\begin{proof}
Choosing the test function $\varphi = \gamma$ which by definition is zero on the boundary $\p\Omega$, we can write the weak form of the PDE~\cref{eq:gamma_weak} as
\begin{align*}
    q_{\nu_1}(\p_t\gamma, \gamma) + q_{E_1}(\gamma, \gamma) = -\langle g, \p_y\gamma\rangle.
\end{align*}
Now using Cauchy-Schwarz we can write
\begin{align*}
    \frac{1}{2}\frac{\d}{\dt}\|\gamma\|_{H_0^1, \nu_1}^2 + \|\gamma\|_{H_0^1, E_1}^2 \leq \|g\|\|\gamma\|_{H_0^1}
\end{align*}
and applying \cref{lem:norm_equiv} gives us
\begin{align*}
    \frac{1}{2}\frac{\d}{\dt}\|\gamma\|_{H_0^1, \nu_1}^2 + \frac{E_{\min}}{\nu_{\max}}\|\gamma\|_{H_0^1, \nu_1}^2 \leq \frac{1}{\sqrt{\nu_{\min}}}\|g\|\|\gamma\|_{H_0^1, \nu_1}
\end{align*}
and applying Young's inequality for any $\delta > 0$ gives us
\begin{align*}
    \frac{1}{2}\frac{\d}{\dt}\|\gamma\|_{H_0^1, \nu_1}^2 + \frac{E_{\min}}{\nu_{\max}}\|\gamma\|_{H_0^1, \nu_1}^2 \leq \frac{1}{2\delta^2}\|g\|^2 + \frac{\delta^2}{2\nu_{\min}}\|\gamma\|_{H_0^1, \nu_1}.
\end{align*}
Now setting $\delta^2 = \frac{E_{\min}\nu_{\min}}{\nu_{\max}}$ we have
\begin{align*}
    \frac{\d}{\dt}\|\gamma\|_{H_0^1, \nu_1}^2 + \frac{E_{\min}}{\nu_{\max}}\|\gamma\|_{H_0^1, \nu_1}^2 \leq \frac{\nu_{\max}}{E_{\min}\nu_{\min}}\|g\|_\cZ^2
\end{align*}
Note that $\gamma(y, 0) = 0$ since $u_1(y, 0) = u_2(y, 0) = 0$. Hence, by Gronwall's inequality we get that
\begin{equation}\label{eq:gamma_bound}
    \sup_{t \in \mathcal{T}}\|\gamma\|_{H_0^1, \nu_1}^2 \leq \Big(\frac{\nu_{\max}}{E_{\min}}\Big)^2\frac{1}{\nu_{\min}}\|g\|_\cZ^2
\end{equation}
which proves the first bound.

To prove the second bound, we substitute $\varphi = \p_t\gamma$ into the weak form of the PDE to get
\begin{align*}
    q_{\nu_1}(\p_t\gamma, \p_t\gamma) + q_{E_1}(\gamma, \p_t\gamma) = -\langle g, \p_{yt}\gamma\rangle.
\end{align*}
Now rearranging terms and using Cauchy-Schwarz we can write
\begin{align*}
    \|\p_t\gamma\|_{H_0^1, \nu_1}^2 \leq \|g\|\|\p_t\gamma\|_{H_0^1} + \|\gamma\|_{H_0^1, E_1}\|\p_t\gamma\|_{H_0^1, E_1}.
\end{align*}
Applying \cref{lem:norm_equiv} gives us
\begin{align*}
    \|\p_t\gamma\|_{H_0^1, \nu_1}^2 \leq \frac{1}{\sqrt{\nu_{\min}}}\|g\|\|\p_t\gamma\|_{H_0^1, \nu_1} + \frac{E_{\max}}{\nu_{\min}}\|\gamma\|_{H_0^1, \nu_1}\|\p_t\gamma\|_{H_0^1, \nu_1}.
\end{align*}
Dividing by $\|\p_t\gamma\|_{H_0^1, \nu_1}$ on both sides we get
\begin{align*}
    \|\p_t\gamma\|_{H_0^1, \nu_1} \leq \frac{1}{\sqrt{\nu_{\min}}}\|g\| + \frac{E_{\max}}{\nu_{\min}}\|\gamma\|_{H_0^1, \nu_1}
\end{align*}
and applying our first bound in~\cref{eq:gamma_bound} gives us
\begin{equation}
    \|\p_t\gamma\|_{H_0^1, \nu_1} \leq \Big(\frac{E_{\max}\nu_{\max}}{E_{\min}\nu_{\min}^\frac{3}{2}} + \frac{1}{\sqrt{\nu_{\min}}}\Big)\|g\|_\cZ.
\end{equation}
\end{proof}

\section{Approximation Proofs}
\label{sec:UA}

This lemma is used to show, in one dimension, that functions that are integrable and of bounded variation are approximable by piecewise constants. This is a classical and well-known result which we derive here for convenience in the one dimensional setting of our theory, and can be extended to higher dimensions by using the equivalence between functions of bounded variation and Lipschitz-1 functions~\cite[Lemma C.1]{bhattacharya2024learning}.
\begin{lemma}\label{lem:bv_l1_approx}
    For the domain $\Omega = [0, 1]$  take any integrable function of bounded variation $f \in L^1(\Omega) \cap \BV(\Omega)$. Then defining a grid $y_i = \frac{i}{M}$ for $i \in \{0, \hdots, M\}$ there exists a piecewise constant function $f: [0, 1] \to \mathbb{R}$ with $L$ pieces given by
    \begin{equation}
        f_{\pc}(y) = f_{\pc}^{i} := M\int_{y_i}^{y_{i+1}}f(z)\dz, \quad y \in [y_i, y_{i+1}]
    \end{equation}
    satisfying the approximation bound
    \begin{equation}
        \|f - f_{\pc}\|_{L^1} \leq \frac{|f|_{\BV}}{M}.
    \end{equation}
\end{lemma}
\begin{proof}
    We begin by noting that
    \begin{align*}
        \|f - f_{\pc}\|_{L^1} = \sum_{i=0}^{M-1}\int_{y_i}^{y_{i+1}}|f(z) - f_{\pc}^i|\dz.
    \end{align*}
    Studying one of the terms in the sum we can bound
    \begin{align*}
        \sup_{y \in [y_i, y_{i+1}]}|f(y) - f_{\pc}^i| &= \sup_{y \in [y_i, y_{i+1}]}M\Big|\int_{y_i}^{y_{i+1}}(f(y) - f(z))\dz\Big|\\
        &\leq \sup_{y \in [y_i, y_{i+1}]}M\int_{y_i}^{y_{i+1}}|f(y) - f(z)|\dz\\
        &\leq \sup_{y, y' \in [y_i, y_{i+1}]}|f(y) - f(y')|.
    \end{align*}
    Now suppose that for each of the $M$ intervals $[y_i, y_{i+1}]$ we choose two points $y, y' \in [y_i, y_{i+1}]$ and form a partition from the union of all of these points. This partition now has $2M$ points if we also include the endpoints $y = 0, 1$. This implies that
    \begin{align*}
        \|f - f_{\pc}\|_{L^1} &\leq \frac{1}{M}\sum_{i=0}^{M-1}\sup_{y, y' \in [y_i, y_{i+1}]}|f(y) - f(y')|\\
        &\leq \frac{1}{M}\sup\Big\{\sum_{i=0}^{2M-1}|f(z_{i+1}) - f(z_i)| \ \Big| \ 0 = z_0 < z_1 < \hdots < z_{2M} = 1\Big\}\\
        &\leq \frac{|f|_{\BV}}{M}
    \end{align*}
    where the last line is given by the definition of the total variation norm in~\cref{eq:tv_norm}.
\end{proof}

The following two propositions are critical to the RNO approximation result. The first is a general universal approximation result for FNMs with both finite and infinite-dimensional inputs, and the second applies this general result to the constitutive map of interest in this work.

\begin{proposition}\label{prop:UA-FNM}
    Consider a bounded set $K_v \subset \R^{d_{\inn}^{v}}$ and a compact set $K_f \subset L^2(\Td;\R^{d_{\inn}^f})$. Let $\Phi^\dagger: K_v \times K_f \mapsto \R^{d_{\out}^v}$ be continuous. Then for any $\e > 0$, there exists an $\FNM$ $\Phi: K_v \times K_f \mapsto \R^{d_{\out}^{v}}$ of the form 
    \begin{equation}\label{eqn:FNM-form}
        \Phi = Q_v \circ \sG \circ \sL_T \circ \dots \circ \sL_1 \circ S_f \circ (\sD \circ S_v, I_{d_{\inn}^f}),
    \end{equation}
    for some $T \in \N_{>0}$ where $S_v$ acts on the input from $K_v$ and $I_{d_{\inn}^f}$ acts on the input from $K_f$, such that 
    \begin{equation}\label{eqn:FNM_form} \sup_{u \in K_v \times K_f} \|\Phi^\dagger(u) -\Phi(u)\|\leq \e.\end{equation}
    The layers in \cref{eqn:FNM-form} take the form of their homonymous counterparts in \cref{def:FNM}.
\end{proposition}

This proposition closely follows those of Theorems 3.2 and 3.3 in~\cite{huang2024operator}, but we state a proof here for completeness. We adapt the proofs of these theorems to our setting because our FNM architecture is novel in the sense that it accepts vector and function inputs jointly as opposed to purely vector or purely function inputs discussed in~\cite{huang2024operator}. The workhorse of this theoretical result is the Dugundji extension theorem that allows us to extend an operator acting on a compact subset of $L^2$ to all of $L^2$.

\begin{proof}
   Let $\one(x) = 1$ be the constant function and define the vector to function map $\L: \R^{d_{\inn}^{v}} \to L^2(\Td;\R^{d_{\inn}^{v}})$ given by $\L(z) \mapsto z\one$. Note that clearly $\|\L z\|_{L^2(\Td;\R^{d_{\inn}^{v}})} = \|z\|$, so $\L$ is continuous. Let $K_{vf}: = \{\L z: z \in K_v\}$ and note that this set is compact since $K_v$ is compact and $\L$ is continuous. Define $K = K_{vf} \times K_f \subset L^2(\Td; \R^{d_{\inn}^{v} + d_{\inn}^f})$ by $K_{vf} \times K_f$ which is also a compact set. Define $\Phi_{fv}^\dagger: K \mapsto \R^{d_{\out}^{v}}$ by $\L z \times f \mapsto \Phi^\dagger(z,f)$. 

    Lastly, define $\Phi_{ff}^\dagger: K \mapsto L^2(\Td; \R^{d_{\out}^{v}})$ by the map $z \mapsto \Phi_{fv}^\dagger(z)\one$. We first show that $\Phi_{ff}^\dagger$ is continuous. 
    \begin{align*}
        \|\Phi^\dagger_{ff}((Lz, f))\|^2_{\cY} & = \int_{\Td} \|\Phi^\dagger_{ff}((Lz, f))\|_{\R^{d_{\out}^{v}}}^2 \dx \\
        & = \int_{\Td}\|\Phi^\dagger(z,f)\one\|_{\R^{d_{\out}^{v}}}^2 \dx \\
        & = \|\Phi^\dagger(z,f)\|_{\R^{d_{\out}^{v}}}^2 \\
        & = \left\|\Phi^\dagger\left(\int_{\Td}Lz \; \dx, f\right)\right\|^2_{\R^{d_{\out}^{v}}}.
    \end{align*}
    Since the averaging operator $\int_{\Td} \cdot \dx$ over the torus is continuous and $\Phi^{\dagger}$ is also continuous, $\Phi^\dagger_{ff}$ is continuous from $K$ to $L^2(\Td; \R^{d_{\out}^{v}})$. 
    By the Dugundji extension theorem, there exists a continuous operator $\tilde{\Phi}^\dagger: L^2(\Td; \R^{d_{\inn}^{v} + d_{\inn}^f}) \to L^2(\Td; \R^{d_{\out}^{v}})$ such that $\tilde{\Phi}^\dagger(u) = \Phi_{ff}^\dagger(u)$ for every $u \in K$. By Theorem 9 of \cite{kovachki2021universal}, for any $\e$, there exists an FNO $\widehat{\Phi}$ of the form 
    \begin{equation}\label{eqn:FNOform}
        \widehat{\Phi} = \widehat{Q}_f \circ \widehat{\sL}_T \circ \dots \circ \widehat{\sL}_1 \circ \widehat{S}_f
    \end{equation}
    such that 
    \begin{equation}\label{eqn:FNOapprox}
    \sup_{a \in K} \|\tilde{\Phi}^\dagger(a)- \widehat{\Phi}(a)\|_{L^2(\Td;\R^{d_{\out}^{v}})} = \sup_{a \in K} \|\Phi_{ff}^\dagger(a)- \widehat{\Phi}(a)\|_{L^2(\Td;\R^{d_{\out}^{v}})} < \e.
    \end{equation}
    Specifically, in \cref{eqn:FNOform} we have that $\widehat{S}_f: \R^{d_{\inn}^{v}+d_{\inn}^f} \to \R^{d_0}$, and $\widehat{Q}_f: \R^{d_T} \to \R^{d_{\out}^{v}}$. 

    Next we define the following FNM architecture
    \begin{equation}
        \Phi = Q_v \circ \sG \circ \sL_T \circ \dots \sL_1 \circ S_f \circ (\sD\circ S_v, I_{d_{\inn}^f})
    \end{equation}
    where we will define its layers accordingly to be equivalent to the FNO architecture in~\cref{eqn:FNOform}. The layers are given by 
    \begin{align*}
        S_f & = \widehat{S}_f\\
        S_v & = I_{d_{\inn}^{v}}\\
        z\mapsto \sD z&= \one z \\
        u\mapsto \sL_1 u(x) & = \sigma\left(W_1u(x) + \sum_{k \in \Z^d}\left(\sum_{j=1}^{d_0}\left(P_1^{(k)}\right)_j \langle \psi_k, u_j \rangle_{L^2(\Td;\C)}\right)\psi_k(x) + b_1\right)\\
        & \text{where } W_1 = \widehat{W}_1, \; P^{(k)}_1 = \widehat{P}^{(k)}_1 \text{ so in other words }\\
        \sL_t & = \widehat{\sL_t}, \ t = \{1, \dots, T\}\\
        z\mapsto \sG z& = \int_{\Td}\kappa_f(x) z(x) \dx \text{ where } \kappa_f = \widehat{Q}_f\one \\
        Q_v & = I_{d_{\out}^{v}}.
    \end{align*}
In the preceding display, $\widehat{W}_1$ and $\widehat{P}_1$ are the associated coefficients in $\widehat{\sL}_1$ of the FNO in \cref{eqn:FNOform}. One can check by this construction that $\int_{\Td}\widehat{\Phi}(Lv,f) \dx = \Phi(v,f)$. Finally, this allows us to show that
\begin{align*}
    \sup_{v,f \in K_v \times K_f}&\|\Phi^\dagger(v,f) - \Phi(v,f)\|_{\R^{d_{\out}^{v}}} \\ &\leq \sup_{v,f \in K_v \times K_f} \left\|\Phi^\dagger(v,f) - \int_{\Td} \Phi^\dagger_{ff}(Lv,f)\dx\right\|_{\R^{d_{\out}^{v}}}\\
    &\qquad+ \left\| \int_{\Td}\Big(\Phi^\dagger_{ff}(Lv,f) - \widehat{\Phi}(Lv,f)\Big)\dx\right\|_{\R^{d_{\out}^{v}}}\\
    &\qquad+ \left\| \int_{\Td}\widehat{\Phi}(Lv,f) \dx - \Phi(v,f) \right\|_{\R^{d_{\out}^{v}}}\\
    & \leq \e.
\end{align*}
The transition from the second to third line holds because $\int_{\Td} \Phi^\dagger_{ff}(Lv,f)\dx$ \allowbreak$=$\allowbreak $\Phi^\dagger(v,f)$, by the approximation result in~\cref{eqn:FNOapprox}, and by the fact that $\int_{\Td}\widehat{\Phi}(Lv,f) \dx$\allowbreak$=$\allowbreak$\Phi(v,f)$. Since $\e$ was arbitrary, the lemma is proven. 
\end{proof}

\begin{proposition}\label{prop:UA-hom}
    Under \cref{assump:E_nu_eps}, for any $\epsbar_{\max}, \dot{\epsbar}_{\max}, \xi_{\max} > 0$ and\linebreak$\e_F, \e_G > 0$, there exist FNMs $F_{\FNM}$ and $G_{\FNM}$, in 
    \cref{eqn:PsiRNO}, that approximate $F_{\pc}$ and $G_{\pc}$ of \cref{eqn:FGpc} such that 
    \begin{gather}
        \sup_{\substack{|z_1| \leq \epsbar_{\max}, |z_2| \leq \dot{\epsbar}_{\max}, \|z_3\|\leq \xi_{\max},\\ E \in \cM_{E_{\min}, E_{\max}}^B, \nu \in \cM_{\nu_{\min}, \nu_{\max}}^B}} |F_{\FNM}(z_1, z_2, z_3; E, \nu) - F_{\pc}(z_1, z_2, z_3; E, \nu)| < \e_F\\
        \sup_{\substack{|z_1| \leq \epsbar_{\max}, \|z_3\|\leq \xi_{\max},\\ E \in \cM_{E_{\min}, E_{\max}}^B, \nu \in \cM_{\nu_{\min}, \nu_{\max}}^B}} \|G_{\FNM}(z_1, z_3; E, \nu) - G_{\pc}(z_1, z_3; E, \nu)\| < \e_G.
    \end{gather}
\end{proposition}
\begin{proof}
The proof is a simple application of \cref{prop:UA-FNM}. The function inputs $E$ and $\nu$ are on the same domain $\Td$ and may have their outputs concatenated to form a single function input set $K_f$ consistent with the statement of \cref{prop:UA-FNM}. Since $\cM_{E_{\min}, E_{\max}}^B$ and $\cM_{\nu_{\min}, \nu_{\max}}^B$ are compact in $L^2$ due to the embedding $\BV(\T) \cap L^{\infty}(\T) \hookrightarrow L^2(\T)$, the set of input functions is a compact set. For a proof of this embedding result, see \cite[Lemma C.1]{bhattacharya2024learning}. Similarly, since all the finite inputs are bounded, their product set is also a compact set $K_v$ consistent with \cref{prop:UA-FNM}. Applying the lemma gives the result for both $F_{\FNM}$ and $G_{\FNM}$ to arbitrary accuracy $\e_F$ and $\e_G$. 
\end{proof}

The following assumptions are necessary to derive the Lipschitz constant of the FNM in \cref{lem:FNO_lip}. 
\begin{assumptions} \label{ass:FNM-lip}
    We assume 
    \begin{enumerate}
        \item The activation $\sigma$ is $B-$Lipschitz.
        \item $\|P_t\|_{\infty}$ and $ \|W_t\|_{\infty}$ are bounded.
        \item $\Big(\sum_{k \in \Zd}\Big\|P_v^{(k)}\Big\|_F^2\Big)^\frac{1}{2}$ and $\Big(\sum_{k \in \Zd}\Big\|P_f^{(k)}\Big\|_F^2\Big)^\frac{1}{2}$ are bounded.
        \item $S_f, S_v, Q_f$ and $Q_v$ are feedforward neural nets with activation $\sigma$, bounded network weights, fixed maximum width, and fixed number of layers.
    \end{enumerate}
\end{assumptions}
\begin{lemma}\label{lem:FNO_lip}
   Under Assumptions \cref{ass:FNM-lip} and using notation from Definition \cref{def:FNM}, a Fourier Neural Mapping $\Phi$ of the form \[ Q_v \circ \sG \circ \sL_T \circ \dots \circ \sL_1 \circ S_f \circ (\sD \circ S_v, I_{d_{\inn}^f}) \] is Lipschitz in the vector inputs: there exists some $C>0$ such that for vector inputs $v_1$ and $v_2$ and function input $f$, 
   \begin{equation*}
       \|\Phi(v_1, f) - \Phi(v_2, f)\|_2 \leq C \|v_1 - v_2\|_2. 
   \end{equation*}
\end{lemma}
\begin{proof}~
\begin{enumerate}[label=(\roman*)]
    \item Claim: The Fourier layers $\sL_t: L^2(\Td; \R^{d_{t-1}}) \to L^2(\Td; \R^{d_t})$ for $t \in [T]$ are Lipschitz

    Take any $u_1, u_2 \in L^2(\Td; \R^{d_{t-1}})$ and write
    \begin{align*}
        &\|\sL_t(u_1) - \sL_t(u_2)\|_{L^2(\Td; \R^{d_t})}\\
        =& \Big\|\sigma\Big(W_t u_1 + (\cK_t u_1) + b_t\Big) - \sigma\Big(W_t u_2 + (\cK_t u_2) + b_t\Big)\Big\|_{L^2}\\
        \leq& B\Big\|W_t (u_1 - u_2) + (\cK_t (u_1 - u_2))\Big\|_{L^2}\\
        \leq& B\|W_t\|_\infty\sqrt{d_{t-1}d_t}\|u_1 - u_2\|_{L^2} + \|\cK_t (u_1 - u_2)\|_{L^2}.
    \end{align*}
    Now to bound the second term above we let $u = u_1 - u_2$ and write
    \begin{align*}
        \|\cK_t u\|_{L^2(\Td; \R^{d_t})}^2 &= \sum_{l=1}^{d_t}\Big\|\sum_{k \in \Zd}\Big(\sum_{j=1}^{d_{t-1}} (P^{(k)}_t)_{\ell j}\ip{\psi_k}{u_j}_{L^2(\Td;\C)} \Big)\psi_k(x)\Big\|_{L^2(\Td)}^2\\
        &= \sum_{l=1}^{d_t}\sum_{k \in \Zd}\Big(\sum_{j=1}^{d_{t-1}} (P^{(k)}_t)_{\ell j}\ip{\psi_k}{u_j}_{L^2(\Td;\C)} \Big)^2\\
        &\leq \|P_t\|_\infty^2d_{t-1}d_t\sum_{k \in \Zd}\sum_{j=1}^{d_{t-1}} \ip{\psi_k}{u_j}_{L^2(\Td;\C)}^2\\
        &= \|P_t\|_\infty^2d_{t-1}d_t\|u\|_{L^2(\Td;\R_{d_{t-1}})}^2.
    \end{align*}
    Combining everything together we get
    \begin{equation}
    \begin{aligned}
        \|\sL_t(u_1) - \sL_t(u_2)\|_{L^2(\Td; \R^{d_t})} &\leq \big(B\|W_t\|_\infty + \|P_t\|_\infty\big)\sqrt{d_{t-1}d_t} \\
        &\qquad\times\|u_1 - u_2\|_{L^2(\Td;\R_{d_{t-1}})}
    \end{aligned}
    \end{equation}
    
    \item Claim: The lifting and projection layers $S_v$, $S_f$, and $Q_v$ are Lipschitz. 

    Each of these layers is a feed-forward neural network with Lipschitz activation functions and hence it is Lipschitz by composition.
    
    \item The vector to function layer $\sD: \R^{d_{\lift}^v} \to L^2(\Td; \R^{d_{\lift}^{vf}})$ is Lipschitz.

    For any $z \in \R^{d_{\lift}^v}$ we write out
    \begin{align*}
        \|\sD z\|_{L^2(\Td; \R^{d_{\lift}^{vf}})}^2 &= \sum_{j=1}^{d_{\lift}^{vf}}\Big\|\sum_{k \in \Zd} \left(P_v^{(k)}z\right)_j \psi_k\Big\|_{L^2(\Td)}^2 = \sum_{j=1}^{d_{\lift}^{vf}}\sum_{k \in \Zd}\left(P_v^{(k)}z\right)_j^2\\
        &= \sum_{k \in \Zd}\Big\|P_v^{(k)}z\Big\|_2^2 \leq \Big(\sum_{k \in \Zd}\Big\|P_v^{(k)}\Big\|_F^2\Big)\|z\|_2^2
    \end{align*}
    where the second line follows by an application of Cauchy-Schwarz. So for any $z_1, z_2 \in \R^{d_{\lift}^v}$ we have that
    \begin{equation}
        \|\sD(z_1 - z_2)\|_{L^2(\Td; \R^{d_{\lift}^{vf}})} \leq \Big(\sum_{k \in \Zd}\Big\|P_v^{(k)}\Big\|_F^2\Big)^\frac{1}{2}\|z_1 - z_2\|_2.
    \end{equation}

    \item The function to vector layer $\sG: L^2(\Td; \R^{d_T}) \to \R^{d_{\proj}^{fv}}$ is Lipschitz.

    For any $h \in L^2(\Td; \R^{d_T})$ we can write
    \begin{align*}
        \|\sG h\|_{L^2} \leq \|\sG h\|_{L^1} &= \sum_{l=1}^{d_{\proj}^{fv}}\Big|\sum_{k \in \Zd}\sum_{j=1}^{d_T} (P^{(k)}_f)_{\ell j} \ip{\psi_k}{h_j}_{L^2(\Td;\C)}\Big|\\
        &\leq \sum_{l=1}^{d_{\proj}^{fv}}\sum_{j=1}^{d_T} \Big|\sum_{k \in \Zd}(P^{(k)}_f)_{\ell j} \ip{\psi_k}{h_j}_{L^2(\Td;\C)}\Big|\\
        &\leq \sum_{l=1}^{d_{\proj}^{fv}}\sum_{j=1}^{d_T} \Big(\sum_{k \in \Zd}(P^{(k)}_f)_{\ell j}^2\Big)^\frac{1}{2} \Big(\sum_{k \in \Zd}\ip{\psi_k}{h_j}_{L^2(\Td;\C)}\Big)^\frac{1}{2}\\
        &= \sum_{l=1}^{d_{\proj}^{fv}}\sum_{j=1}^{d_T} \Big(\sum_{k \in \Zd}(P^{(k)}_f)_{\ell j}^2\Big)^\frac{1}{2}\|h_j\|_{L^2(\Td;\C)}\\
        &\leq \sum_{l=1}^{d_{\proj}^{fv}}\Big(\sum_{j=1}^{d_T} \sum_{k \in \Zd}(P^{(k)}_f)_{\ell j}^2\Big)^\frac{1}{2}\|h\|_{L^2(\Td;\R^{d_T})}\\
        &\leq \sqrt{d_{\proj}^{fv}}\Big(\sum_{k \in \Zd}\|P^{(k)}_f\|_F^2\Big)^\frac{1}{2}\|h\|_{L^2(\Td;\R^{d_T})}
    \end{align*}
    where the third, fifth, and sixth lines above follow from an application of Cauchy-Schwarz.
\end{enumerate}
\end{proof}

\section{Loss Function Penalty Term for Viscoelastic FNM--RNO}\label{appsec:penalty}

To demonstrate the necessity of the penalty term for the constraint\linebreak$\|G_{\text{FNM}}(0, 0; E,\nu)\|$\allowbreak$=$\allowbreak$0$, we train FNM--RNO with varying numbers of internal variables with and without the penalty term in the loss function \cref{eq:loss_function}. The FNM--RNO trained without penalty achieves a slightly smaller relative $L^2$ testing error on both datasets, while the relative $L^{\infty}$ testing error is much larger. In \cref{fig:l_infity_error_distribution}, we visualize the distributions of their relative $L^{\infty}$ testing error on the PC and HMC dataset. For the PC dataset, the relative $L^{\infty}$ prediction error of the FNM--RNO trained without penalty is, on average, much larger than that of the linear stress response without memory effects. For the HMC dataset, this discrepancy is less pronounced. On the other hand, the FNM--RNO trained with penalty typically achieves smaller relative $L^{\infty}$ testing errors on average compared to linear stress response without memory effects, except when the number of internal variables is large. 

In \cref{fig:sample_wo_penalty}, we visualize the FNM--RNO predictions at the sample in the PC testing datasets with large relative $L^{\infty}$ error when the FNO--RNO uses 5 internal variables. When the FNM--RNO is trained without penalty, its averaged stress prediction at the initial time has a large error, while the prediction by the FNM--RNO trained with penalty has no visible error. Rates of change of the internal variables are large at the initial time when trained without penalty, and the constraint $\|G_{\text{FNM}}(0, 0; E, \nu)\|=0$ is violated. This behavior is consistent when testing on the HMC dataset and when different FNO--RNO architectures are used  (e.g., different numbers of internal variables, channels, and Fourier modes).

\begin{figure}[htb]
    \centering
    \scalebox{0.9}{
    \begin{tabular}{c}
        \hspace{0.08\linewidth}\bf Relative $L^{\infty}$ testing error on the PC dataset \\ 
        \includegraphics[width=0.75\linewidth]{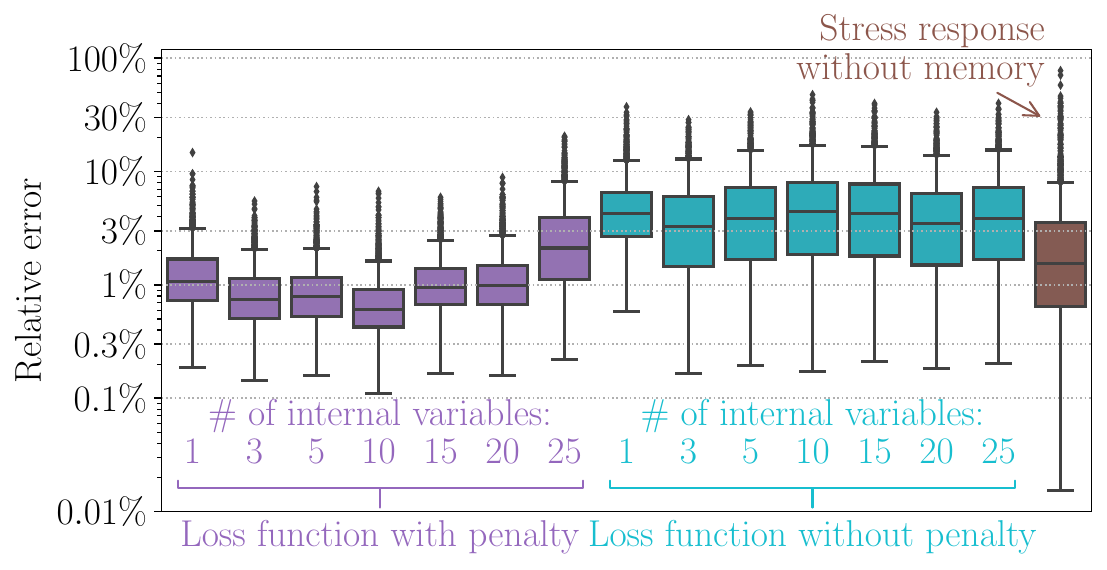} \vspace{0.03\linewidth}\\
        \hspace{0.08\linewidth} \bf Relative $L^{\infty}$ testing error on the HMC dataset \\
        \includegraphics[width=0.75\linewidth]{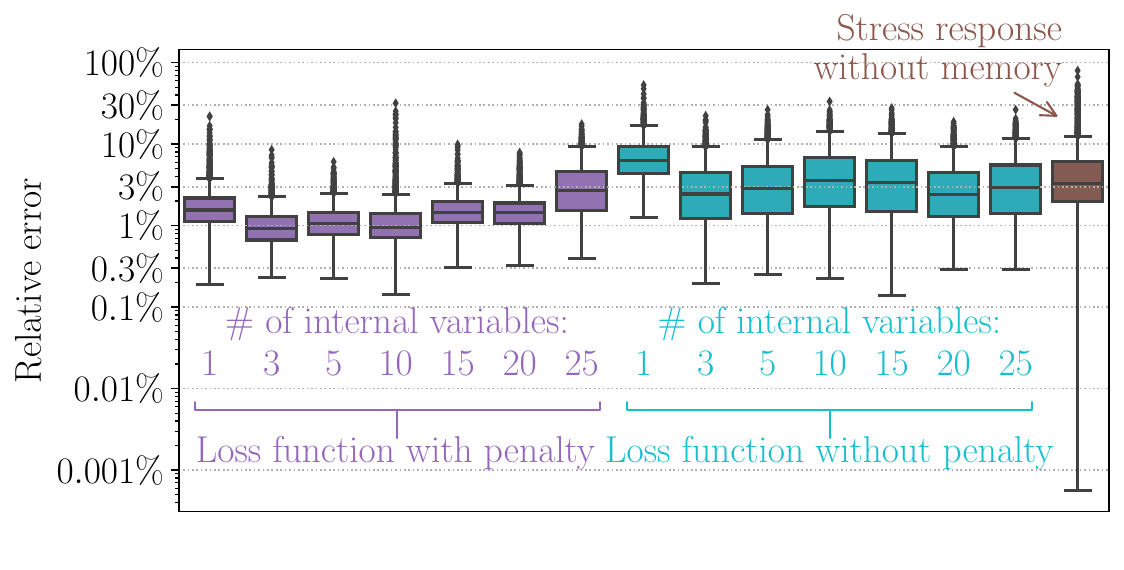}
    \end{tabular}
    }
    \caption{The distributions of the relative $L^{\infty}$ error on 2,500 testing samples from the PC dataset (\textit{top}) and the HMC dataset (\textit{bottom}). We visualize the errors in the FNM--RNO predictions where FNM--RNOs (i) are trained with or without the penalty term in \cref{eq:loss_function}, and (ii) have a varying number of internal variables. We also visualize the distribution of error given by the linear stress response without memory effects, where the response function is obtained using \cref{eqn:kernelform} with $K(t)=0$ for all $t\in[0,1]$.}
    \label{fig:l_infity_error_distribution}
\end{figure}
\newpage

\begin{figure}[htb]
    \centering
    \scalebox{0.76}{
    \addtolength{\tabcolsep}{-6pt}
    \begin{tabular}{H C I I}
        \hspace{0.2\linewidth}\makecell{\bf Testing sample\\\bf inputs}& & \hspace{0.17\linewidth}\makecell{\bf FNM--RNO trained \\\bf without penalty} & \hspace{0.17\linewidth} \makecell{\bf FNM--RNO trained\\\bf with penalty}  \\
         \multirow{3}{*}{\includegraphics[width=\linewidth]{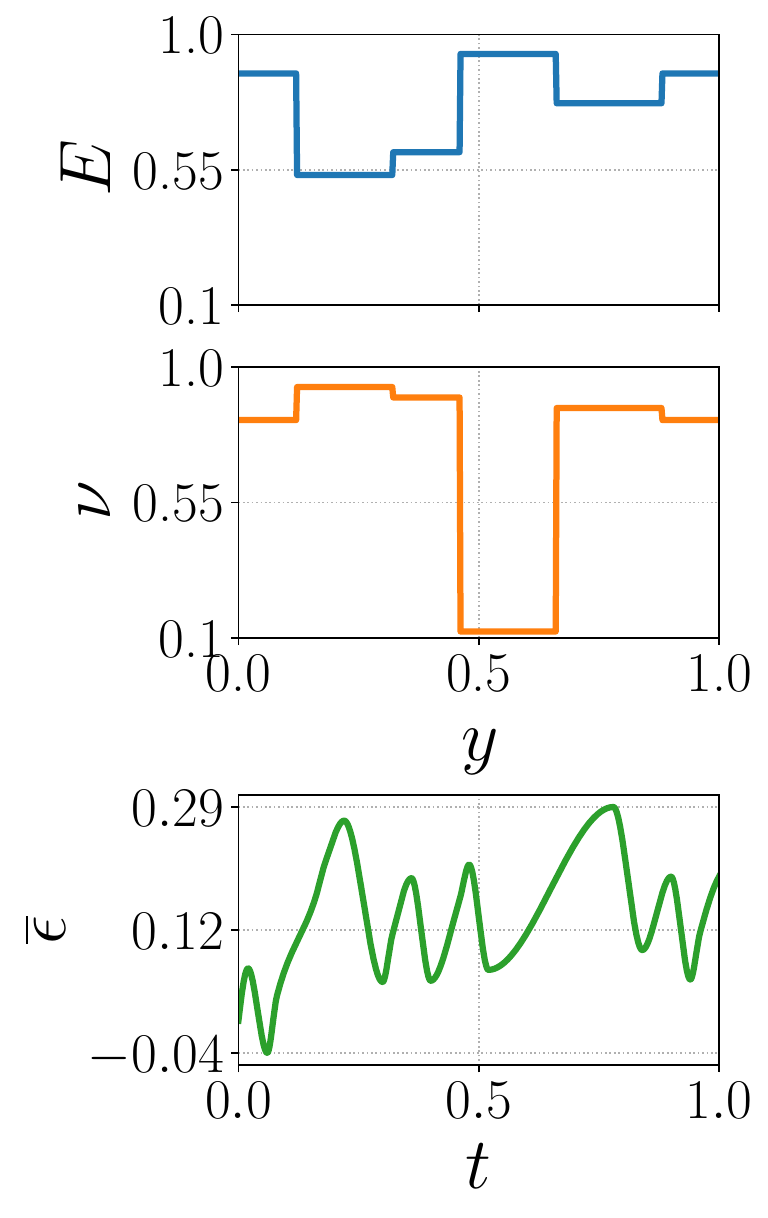}}  & & \multirow{4}{*}{\includegraphics[width=\linewidth]{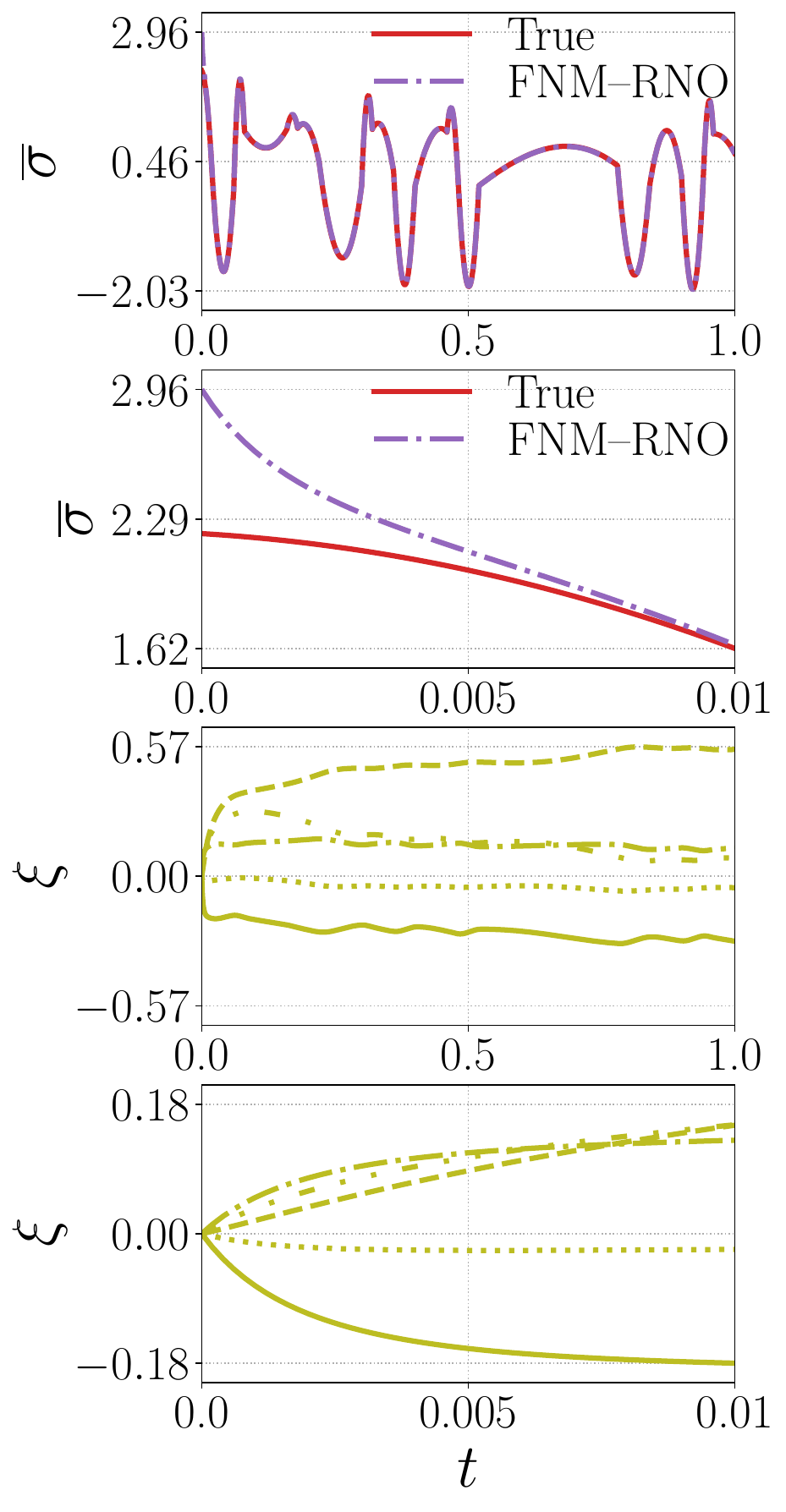}} & \multirow{4}{*}{\includegraphics[width=\linewidth]{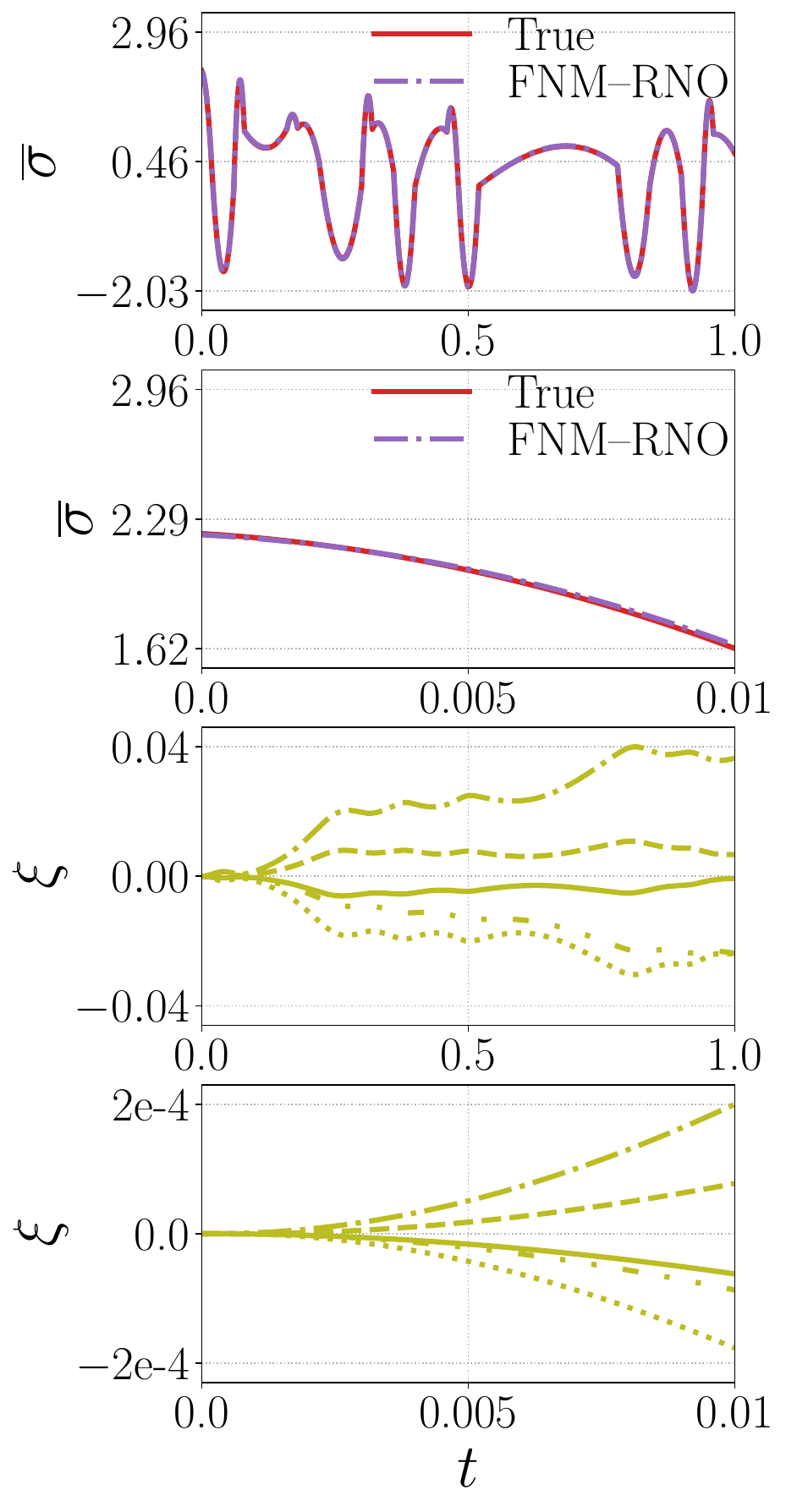}} \\[-0.02\linewidth]
         & \rotatebox{90}{\makecell{\bf Full\\\bf  trajectories}} & &\\[0.1\linewidth]
         & \rotatebox{90}{\makecell{\bf Enlarged\\\bf  view}} & &\\[0.08\linewidth]
         & \rotatebox{90}{\makecell{\bf Full\\\bf  trajectories}} & &\\[0.1\linewidth]
            & \rotatebox{90}{\makecell{\bf Enlarged\\\bf view}} & &\\[0.15\linewidth]
    \end{tabular}
    \addtolength{\tabcolsep}{-6pt}
    }
    \caption{Visualization of FNM--RNO predictions at the sample from the PC testing dataset with the largest relative $L^{\infty}$ testing error (5 internal variables and without penalty in \cref{fig:l_infity_error_distribution}). We compare the predictions by FNM--RNOs trained using the loss function in \cref{eq:loss_function} without the penalty term and with the penalty term. For the averaged stress and internal variable predictions, we show their full trajectories in $t\in[0,1]$ and enlarged views in $t\in[0, 0.01]$.}
    \label{fig:sample_wo_penalty}
\end{figure}

%% file: arxiv.bbl
\begin{thebibliography}{10}

\bibitem{abdulle2012heterogeneous}
{\sc A.~Abdulle, E.~Weinan, B.~Engquist, and E.~Vanden-Eijnden}, {\em The
  heterogeneous multiscale method}, Acta Numerica, 21 (2012), pp.~1--87.

\bibitem{akerson2024learning}
{\sc A.~Akerson, A.~Rajan, and K.~Bhattacharya}, {\em Learning constitutive
  relations from experiments: 1. pde constrained optimization}, arXiv preprint
  arXiv:2412.02864,  (2024).

\bibitem{allaire1992homogenization}
{\sc G.~Allaire}, {\em Homogenization and two-scale convergence}, SIAM Journal
  on Mathematical Analysis, 23 (1992), pp.~1482--1518.

\bibitem{bensoussan2011asymptotic}
{\sc A.~Bensoussan, J.-L. Lions, and G.~Papanicolaou}, {\em Asymptotic analysis
  for periodic structures}, vol.~374, American Mathematical Soc., 2011.

\bibitem{bhattacharya2024learning}
{\sc K.~Bhattacharya, N.~B. Kovachki, A.~Rajan, A.~M. Stuart, and M.~Trautner},
  {\em Learning homogenization for elliptic operators}, SIAM Journal on
  Numerical Analysis, 62 (2024), pp.~1844--1873.

\bibitem{bhattacharya2023learning}
{\sc K.~Bhattacharya, B.~Liu, A.~Stuart, and M.~Trautner}, {\em Learning
  markovian homogenized models in viscoelasticity}, Multiscale Modeling \&
  Simulation, 21 (2023), pp.~641--679.

\bibitem{billington1982physics}
{\sc E.~Billington, A.~Tate, and W.~Williams}, {\em The physics of deformation
  and flow}, 1982.

\bibitem{bishara2023state}
{\sc D.~Bishara, Y.~Xie, W.~K. Liu, and S.~Li}, {\em A state-of-the-art review
  on machine learning-based multiscale modeling, simulation, homogenization and
  design of materials}, Archives of computational methods in engineering, 30
  (2023), pp.~191--222.

\bibitem{blanc2023homogenization}
{\sc X.~Blanc and C.~Le~Bris}, {\em Homogenization Theory for Multiscale
  Problems: An Introduction}, vol.~21, Springer Nature, 2023.

\bibitem{brenner2013overall}
{\sc R.~Brenner and P.~Suquet}, {\em Overall response of viscoelastic
  composites and polycrystals: exact asymptotic relations and approximate
  estimates}, International Journal of Solids and Structures, 50 (2013),
  pp.~1824--1838.

\bibitem{chen2018neural}
{\sc R.~T. Chen, Y.~Rubanova, J.~Bettencourt, and D.~K. Duvenaud}, {\em Neural
  ordinary differential equations}, Advances in neural information processing
  systems, 31 (2018).

\bibitem{cioranescu1999introduction}
{\sc D.~Cioranescu and P.~Donato}, {\em An introduction to homogenization},
  Oxford university press, 1999.

\bibitem{coleman1961foundations}
{\sc B.~D. Coleman and W.~Noll}, {\em Foundations of linear viscoelasticity},
  Reviews of modern physics, 33 (1961), p.~239.

\bibitem{coleman1961recent}
{\sc B.~D. Coleman and W.~Noll}, {\em Recent results in the continuum theory of
  viscoelastic fluids}, Annals of the New York Academy of Sciences, 89 (1961),
  pp.~672--714.

\bibitem{darrow2025spectral}
{\sc D.~Darrow and G.~Stepaniants}, {\em A spectral theory of scalar volterra
  equations}, arXiv preprint arXiv:2503.06957,  (2025).

\bibitem{dupont2019augmented}
{\sc E.~Dupont, A.~Doucet, and Y.~W. Teh}, {\em Augmented neural odes},
  Advances in neural information processing systems, 32 (2019).

\bibitem{eggersmann2019model}
{\sc R.~Eggersmann, T.~Kirchdoerfer, S.~Reese, L.~Stainier, and M.~Ortiz}, {\em
  Model-free data-driven inelasticity}, Computer Methods in Applied Mechanics
  and Engineering, 350 (2019), pp.~81--99.

\bibitem{ferry1980viscoelastic}
{\sc J.~Ferry}, {\em Viscoelastic Properties of Polymers}, vol.~264, Wiley,
  1980.

\bibitem{francfort1986homogenization}
{\sc G.~A. Francfort and P.~M. Suquet}, {\em Homogenization and mechanical
  dissipation in thermoviscoelasticity}, Archive for Rational Mechanics and
  Analysis, 96 (1986), pp.~265--293.

\bibitem{fuhg2024review}
{\sc J.~N. Fuhg, G.~Anantha~Padmanabha, N.~Bouklas, B.~Bahmani, W.~Sun, N.~N.
  Vlassis, M.~Flaschel, P.~Carrara, and L.~De~Lorenzis}, {\em A review on
  data-driven constitutive laws for solids}, Archives of Computational Methods
  in Engineering,  (2024), pp.~1--43.

\bibitem{fuhg2022physics}
{\sc J.~N. Fuhg and N.~Bouklas}, {\em On physics-informed data-driven isotropic
  and anisotropic constitutive models through probabilistic machine learning
  and space-filling sampling}, Computer Methods in Applied Mechanics and
  Engineering, 394 (2022), p.~114915.

\bibitem{ghavamian2019accelerating}
{\sc F.~Ghavamian and A.~Simone}, {\em Accelerating multiscale finite element
  simulations of history-dependent materials using a recurrent neural network},
  Computer Methods in Applied Mechanics and Engineering, 357 (2019), p.~112594.

\bibitem{gross1968mathematical}
{\sc B.~Gross}, {\em Mathematical structure of the theories of
  viscoelasticity}, (No Title),  (1968).

\bibitem{guedes1990preprocessing}
{\sc J.~Guedes and N.~Kikuchi}, {\em Preprocessing and postprocessing for
  materials based on the homogenization method with adaptive finite element
  methods}, Computer methods in applied mechanics and engineering, 83 (1990),
  pp.~143--198.

\bibitem{haghighat2023constitutive}
{\sc E.~Haghighat, S.~Abouali, and R.~Vaziri}, {\em Constitutive model
  characterization and discovery using physics-informed deep learning},
  Engineering Applications of Artificial Intelligence, 120 (2023), p.~105828.

\bibitem{horstemeyer2010historical}
{\sc M.~F. Horstemeyer and D.~J. Bammann}, {\em Historical review of internal
  state variable theory for inelasticity}, International Journal of Plasticity,
  26 (2010), pp.~1310--1334.

\bibitem{huang2024operator}
{\sc D.~Z. Huang, N.~H. Nelsen, and M.~Trautner}, {\em An operator learning
  perspective on parameter-to-observable maps}, Foundations of Data Science,
  (2024).

\bibitem{jones2022neural}
{\sc R.~E. Jones, A.~L. Frankel, and K.~Johnson}, {\em A neural ordinary
  differential equation framework for modeling inelastic stress response via
  internal state variables}, Journal of Machine Learning for Modeling and
  Computing, 3 (2022).

\bibitem{karimi2024learning}
{\sc M.~Karimi and K.~Bhattacharya}, {\em A learning-based multiscale model for
  reactive flow in porous media}, Water Resources Research, 60 (2024),
  p.~e2023WR036303.

\bibitem{kim2024experimental}
{\sc J.~H. Kim, D.~Yang, and S.~Park}, {\em Experimental validation for the
  interconversion between generalized kelvin--voigt and maxwell models using
  human skin tissues}, Journal of Biomechanics, 162 (2024), p.~111908.

\bibitem{kovachki2021universal}
{\sc N.~Kovachki, S.~Lanthaler, and S.~Mishra}, {\em On universal approximation
  and error bounds for fourier neural operators}, Journal of Machine Learning
  Research, 22 (2021), pp.~1--76.

\bibitem{kovachki2023neural}
{\sc N.~Kovachki, Z.~Li, B.~Liu, K.~Azizzadenesheli, K.~Bhattacharya,
  A.~Stuart, and A.~Anandkumar}, {\em Neural operator: Learning maps between
  function spaces with applications to pdes}, Journal of Machine Learning
  Research, 24 (2023), pp.~1--97.

\bibitem{kozlov1980averaging}
{\sc S.~M. Kozlov}, {\em Averaging of random operators}, Sbornik: Mathematics,
  37 (1980), pp.~167--180.

\bibitem{kraus2017parameter}
{\sc M.~A. Kraus, M.~Schuster, J.~Kuntsche, G.~Siebert, and J.~Schneider}, {\em
  Parameter identification methods for visco-and hyperelastic material models},
  Glass Structures \& Engineering, 2 (2017), pp.~147--167.

\bibitem{lahellec2024effective}
{\sc N.~Lahellec, R.~Masson, and P.~Suquet}, {\em Effective thermodynamic
  potentials and internal variables: linear viscoelastic composites}, Journal
  of the Mechanics and Physics of Solids, 188 (2024), p.~105649.

\bibitem{lahellec2024effective2}
{\sc N.~Lahellec, R.~Masson, and P.~Suquet}, {\em Effective thermodynamic
  potentials and internal variables: Particulate thermoviscoelastic
  composites}, Journal of the Mechanics and Physics of Solids, 193 (2024),
  p.~105891.

\bibitem{li2020fourier}
{\sc Z.~Li, N.~Kovachki, K.~Azizzadenesheli, B.~Liu, K.~Bhattacharya,
  A.~Stuart, and A.~Anandkumar}, {\em Fourier neural operator for parametric
  partial differential equations}, arXiv preprint arXiv:2010.08895,  (2020).

\bibitem{liu2022learning}
{\sc B.~Liu, N.~Kovachki, Z.~Li, K.~Azizzadenesheli, A.~Anandkumar, A.~M.
  Stuart, and K.~Bhattacharya}, {\em A learning-based multiscale method and its
  application to inelastic impact problems}, Journal of the Mechanics and
  Physics of Solids, 158 (2022), p.~104668.

\bibitem{liu2023learning}
{\sc B.~Liu, E.~Ocegueda, M.~Trautner, A.~M. Stuart, and K.~Bhattacharya}, {\em
  Learning macroscopic internal variables and history dependence from
  microscopic models}, Journal of the Mechanics and Physics of Solids,  (2023),
  p.~105329.

\bibitem{liu2021review}
{\sc X.~Liu, S.~Tian, F.~Tao, and W.~Yu}, {\em A review of artificial neural
  networks in the constitutive modeling of composite materials}, Composites
  Part B: Engineering, 224 (2021), p.~109152.

\bibitem{liu2019deep}
{\sc Z.~Liu, C.~Wu, and M.~Koishi}, {\em A deep material network for multiscale
  topology learning and accelerated nonlinear modeling of heterogeneous
  materials}, Computer Methods in Applied Mechanics and Engineering, 345
  (2019), pp.~1138--1168.

\bibitem{milton_book}
{\sc G.~W. Milton}, {\em The Theory of Composites}, Cambridge University Press,
  Cambridge, 2002.

\bibitem{mishra2016comparative}
{\sc N.~Mishra, J.~Vond{\v{r}}ejc, and J.~Zeman}, {\em A comparative study on
  low-memory iterative solvers for fft-based homogenization of periodic media},
  Journal of Computational Physics, 321 (2016), pp.~151--168.

\bibitem{moulinec1998numerical}
{\sc H.~Moulinec and P.~Suquet}, {\em A numerical method for computing the
  overall response of nonlinear composites with complex microstructure},
  Computer methods in Applied Mechanics and Engineering, 157 (1998),
  pp.~69--94.

\bibitem{mozaffar2019deep}
{\sc M.~Mozaffar, R.~Bostanabad, W.~Chen, K.~Ehmann, J.~Cao, and M.~Bessa},
  {\em Deep learning predicts path-dependent plasticity}, Proceedings of the
  National Academy of Sciences, 116 (2019), pp.~26414--26420.

\bibitem{nikonov2005determination}
{\sc A.~Nikonov, A.~R. Davies, and I.~Emri}, {\em The determination of creep
  and relaxation functions from a single experiment}, Journal of Rheology, 49
  (2005), pp.~1193--1211.

\bibitem{ostoja2018does}
{\sc M.~Ostoja-Starzewski and J.~Zhang}, {\em Does a fractal microstructure
  require a fractional viscoelastic model?}, Fractal and Fractional, 2 (2018),
  p.~12, \url{https://doi.org/10.3390/fractalfract2010012}.

\bibitem{pavliotis2008multiscale}
{\sc G.~Pavliotis and A.~Stuart}, {\em Multiscale methods: averaging and
  homogenization}, Springer Science \& Business Media, 2008.

\bibitem{phillips2001crystals}
{\sc R.~Phillips and P.~Rob}, {\em Crystals, defects and microstructures:
  modeling across scales}, Cambridge University Press, 2001.

\bibitem{rice1971inelastic}
{\sc J.~R. Rice}, {\em Inelastic constitutive relations for solids: an
  internal-variable theory and its application to metal plasticity}, Journal of
  the Mechanics and Physics of Solids, 19 (1971), pp.~433--455.

\bibitem{sanchez1980non}
{\sc E.~S{\'a}nchez-Palencia}, {\em Non-homogeneous media and vibration
  theory}, Lecture Note in Physics, Springer-Verlag, 320 (1980), pp.~57--65.

\bibitem{serra2019viscoelastic}
{\sc A.~Serra-Aguila, J.~Puigoriol-Forcada, G.~Reyes, and J.~Menacho}, {\em
  Viscoelastic models revisited: characteristics and interconversion formulas
  for generalized kelvin--voigt and maxwell models}, Acta Mechanica Sinica, 35
  (2019), pp.~1191--1209.

\bibitem{shanbhag2023computer}
{\sc S.~Shanbhag}, {\em A computer program for interconversion between creep
  compliance and stress relaxation}, Journal of Rheology, 67 (2023),
  pp.~965--975.

\bibitem{suquet1987elements}
{\sc P.~Suquet}, {\em Elements of homogenization for inelastic solid
  mechanics}, Homogenization techniques for composite media, 272 (1987),
  pp.~193--278.

\bibitem{tartar1991memory}
{\sc L.~Tartar}, {\em Memory effects and homogenization}, in Mechanics and
  Thermodynamics of Continua: A Collection of Papers Dedicated to BD Coleman on
  His Sixtieth Birthday, Springer, 1991, pp.~537--549.

\bibitem{tschoegl2012phenomenological}
{\sc N.~W. Tschoegl}, {\em The phenomenological theory of linear viscoelastic
  behavior: an introduction}, Springer Science \& Business Media, 2012.

\bibitem{weinan2011principles}
{\sc E.~Weinan}, {\em Principles of multiscale modeling}, Cambridge University
  Press, 2011.

\bibitem{zhang2024iterated}
{\sc Y.~Zhang and K.~Bhattacharya}, {\em Iterated learning and multiscale
  modeling of history-dependent architectured metamaterials}, arXiv preprint
  arXiv:2402.12674,  (2024).

\bibitem{zohdi2008introduction}
{\sc T.~I. Zohdi and P.~Wriggers}, {\em An introduction to computational
  micromechanics}, Springer Science \& Business Media, 2008.

\end{thebibliography}


\begin{thebibliography}{1}

\bibitem{bhattacharya2024learning}
{\sc K.~Bhattacharya, N.~B. Kovachki, A.~Rajan, A.~M. Stuart, and M.~Trautner},
  {\em Learning homogenization for elliptic operators}, SIAM Journal on
  Numerical Analysis, 62 (2024), pp.~1844--1873.

\bibitem{bhattacharya2023learning}
{\sc K.~Bhattacharya, B.~Liu, A.~Stuart, and M.~Trautner}, {\em Learning
  markovian homogenized models in viscoelasticity}, Multiscale Modeling \&
  Simulation, 21 (2023), pp.~641--679.

\bibitem{huang2024operator}
{\sc D.~Z. Huang, N.~H. Nelsen, and M.~Trautner}, {\em An operator learning
  perspective on parameter-to-observable maps}, Foundations of Data Science,
  (2024).

\bibitem{kovachki2021universal}
{\sc N.~Kovachki, S.~Lanthaler, and S.~Mishra}, {\em On universal approximation
  and error bounds for fourier neural operators}, Journal of Machine Learning
  Research, 22 (2021), pp.~1--76.

\bibitem{pavliotis2008multiscale}
{\sc G.~Pavliotis and A.~Stuart}, {\em Multiscale methods: averaging and
  homogenization}, Springer Science \& Business Media, 2008.

\end{thebibliography}
